\documentclass[11pt,a4paper]{article}
\usepackage[utf8]{inputenc} % proper encryption
\usepackage[USenglish]{babel} % language 
\usepackage{graphicx} % images
\usepackage{comment} % doing long comments
\usepackage{amsmath}
\usepackage{amsfonts}
\usepackage{amssymb}\usepackage{amsmath} % math
\usepackage{amssymb}
\usepackage[singlespacing]{setspace}
\usepackage{mathtools} % special math symbols
\usepackage{stmaryrd}

\usepackage[maxbibnames=99, maxcitenames=99,giveninits=true, sortcites=true,url=false,citestyle=numeric]{biblatex}
\usepackage{amsthm}
\usepackage{abstract} % abstract
\usepackage{float}
\usepackage[mathscr]{euscript}
\usepackage[singlelinecheck=false,justification=RaggedRight,format=plain]{caption}
\usepackage{subcaption}
\usepackage{geometry}
\usepackage{mathabx}
\usepackage{tikz-cd}
\usepackage{xcolor}
\usepackage{array}

\usepackage[hidelinks]{hyperref}

\newcommand{\N}{\mathbb{N}}

\newcommand{\R}{\mathbb{R}}
\newcommand{\C}{\mathbb{C}}
\newcommand{\X}{\mathfrak{X}}
\DeclareMathOperator{\id}{id}

\newcommand{\dd}{\mathrm{d}}
\newcommand{\ii}{\mathrm{i}}

\DeclareMathOperator{\im}{im}

\newcommand{\dn}{\mathrm{d}_{\nabla}}

\DeclareMathOperator{\Aut}{Aut}
\DeclareMathOperator{\Auteq}{Auteq}
\DeclareMathOperator{\Autex}{Auteq^{ex}}
\DeclareMathOperator{\auteq}{\mathfrak{auteq}}
\DeclareMathOperator{\autex}{{\mathfrak{h}}_{ex}}
\DeclareMathOperator{\Der}{Der}
\DeclareMathOperator{\End}{End}
\DeclareMathOperator{\ad}{ad}

\newcommand{\dbar}{\bar{\partial}}

\DeclareMathOperator{\Diff}{Diff}

\DeclareMathOperator{\pr}{pr}

\DeclareMathOperator{\Hproj}{pr_{\mathcal{H}}}
\DeclareMathOperator{\G}{G}
\DeclareMathOperator{\Q}{Q}
\DeclareMathOperator{\QdR}{Q_{dR}}
\DeclareMathOperator{\HprojL}{pr_{\mathcal{H}_L}}

\renewcommand{\O}{\operatorname{O}}

\makeatletter
\g@addto@macro\bfseries{\boldmath}
\makeatother

\newtheorem{theorem}{Theorem}[section]
\newtheorem*{theorem*}{Theorem}
\newtheorem{corollary}[theorem]{Corollary}
\newtheorem{lemma}[theorem]{Lemma}
\newtheorem{proposition}[theorem]{Proposition}

\theoremstyle{definition}
\newtheorem{definition}[theorem]{Definition}

\newtheorem{remark}[theorem]{Remark}

\addto\captionsUSenglish{}

\title{Deformations of generalized complex structures on $\mathcal{G}$-flat transitive Courant algebroids}
\author{ Vicente Cortés and Paula Naomi Pilatus}
\date{}

\begin{document}
\newgeometry{inner=2.50cm, outer=2.50cm, top=2.50cm, bottom=3.50cm}

\maketitle
\begin{abstract}
We extend Gualtieri's deformation theorem for generalized complex structures on exact Courant algebroids to the more general class of $\mathcal{G}$-flat transitive Courant algebroids.  Our approach relies on the framework of tame Fréchet spaces and Hamilton's Nash--Moser (implicit function) theorem. As an important step of the proof of our deformation theorem, we endow the group of autoequivalences of a $\mathcal{G}$-flat transitive Courant algebroid with a tame Fréchet Lie group structure and show that there exists a tame subgroup of exact autoequivalences, whose infinitesimal action is by inner derivations.

for differential geometric structures)

\end{abstract}
\tableofcontents 
\section{Introduction}

\emph{Generalized complex structures} are geometric structures that encompass both complex and symplectic structures as special cases. They were first introduced by Hitchin in~\cite{Hitchin2003} and developed further by Gualtieri in~\cite{Gualtieri:2003dx,gualtieri2011}. Generalized complex structures have applications to physics: they appear naturally in the context of supersymmetric non-linear $\sigma$-models~\cite{Gates1984,Gualtieri2014} and play a role in the context of T-duality in string theory~\cite{cavalcanti2010generalized,Severa2015,Cavalcanti2026,Corts2023}. 

While classical complex structures can be considered as endomorphisms on the tangent bundle of a manifold $M$, generalized complex structures are endomorphisms on the generalized tangent bundle $\mathbb{T}M=TM\oplus T^*M$ or, more generally, on \emph{Courant algebroids}. The latter are vector bundles equipped with the data of a scalar product, a bracket on the sections and a vector bundle morphism to the tangent bundle, called \emph{anchor map}, satisfying compatibility conditions.

In this work, we are concerned with deformations of generalized complex structures. These are relevant to moduli problems in supergravity theories (see e.g.~\cite{GRANA2004979,KAPUSTIN2004,Butti2008}). The deformation theory of generalized complex structures on \emph{exact} Courant algebroids, i.e.\ those isomorphic to a (twisted) generalized tangent bundle, has been developed by Gualtieri in the seminal work~\cite{Gualtieri:2003dx,gualtieri2011}. Whereas deformations of (classical) complex structures are studied up to the action of the diffeomorphism group, deformations of \textit{generalized} complex structures must be studied up to so-called \emph{Courant autoequivalences}: vector bundle equivalences that preserve the Courant algebroid structure. Within this setting, Gualtieri established the existence of a locally complete family of generalized complex deformations. His approach was inspired by Kuranishi's simplified proof in~\cite{kuranishi64} of the existence of a locally complete family of (classical) complex deformations.  

A long-term objective (which goes beyond the scope of this paper) is to eventually extend Gualtieri's theorem to the larger class of \emph{transitive} Courant algebroids, i.e.\ Courant algebroids with a surjective anchor map. These are isomorphic to Courant algebroids with underlying vector bundle $TM\oplus\mathcal{G}\oplus T^*M$, where $\mathcal{G}$ is a bundle of quadratic Lie algebras. Transitive Courant algebroids have applications to supergravity theories, where in the presence of additional gauge fields, it is often not enough to only consider sections of the generalized tangent bundle~\cite{GarciaFernandez2014}. 

As was already pointed out in~\cite{paperoncomplexdefs}, in order to achieve such an extension, the proofs of Kuranishi~\cite{kuranishi64} and Gualtieri~\cite{Gualtieri:2003dx,gualtieri2011} need to be revisited. These proofs rely on Sobolev spaces and refer to both the implicit function theorem and the inverse function theorem on Banach spaces. However, a justification  that this implicit function theorem is indeed applicable is not provided. In addition, while Kuranishi works  with local charts of the diffeomorphism group, Gualtieri works with flows, which raises additional analytic questions.

In~\cite{paperoncomplexdefs}, we recovered Kuranishi's theorem~\cite{kuranishi64} asserting the existence of a locally complete family of complex deformations using Hamilton--Nash--Moser theory. In the present work we will use the same framework to extend Gualtieri's   deformation theorem~\cite{Gualtieri:2003dx,gualtieri2011} to the class of $\mathcal{G}$-flat transitive Courant algebroids. This is a class of transitive Courant algebroids which properly contains the exact Courant algebroids. More precisely, we prove the following theorem.
\begin{theorem*}[cf.\ Theorem~\ref{thm: deformation theorem}]
Let $\mathcal{J}$ be a generalized complex structure on a $\mathcal{G}$-flat transitive Courant algebroid $E\rightarrow M$ over a compact manifold. Let $L$ be the Lie algebroid corresponding to $\mathcal{J}$. There exists an open neighborhood $\mathcal{W}\subset H^2(M,L)$ of $0$, a family $\tilde{\mathcal{M}}=\{\varphi_t:t\in \mathcal{W}\}$ of generalized almost complex
deformations of $L$, and an analytic obstruction map 
\begin{equation*}
\phi\colon\mathcal{W}\rightarrow H^3(M,L)  \end{equation*}
such that the generalized almost complex deformations $\mathcal{M}:=\{\varphi_t\in\tilde{\mathcal{M}}:\phi(t)=0\}$ are precisely the integrable ones in $\tilde{\mathcal{M}}$. Any sufficiently small generalized complex deformation of $\mathcal J$ is equivalent to at least one member
of the family $\mathcal{M}$. In the case that the obstruction map vanishes, $\mathcal{M}$ is a smooth locally complete family of generalized complex deformations.    
\end{theorem*}
The proof deviates in several interesting ways from the proof of the corresponding theorem for classical complex structures  in~\cite{paperoncomplexdefs} (and~\cite{kuranishi64}).

First of all, the proof requires that the group of autoequivalences of a $\mathcal{G}$-flat transitive Courant algebroid is endowed with an infinite-dimensional Lie group structure, more specifically, a tame (Fréchet) Lie group structure. While it has been long established that the diffeomorphism group is a tame Lie group \cite{hamilton82}, as part of the proof of our theorem we need to establish such a Lie group structure for the group of autoequivalences of a $\mathcal{G}$-flat transitive Courant algebroid. 

Moreover, like Gualtieri did in~\cite{Gualtieri:2003dx,gualtieri2011}, to prove local completeness, we will restrict to the action of the subgroup of \emph{exact} autoequivalences, i.e.\ autoequivalences whose infinitesimal action is by the Dorfman bracket of the Courant algebroid. This turns out to be sufficient and has the advantage that the setup resembles the classical complex case, where the infinitesimal action of the diffeomorphism group is by the Lie bracket of vector fields. Therefore, as part of the proof we show that the Lie algebra of inner derivations of the Courant algebroid indeed integrates to a (tame) Lie subgroup. The statement that the (exact) autoequivalences of a $\mathcal{G}$-flat transitive Courant algebroid constitute a tame Lie group is of independent interest and relevant to other moduli problems in generalized geometry.

We would like to point out that it was already shown by Rubio and Tipler in~\cite{rubio20} that the group of autoequivalences of \textit{exact} Courant algebroids has a strong ILH Lie group structure, which is another notion of infinite-dimensional Lie group structure.  Moreover, it is shown in~\cite{rubio20} that the exact autoequivalences of an exact Courant algebroid indeed form a (strong ILH) Lie subgroup and some of our results build on this fact. 

Another subtlety which appears is the following. On a complex manifold, the Lie algebra of the diffeomorphism group may be identified with the space $\Gamma(M,T^{1,0}M)=\mathcal{A}^0(M,T^{1,0}M)$. In a similar way, the Lie algebra of exact autoequivalences of a ($\mathcal{G}$-flat) Courant algebroid with a generalized complex structure may be identified with \emph{a subspace of} $\Gamma(M,L^*)$. To prove local completeness of the deformation family, like in the complex case, we use Hodge theory to show the existence of a family of right inverses to a certain family of linear maps. For this, we need to ensure that the $\dd_L^*$-exact elements of $\Gamma(M,L^*)$ are contained in that subspace. This problem does not arise in the complex case where the entire space $\mathcal{A}^0(M,T^{1,0}M)$ trivially contains the $\dbar^*$-exact elements. 

At the end of the paper we will also explicitly describe deformations of a class of  examples of generalized complex structures on $\mathcal{G}$-flat Courant algebroids. The structures considered are essentially direct sums of ordinary complex structures on the base manifold and complex structures on the fiber of the quadratic Lie algebra bundle. We describe the space of infinitesimal deformations and show that compared to the exact case three new types of infinitesimal deformations occur.

Finally, we want to emphasize that the core of our proof of the main theorem only relies on the fact that the Lie algebroid complex associated with the generalized complex structure is exact (which we prove to hold true for every generalized complex structure on a transitive Courant algebroid) as well as on the fact that the exact autoequivalences (as well as all the autoequivalences) of a $\mathcal{G}$-flat  transitive Courant algebroid form a tame Lie group acting smooth tamely on the sections of the Courant algebroid. It does \emph{not} depend on the specific formulas obtained for the $\mathcal{G}$-flat transitive Courant algebroids. We conjecture that the above-stated facts are true for the group of (exact) autoequivalences on general transitive Courant algebroids. Therefore, we believe that the present paper can serve as the foundation to the extension of Gualtieri's theorem to a theorem on the existence of a locally complete family of generalized complex deformations on transitive Courant algebroids. 

\textbf{This paper is structured as follows.} We begin by introducing some background on (transitive) Courant algebroids and generalized complex structures in Section~\ref{sect: prelim on gc}. To every generalized complex structure one can naturally associate a Lie (bi-)algebroid. We show that for generalized complex structures on transitive Courant algebroids the corresponding Lie algebroid complex is elliptic. Moreover, we discuss some identities for the Schouten bracket and the Lie algebroid differential, as well as some facts from Hodge theory, that are required. 

In Section~\ref{sect: all autoeq}, we endow the group of autoequivalences of a $\mathcal{G}$-flat transitive Courant algebroid with a tame Lie group structure and show that there is a tame Lie subgroup (the group of \emph{exact} autoequivalences) whose Lie algebra is that of inner derivations. Moreover, in Section~\ref{sect: auteq acting on gc defs} we show that the action of the groups of autoequivalences and exact autoequivalences on generalized (almost) complex structures on a $\mathcal{G}$-flat transitive Courant algebroid are smooth tame. We also compute a partial derivative of the latter action, which is essential for the proof of local completeness of the deformation family. 

Then  we prove our main result, Theorem~\ref{thm: deformation theorem} stated above, in Section~\ref{defothm:sec}, using both the Nash--Moser inverse function theorem and Hamilton's Nash--Moser implicit function theorem. 

Finally, in Section~\ref{example of deformation}, we will explicitly describe deformations of the aforementioned class of generalized complex structures of ``product type.''  

The appendix includes (in Section~\ref{appendix proof of lemma}) proofs of the identities for generalized complex structures stated in Section~\ref{sect: prelim on gc}.  Furthermore (in Section~\ref{appendix HNM}) we review some notions and results from Hamilton--Nash--Moser theory that are used in this article and which have not been included in our  previous paper~\cite{paperoncomplexdefs}.

\subsection*{Acknowledgments}
We thank Jan Heck for early collaboration. Moreover, we are grateful to Jan Heck and Aldo Witte for valuable comments on the paper. The authors have received funding from the Deutsche Forschungsgemeinschaft 
(DFG, German Research Foundation) under Germany's Excellence Strategy, EXC 2121 ``Quantum Universe'', 390833306 and under -- SFB-Gesch\"aftszeichen 1624 -- Projektnummer 506632645.

\section{Generalized complex structures on Courant algebroids}\label{sect: prelim on gc}
In this section we give a first overview of the geometric objects studied in our paper, namely generalized complex structures on Courant algebroids. Any such structure gives rise to a pair of Lie algebroids and corresponding differential graded algebras. 
Our focus is on deformations of generalized complex structures (and smooth families of deformations). These are described by solutions of a Maurer--Cartan equation, which makes use of the pair of differential graded algebras associated with the generalized complex structure being deformed.
We show that the Lie algebroid complexes are elliptic if the Courant algebroid is transitive. Thus the Lie
algebroid cohomology, and in particular the space of infinitesimal deformations, is finite-dimensional if the base manifold is compact. Finally, we collect some basic Hodge-theoretic relations to be used later for the study of families of deformations and their local completeness.   

\subsection{Courant algebroids}
\begin{definition}
\label{def: courant algebroid}
    A \emph{Courant algebroid} on a manifold $M$ is a vector bundle $E \rightarrow M$ equipped with a non-degenerate symmetric bilinear form $\langle\cdot, \cdot\rangle \in$ $\Gamma\left(\operatorname{Sym}^2\left(E^*\right)\right)$ (called the \emph{scalar product}), a bilinear operation $[\cdot, \cdot]$ (called the \emph{Dorfman bracket}) on the space of smooth sections $\Gamma(E)$ of $E$ and a homomorphism of vector bundles $\pi: E \rightarrow T M$ (called the \emph{anchor}), such that the following conditions are satisfied: for all sections $u, v, w \in \Gamma(E)$,
    \begin{enumerate}
        \item $[u,[v, w]]=[[u, v], w]+[v,[u, w]]$,
        \item $\pi(u)\langle v, w\rangle=\langle[u, v], w\rangle+\langle v,[u, w]\rangle$,
        \item $\langle [u,v]+[v,u],w\rangle=\pi(w)\langle u,v\rangle$.
    \end{enumerate}
\end{definition}
The following consequences are well-known.
\begin{proposition}\label{prop: extra properties of dorfman bracket}
    The bracket and anchor of a Courant algebroid $E\rightarrow M$ satisfy the following:
    \begin{enumerate}
        \item $[u,fv]=\pi(u)(f)v+f[u,v]$,
        \item $\pi[u,v]=\mathcal{L}_{\pi u}(\pi v)$,
    \end{enumerate}
    for all $u,v\in\Gamma(E)$, $f\in C^{\infty}(M)$.
\end{proposition}

\begin{definition}
    A Courant algebroid $E$ is called 
    \begin{enumerate}
        \item \emph{transitive} if the anchor map $\pi\colon E\rightarrow TM$ is surjective;
        \item \emph{exact} if the following sequence is exact $$0\rightarrow TM^*\xrightarrow{\pi^*}E \xrightarrow{\pi}TM\rightarrow 0.$$
    \end{enumerate}
\end{definition}

\begin{definition}
\label{def: isom}
    A \emph{(Courant) equivalence} between Courant algebroids $(E_1,\pi_1$,  $[\cdot,\cdot]_1,\langle\cdot,\cdot\rangle_1)$ and $(E_2,\pi_2,[\cdot,\cdot]_2,\langle\cdot,\cdot\rangle_2)$ over the same base manifold $M$ is an equivalence of vector bundles $F\colon E_1\to E_2$ covering a diffeomorphism $f$ such that for all $u,v\in \Gamma(M,E_1)$
    \begin{align*}
        &\dd f\circ\pi_1=\pi_2\circ F\\
        &\langle u,v\rangle_1\circ f=\langle F(u),F(v)\rangle_2\\
        &F([u,v]_1)=[F(u),F(v)]_2.
    \end{align*}
A \emph{(Courant) autoequivalence} is an equivalence from a Courant algebroid to itself. A \emph{(Courant) automorphism} is an autoequivalence covering the identity.
\end{definition}

For Courant autoequivalences, the first property in Definition~\ref{def: isom} follows from the second and third.

\begin{definition}
    Let $E$ be a Courant algebroid. A \emph{derivation} of $E$ covering $X\in\X(M)$ is an endomorphism $\mathcal{D}\in \End(\Gamma(M,E))$ such that
    \begin{align*}
    \mathcal{D}(fu)&=f\mathcal{D} u+(\mathcal{L}_X f)u,\\
      X \langle\cdot,\cdot\rangle&=\langle \mathcal{D}\cdot,\cdot\rangle+\langle\cdot,\mathcal{D}\cdot\rangle\\
        \mathcal{D}[\cdot,\cdot]&=[\mathcal{D}\cdot,\cdot]+[\cdot,\mathcal{D}\cdot].
    \end{align*}
\end{definition}
Particular examples of derivations are the maps $\ad(u):=[u,\cdot]$ covering $\pi(u)$ for some $u\in\Gamma(M,E)$. 
The derivations of a Courant algebroid form a Lie subalgebra of $\mathrm{End}\, \Gamma (E)$, which we denote by $\Der(E)$. If $\mathcal{D}_1$ is a derivation covering $X_1$ and $\mathcal{D}_2$ is a derivation covering $X_2$, then $[\mathcal{D}_1,\mathcal{D}_2]$ covers $\mathcal{L}_{X_1}X_2$. The subalgebra $\mathrm{ad}\, \Gamma (E)\subset \Der(E)$ is called the Lie algebra of \emph{inner derivations}.

\subsubsection{Bisections of transitive Courant algebroids}
It was shown in~\cite{chen2013} that every transitive Courant algebroid is isomorphic to a \emph{standard transitive Courant algebroid}. Before we state this, we will introduce some notation.

Let $(\mathcal{G}\to M,\langle\cdot,\cdot\rangle_{\mathcal{G}},[\cdot,\cdot]_{\mathcal{G}})$ be a bundle of quadratic Lie algebras, i.e.\ a vector bundle $\mathcal{G}\to M$ with a scalar product $\langle\cdot,\cdot\rangle_{\mathcal{G}}$ and a Lie bracket $[\cdot,\cdot]_{\mathcal{G}}$ on each fiber such that for every $p\in M$ and $X,Y,Z\in \mathcal{G}_p$ we have
$$\langle [X,Y]_{\mathcal{G}},Z\rangle_{\mathcal{G}}+\langle Y, [X,Z]_{\mathcal{G}}\rangle_{\mathcal{G}}=0.$$

We denote by $\O(\mathcal{G})$ the group of orthogonal equivalences of $\mathcal{G}$, i.e.\ vector bundle morphisms covering a diffeomorphism, which preserve the scalar product of $\mathcal{G}$. Moreover, we denote by $\Auteq(\mathcal{G})$ the group of \emph{autoequivalences} of $\mathcal{G}$, i.e.\ vector bundle morphisms covering a diffeomorphism, which preserve the Lie bracket and scalar product of $\mathcal{G}$. We denote by $\Aut(\mathcal{G})$ the group of \emph{automorphisms}, i.e.\ autoequivalences covering the identity.

Moreover, we denote by $\Omega^p(M,\mathcal{G})$ the $p$-forms with values in $\mathcal{G}$. We define the map
$$
    \langle\cdot\wedge\cdot\rangle_{\mathcal{G}}\colon \Omega^p(M,\mathcal{G})\times \Omega^q(M,\mathcal{G})\longrightarrow \Omega^{p+q}(M)
$$
via
$$\langle \mathbf{A}\wedge \mathbf{B}\rangle_{\mathcal{G}}(X_1,\dots,X_{p+q})=\frac{1}{p!q!}\sum_{\sigma\in S_{p+q}}(-1)^{|\sigma|}\langle \mathbf{A}(X_{\sigma(1)},\dots,X_{\sigma(p)}),\mathbf{B}(X_{\sigma(p+1)},\dots,X_{\sigma(p+q)})\rangle_{\mathcal{G}}
$$
and the map
$$
[\cdot,\cdot]_{\mathcal{G}}\colon \Omega^p(M,\mathcal{G})\times \Omega^q(M,\mathcal{G})\longrightarrow\Omega^{p+q}(M,\mathcal{G})
$$
via
$$[\mathbf{A}, \mathbf{B}]_{\mathcal{G}}(X_1,\dots,X_{p+q})=\frac{1}{p!q!}\sum_{\sigma\in S_{p+q}}(-1)^{|\sigma|}[\mathbf{A}(X_{\sigma(1)},\dots,X_{\sigma(p)}),\mathbf{B}(X_{\sigma(p+1)},\dots,X_{\sigma(p+q)})]_{\mathcal{G}}.
$$
For $\mathbf{A}\in \Omega^p(M,\mathcal{G})$ we furthermore define $\ad_{\mathbf{A}}\in \Omega^p(M,\End \mathcal{G})$ via 
$$\ad_{\mathbf{A}}(X_1,\dots,X_p)(\mathbf{r})=[\mathbf{A}(X_1,\dots,X_p),\mathbf{r}]_{\mathcal{G}}.$$

An autoequivalence $\nu\in \Auteq(\mathcal{G})$ covering $f\in \Diff(M)$ acts on the sections $\Gamma(M,\mathcal{G})$ via $\nu\cdot \mathbf{r}=\nu\circ \mathbf{r}\circ f^{-1}$ and on forms $\mathbf{A}\in\Omega^k(M,\mathcal{G})$ via
$$ (\nu\cdot \mathbf{A})(X_1,\dots,X_k)=\nu\left(\mathbf{A}( df^{-1}X_1,\dots,df^{-1}X_k)\right).$$
Furthermore, given a connection $\nabla$ on $\mathcal{G}$, the autoequivalence $\nu$ acts as
$$(\nu\cdot\nabla)_X r=\nu\cdot\nabla_{f_*^{-1}X}(\nu^{-1}\cdot \mathbf{r}), $$
where $f_*X=df\circ X\circ f^{-1}$ denotes the push-forward.

\begin{definition}
\label{def: standard regular}
A \emph{standard transitive Courant algebroid} $S[\nabla,\mathbf{R},H]$ over a manifold $M$ is a transitive Courant algebroid whose underlying vector bundle is $TM\oplus\mathcal{G}\oplus T^*M$ for some bundle of quadratic Lie algebras $\left(\mathcal{G},[\cdot,\cdot]_{\mathcal{G}},\langle\cdot,\cdot\rangle_{\mathcal{G}}\right)$ with anchor map the projection to $TM$, inner product
\begin{align*}
    \langle X +\mathbf{r}+\xi, Y +\mathbf{s} + \eta\rangle=\tfrac{1}{2}(\xi(Y)+\eta(X))+\langle \mathbf{r},\mathbf{s}\rangle_{\mathcal{G}}
\end{align*}
and Dorfman bracket depending on $H \in \Omega^3(M), \mathbf{R} \in \Omega^2(M, \mathcal{G})$ and some connection $\nabla$ on $\mathcal{G}$ satisfying the conditions
\begin{align*}
& X\langle \mathbf{r}, \mathbf{s}\rangle_{\mathcal{G}}=\left\langle\nabla_X \mathbf{r}, \mathbf{s}\right\rangle_{\mathcal{G}}+\left\langle \mathbf{r}, \nabla_X \mathbf{s}\right\rangle_{\mathcal{G}} \\
& \nabla_X[\mathbf{r}, \mathbf{s}]_{\mathcal{G}}=\left[\nabla_X \mathbf{r}, \mathbf{s}\right]_{\mathcal{G}}+\left[\mathbf{r}, \nabla_X \mathbf{s}\right]_{\mathcal{G}} \\
& \dd_{\nabla} \mathbf{R}=0 \\
& \dd_{\nabla}^2=\mathrm{ad}_\mathbf{R} \\
& \dd H=\tfrac{1}{2}\langle \mathbf{R} \wedge \mathbf{R}\rangle_{\mathcal{G}}.
\end{align*}
 The various components of the Dorfman bracket are given by
\begin{align*}
    \begin{aligned}
& {[X, Y]=i_Y i_X H + \mathbf{R}(X, Y) + \mathcal{L}_X Y} \\
& {[X, \mathbf{r}]=-[\mathbf{r}, X]=-2\langle \mathbf{r}, \iota_X \mathbf{R}\rangle_{\mathcal{G}} + \nabla_X \mathbf{r}} \\
& {[X, \xi]=\mathcal{L}_X \xi} \\
& {[\xi, X]=-i_X \dd\xi}\\
& {[\mathbf{r}, \mathbf{s}]=2 \left\langle \mathbf{s}, \nabla \mathbf{r}\right\rangle_{\mathcal{G}} +[\mathbf{r}, \mathbf{s}]_{\mathcal{G}}} \\
& {[\xi, \mathbf{r}]=[\mathbf{r}, \xi]=[\xi, \eta]=0}
\end{aligned}
\end{align*}
for $\mathbf{r},\mathbf{s}\in\Gamma(\mathcal{G})$, $X,Y\in \Gamma(TM)$ and $\xi,\eta\in \Gamma(T^*M)$.

\end{definition}

\begin{definition}
 A \emph{bisection} of a Courant algebroid $E$ is a Courant isomorphism from $E$ to a standard transitive Courant algebroid.
\end{definition}

\begin{theorem}[\cite{chen2013}]
\label{thm: trans standard}
    Every transitive Courant algebroid admits a bisection. 
\end{theorem}

\begin{remark}
    Explicitly, the isomorphism is given as follows: Let $\mathcal{G}:=\ker\pi/(\ker \pi)^{\perp}$, where the orthogonal complement is taken with respect to the scalar product of the Courant algebroid $E$. Then it turns out that $\mathcal{G}$ together with the induced scalar product and bracket is a bundle of quadratic Lie algebras. It was shown in~\cite{chen2013} that there exist
    \begin{enumerate}
        \item a vector bundle morphism $\lambda\colon TM\longrightarrow E$ such that $\pi\circ\lambda=\id$ and such that $\im\lambda$ is isotropic with respect to $\langle\cdot,\cdot\rangle$,
    \item and a vector bundle morphism $\sigma\colon \mathcal{G}\longrightarrow \ker \pi$ such that $\pr_{\mathcal{G}}\circ \sigma=\id$ such that $\im\sigma$ is orthogonal to $\im\lambda$ with respect to $\langle\cdot,\cdot\rangle$.
    \end{enumerate}
    The isomorphism is then given by
    \begin{align*}
        TM\oplus\mathcal{G}\oplus T^*M&\longrightarrow E\\
        X+\mathbf{r}+\alpha &\longmapsto \lambda(X)+\sigma(\mathbf{r})+\tfrac{1}{2}\pi^*\alpha.
    \end{align*}
\end{remark}

\begin{theorem}[\cite{chen2013}]
    Let $E$ be a transitive Courant algebroid, and let $S[\nabla,\mathbf{R},H]$ and $S[\hat{\nabla},\hat{\mathbf{R}}, \hat{H}]$ be different choices of bisections of $E$. Then there is some $\nu\in \O(\mathcal{G})$ covering the identity, $\mathbf{A}\in \Omega^1(M,\mathcal{G})$ and $B\in \Omega^2(M)$, such that the isomorphism $\delta\colon S[\nabla,\mathbf{R},H]\longrightarrow S[\hat{\nabla},\hat{\mathbf{R}}, \hat{H}]$ describing the change of bisection is given by
    \begin{align*}
        \delta=\begin{pmatrix}
            \id & 0 &0 \\
           \nu\circ \mathbf{A} &\nu & 0\\
            B-\langle \mathbf{A},\mathbf{A}\rangle_{\mathcal{G}}&-2\,\langle \mathbf{A},\cdot\rangle_{\mathcal{G}}&\id
        \end{pmatrix}.
    \end{align*}
    Moreover, the data defining the different standard transitive Courant algebroids is related by
    \begin{align*}
        \hat{\nabla}&=\nu\cdot (\nabla -\ad_\mathbf{A})\\
        \hat{\mathbf{R}}&=\nu\cdot \left(\mathbf{R}-\dd_{\nabla}\mathbf{A}+\tfrac{1}{2}\,[\mathbf{A},\mathbf{A}]_{\mathcal{G}}\right)\\
        \hat{H}&=H-\dd B-2\,\langle \mathbf{A}\wedge \mathbf{R}\rangle_{\mathcal{G}}+\langle \mathbf{A}\wedge \dd_{\nabla} \mathbf{A}\rangle_{\mathcal{G}}-\frac{1}{3}\,\langle \mathbf{A}\wedge [\mathbf{A},\mathbf{A}]_{\mathcal{G}}\rangle_{\mathcal{G}}.
    \end{align*}
\end{theorem}

An orthogonal equivalence $\nu\in\O(\mathcal{G})$ of $\mathcal{G}$ covering a diffeomorphism $f\in\mathrm{Diff}(M)$ defines a vector bundle morphism $F_{\nu}$ from $E$ to itself via
\begin{align*}
    F_{\nu}:=\left(\begin{array}{lll}
df & & \\
& \nu & \\
& & (df^{-1})^*
\end{array}\right)
\end{align*}
in the splitting $E=TM\oplus\mathcal{G}\oplus T^*M$.

Given a two-form $B\in\Omega^2(M)$ and a one-form $\mathbf{A}\in\Omega^1(M,\mathcal{G})$ with values in $\mathcal{G}$ we can furthermore define vector bundle morphisms $F_B$ and $F_{\mathbf{A}}$ from $E$ to itself via
\begin{align*}
    F_{B}:=\left(\begin{array}{ccc}
\id &0 & 0 \\
0& \id & 0\\
B&0 & \id
\end{array}\right),\quad 
    F_{\mathbf{A}}:=\left(\begin{array}{ccc}
\id & 0 & 0\\
\mathbf{A}& \id & 0\\
-\langle \mathbf{A},\mathbf{A}\rangle_{\mathcal{G}}&-2\langle \mathbf{A},\cdot\rangle_{\mathcal{G}} & \id
\end{array}\right).
\end{align*}
Note that $F_{\nu}$, $F_B$ and $F_{\mathbf{A}}$ all preserve the scalar product $\langle\cdot,\cdot\rangle$ on $E$ as well as the anchor map. We will denote the morphism obtained by composing $F_{\nu}$, $F_\mathbf{A}$ and $F_B$ by 
\begin{align*}
F_{\nu,\mathbf{A},B}:=F_{\nu}\circ F_{\mathbf{A}}\circ F_B.    
\end{align*}
The group $\O(E,\pi,\langle\cdot,\cdot\rangle)$ of equivalences of $E$ preserving the anchor map and the scalar product is given by \cite{coureaud2017}
\begin{align*}
\O(E,\pi,\langle\cdot,\cdot\rangle)=\{ F_{\nu,\mathbf{A},B}:\nu\in \O(\mathcal{G}),\,  \mathbf{A}\in\Omega^1(M,\mathcal{G}),\, B\in\Omega^2(M)\}. 
\end{align*}
The group of autoequivalences of $E$ is the subgroup of elements in $\O(E,\pi,\langle\cdot,\cdot\rangle)$ that are compatible with the Dorfman bracket of $E$. An explicit description is given in the following theorem.

\begin{theorem}[Theorem 2.37 in \cite{coureaud2017}]
\label{thm: gen automs trans}
    Let $E$ be a standard transitive Courant algebroid. Then the group of autoequivalences of $E$ is given by
    \begin{align*}
        \Auteq(E)= \{F_{\nu,\mathbf{A},B}:&\,\nu\in\Auteq(\mathcal{G})\text{ covers }f\in\mathrm{Diff}(M),\;\mathbf{A}\in\Omega^1(M,\mathcal{G}),\;B\in\Omega^2(M),\\
        &\nabla=\nu\cdot \left(\nabla-\ad_\mathbf{A}\right),\\
        &\mathbf{R}=\nu\cdot \left(\mathbf{R} -\dn \mathbf{A}+\tfrac{1}{2}[\mathbf{A}, \mathbf{A}]_{\mathcal{G}}\right),\\
        &H-f^* H=\dd B+2\langle \mathbf{A}\wedge \mathbf{R}\rangle_{\mathcal{G}}-\langle \mathbf{A}\wedge\dn \mathbf{A}\rangle_{\mathcal{G}}-\frac{1}{3}\left\langle \mathbf{A} \wedge[\mathbf{A} , \mathbf{A}]_{\mathcal{G}}\right\rangle_{\mathcal{G}}\}, 
    \end{align*}
    with the product
    \begin{align*}
        F_{\nu,\mathbf{A},B}\circ F_{{\nu}^{\prime}, \mathbf{A}^{\prime}, B^{\prime}}=F_{\nu^{\prime\prime},\mathbf{A}^{\prime\prime}, B^{\prime\prime}},
    \end{align*}
   where $\nu^{\prime\prime},\mathbf{A}^{\prime\prime},B^{\prime\prime}$ are given by
   \begin{align*}
     \nu^{\prime\prime}&=\nu\nu^{\prime},\\
     \mathbf{A}^{\prime\prime}&=\nu'^{-1}\cdot \mathbf{A}+\mathbf{A}^{\prime},\\
     B^{\prime\prime}&={f^{\prime}}^*B+B^{\prime}+\langle \nu'^{-1}\cdot \mathbf{A}\wedge \mathbf{A}^{\prime}\rangle_{\mathcal{G}}
   \end{align*}
   and the identity element is $((\id_M,0),\id,0)$. The inversion map is given by
    \begin{align*}
    (\nu,\mathbf{A},B)\longmapsto\left(\nu^{-1},-\nu\cdot \mathbf{A},-(f^{-1})^*B\right) 
    \end{align*}
\end{theorem}

\begin{remark}
    Note that in~\cite[Thm 2.37]{coureaud2017} there are some additional factors of $1/2$ appearing. This is due to a different convention for the scalar product of a standard transitive Courant algebroid. Moreover, we remark that the theorems stated above are special cases of more general results in~\cite{chen2013,coureaud2017} for the class of regular Courant algebroids (i.e.\ Courant algebroids whose anchor map has constant rank).
\end{remark}

The group of autoequivalences of a standard transitive Courant algebroid gives rise to a description of the group of autoequivalences of an arbitrary transitive Courant algebroid $E$ by conjugating with the isomorphism to $E$. 

Recall that a Courant algebroid $E$ over a manifold $M$ is called exact if the sequence
 \begin{align}
 \label{eq: exact CA}
0 \longrightarrow T^* \stackrel{\pi^*}{\longrightarrow} E \stackrel{\pi}{\longrightarrow} T \longrightarrow 0
 \end{align}
 is exact. In this case $\mathcal{G}$ vanishes and the above results reduce to the following.
 \begin{theorem}[\cite{Gualtieri:2003dx,gualtieri2011}]
 Every exact Courant algebroid over a smooth manifold $M$ is isomorphic to the generalized tangent bundle $\mathbb{T}M$ with anchor map the canonical projection to $TM$, scalar product defined by
\begin{align*}
    \langle X+\alpha, Y+\beta\rangle=\tfrac{1}{2}(\alpha(Y)+\beta(X))
\end{align*}
and with Dorfman bracket given by
\begin{align*}
    [X+\alpha, Y+\beta]_H=\mathcal{L}_X Y+\mathcal{L}_X \beta-\iota_Y \dd \alpha+H(X, Y, \cdot),
\end{align*}
where $H\in\Omega^3(M)$ is a fixed closed three-form. The isomorphism is defined by a choice of isotropic splitting of the sequence \eqref{eq: exact CA}. Any two choices of such isomorphisms to $(\mathbb{T}M,H)$ and $(\mathbb{T}M,\hat{H})$ are related by the action of some two-form $B\in\Omega^2(M)$ acting on $TM\oplus T^*M$ via
$$F_B:=\begin{pmatrix}
    \id & B\\
    0 & \id
\end{pmatrix}$$
and we must have $\hat{H}=H+\dd B$.

The group of autoequivalences of $(\mathbb{T}M,H)$ is given by
\begin{align*}
        \Auteq(\mathbb{T}M,H)=\{ F_{f,B}:=F_f F_B:H-f^*H=\dd B\},
    \end{align*}
    where for $f\in\Diff(M)$
    \begin{align*}
    F_f:=\begin{pmatrix}
df & 0 \\
0 & {df^{-1}}^*
\end{pmatrix}.
\end{align*}
The multiplication is given by
    \begin{align*}
        F_{f,B}F_{f',B'}=F_{f'',B''},
    \end{align*}
    where $f''=f\circ f'$ and $B''=B'+f'^*B$.
 \end{theorem}

Later, we will consider the following subclass of transitive Courant algebroids for which the bisections take a particularly simple form.
\begin{definition} \label{def: g-flat CA}
    A $\mathcal{G}$-\emph{flat} standard transitive Courant algebroid is a standard transitive Courant algebroid such that $\mathcal{G}\cong M\times \R^K$ is a trivial vector bundle with trivial Lie bracket on the fibers and such that $\nabla$ is flat with trivial holonomy and $\mathbf{R}=0$. A transitive Courant algebroid is called $\mathcal{G}$-\emph{flat} if it is isomorphic to a $\mathcal{G}$-flat standard transitive Courant algebroid. 
\end{definition}

\begin{proposition}
    If $S[\nabla,\mathbf{R},H]$ is a standard $\mathcal{G}$-flat transitive Courant algebroid, then $H$ is closed. 
\end{proposition}
\begin{proof}
    This follows from the condition that $\dd H=\tfrac{1}{2}\langle \mathbf{R} \wedge \mathbf{R}\rangle_{\mathcal{G}}=0$.
\end{proof}

\subsection{Generalized complex structures}
\begin{definition}
    A \emph{generalized almost complex structure} on a Courant algebroid $E\rightarrow M$ is an endomorphism field $\mathcal{J}\in \Gamma(\End E)$ such that $\mathcal{J}^2=-\id$ and such that $\mathcal{J}$ is skew-symmetric with respect to the scalar product of $E$. It is called \emph{integrable} or a \emph{generalized complex structure} if the subbundles $L(\mathcal{J})$, $\bar{L}(\mathcal{J})$ of $E_{\C}$ given by the $+i$- and $-i$-eigenspaces of $\mathcal{J}$ are involutive with respect to the Dorfman bracket. 
\end{definition}
An equivalent definition of a generalized complex structure on a Courant algebroid can be formulated directly in terms of the subbundle $L\subset E_{\C}$:
\begin{definition}
A \emph{generalized almost complex structure} on $E$ is a maximally isotropic subbundle $L\subset E_{\mathbb{C}}$ of the complexification $E_{\C}$ of $E$ such that $E_{\C}=L\oplus \bar{L}$. A generalized almost complex structure is called \emph{integrable} or a \emph{generalized complex structure} if it is closed under the Dorfman bracket.
\end{definition}

Given a generalized complex structure $L$, it follows from the third axiom in Definition~\ref{def: courant algebroid} and the fact that $L$ and $\bar{L}$ are isotropic that these bundles together with the restrictions of $\pi$ and the Dorfman bracket of $E$ form Lie algebroids. 

Both complex and symplectic structures on a manifold $M$ can be seen as examples of generalized complex structures on $(\mathbb{T}M,H)$ for appropriate choices of $H$~\cite{Gualtieri:2003dx,gualtieri2011}.

\subsubsection{Deformations of generalized complex structures}
Let $L$ be a generalized complex structure on some Courant algebroid $E$. Every maximal isotropic having trivial intersection with $\bar{L}$ is described as the graph of some homomorphism $\varphi\colon L\rightarrow \bar{L}$ satisfying $\langle\varphi u, v\rangle+\langle u,\varphi v\rangle=0$ for all $u,v\in \Gamma(L)$. After identifying $\bar{L}\cong L^*$ via the scalar product $\langle\cdot,\cdot\rangle$ of $E$, we may also think of such $\varphi$ as a form $\varphi\in \Gamma(M,\Lambda^2 L^*)$ and the new isotropic is given by
\begin{align*}
    L_{\varphi}:=(1+\varphi)L=\{u+\iota_u\varphi:u\in L\}.
\end{align*}
The isotropic $L_{\varphi}$ will be a generalized almost complex structure if and only if $L_{\varphi}\oplus \overline{L_{\varphi}}=E_{\C}$. Since this is an open condition, it will be satisfied for every $\varphi\in\Gamma(M,\Lambda^2 L^*)$ in a small enough neighborhood of zero. We will call such $\varphi$ a \emph{generalized almost complex deformation}. Next, we will discuss under which conditions the deformed generalized almost complex structure $L_{\varphi}$ is integrable. Before stating the condition, we need to introduce some more structure on the Lie algebroids $L$ and $\bar{L}$.

For every Lie algebroid $L$, we define the Lie algebroid differential
\begin{align*}
    \dd_L\colon \Gamma(\Lambda^k L^*)\longrightarrow \Gamma(\Lambda^{k+1} L^*)
\end{align*}
via
\begin{align*}
    (\dd_L\varphi)(u_0,\dots,u_k):=&\sum_i (-1)^i\pi(u_i)\varphi(u_0,\dots,\hat{u_i},\dots,u_k)\\
    &+\sum_{i<j}(-1)^{i+j}\varphi([u_i,u_j],u_0,\dots,\hat{u_i},\dots,\hat{u_j},\dots,u_k)
\end{align*}
for $\varphi\in\Gamma(\Lambda^k L^*)$, and $u_0,\dots,u_k\in\Gamma(L)$. Then $(\Gamma(\Lambda^{\bullet} L^*),\dd_L)$ forms a differential graded algebra, i.e.~$\dd_L$ is a first order linear differential operator of degree one that satisfies $(\dd_L)^2=0$ and $$\dd_L(\varphi\wedge \psi)=(\dd_L\varphi)\wedge \psi+(-1)^p\varphi\wedge \dd_L\psi\quad \text{for }\varphi\in\Gamma(\Lambda^p L^*),\,\psi\in\Gamma(\Lambda^q L^*).$$

Furthermore, for every Lie algebroid $L$ the \emph{Schouten bracket} 
$$[\cdot,\cdot]_L\colon \Gamma(\Lambda^p L)\times\Gamma(\Lambda^q L)\longrightarrow \Gamma(\Lambda^{p+q-1}L)$$
is defined as follows:
For $u_1,\dots,u_p, v_1,\dots,v_q\in\Gamma(L)$, we set
\begin{align*}
    &[u_1\wedge\dots\wedge u_p,v_1\wedge\dots\wedge v_q]_L\\&:=\sum_{i,j}(-1)^{i+j}[u_i,v_j]\wedge u_1\wedge\dots\wedge \hat{u}_i\wedge\dots\wedge u_p\wedge v_1\wedge\dots\wedge \hat{v}_j\wedge \dots\wedge v_q.
\end{align*}
For $u\in \Gamma(L)$, $f\in C^{\infty}(M)$, we set
\begin{align*}
    [u,f]_L=-[f,u]_L:=\pi(u)(f).
\end{align*}
Finally, we set
\begin{align*}
&[u_1\wedge\dots\wedge u_p,f]_L=(-1)^p[f,u_1\wedge\dots\wedge u_p]_L\\
&=:(-1)^p\left([f,u_1]u_2\wedge\dots\wedge u_p-u_1\wedge [f,u_2]u_3\wedge\dots\wedge u_p\right.\\
&\quad\left. +\dots+(-1)^{p-1}u_1\wedge\dots\wedge [f,u_p]\right).
\end{align*}

\begin{lemma}
\label{lem: bracket and dL}
    Let $L$ be a generalized complex structure. We identify $\bar{L}\cong L^*$. Then for $U\subset M$, $\varphi\in\Gamma(U,\Lambda^p L^*)$, $\psi\in\Gamma(U,\Lambda^q L^*)$, $\eta\in\Gamma(U,\Lambda^r L^*)$ the following hold
    \begin{enumerate}
        \item $[\varphi,\psi]_{L^*}=-(-1)^{(p-1)(q-1)}[\psi,\varphi]_{L^*}$
        \item $[\varphi,[\psi,\eta]_{L^*}]_{L^*}=[[\varphi,\psi]_{L^*},\eta]_{L^*}+(-1)^{(p-1)(q-1)}[\psi,[\varphi,\eta]_{L^*}]_{L^*}$
        \item $\dd_L [\varphi,\psi]_{L^*}=[\dd_L\varphi,\psi]_{L^*}+(-1)^{p-1}[\varphi,\dd_L\psi]_{L^*}$.   
    \end{enumerate}
\end{lemma}
\begin{proof}
    Properties 1 and 2 are general properties of the Schouten bracket. Property 3 follows from the fact that $L$ and $\bar{L}\cong L^*$ form a Lie bialgebroid (see Theorem 2.6 in~\cite{liu1997manin}).
\end{proof}

We will now state the condition for integrability:
\begin{theorem}[Theorem 6.1 in~\cite{liu1997manin}]
    The generalized almost complex structure $L_{\varphi}$ corresponding to a generalized almost complex deformation $\varphi \in \Gamma(M,\Lambda^2L^*)$ is integrable if and only if $\varphi$ satisfies the Maurer--Cartan equation
    \begin{align}
    \label{eq: Maurer-Cartan gc}
        \dd_L\varphi +\tfrac{1}{2}[\varphi,\varphi]_{L^*}=0.
    \end{align}
    In this case $\varphi$ is called a \emph{generalized complex deformation}.
\end{theorem}

We will sometimes drop the subscript $L$ and $L^*$ of the Schouten bracket. In the following we will collect some identities for generalized complex structures that will become very useful in Section~\ref{defothm:sec}.
\begin{lemma}
\label{lem: different formulas for dL and brackets}
    Let $L$ be a generalized complex structure. We identify $\bar{L}\cong L^*$. Then for $U\subset M$, $\varphi,\varphi'\in\Gamma(U,\Lambda^2 L^*)$, $\xi\in\Gamma(U,L^*)$, $u,v,w\in\Gamma(U,L)$ the following hold
    \begin{enumerate}
        \item $\dd_L\xi(u,v)=-\langle [\xi,u],v\rangle$
        \item $\dd_L\varphi(u,v,w)=\langle [u,\iota_v\varphi]+[\iota_u\varphi,v]-\iota_{[u,v]}\varphi,w\rangle$
        \item $[\varphi,\xi](u,v)=\langle [\iota_u\varphi,\xi]+\iota_{\pr_L[\xi,u]}\varphi,v\rangle$
        \item $$\begin{aligned}
        [\varphi,\varphi'](u,v,w)=&\langle [\iota_u\varphi,\iota_v\varphi']-[\iota_v\varphi,\iota_u\varphi'],w\rangle\\
        &+\tfrac{1}{2}\langle \iota_{\pr_L\left([v,\iota_u\varphi]-[\iota_u\varphi,v]-[u,\iota_v\varphi]+ [\iota_v\varphi,u]\right)}\varphi',w\rangle\\
        &+\tfrac{1}{2}\langle \iota_{\pr_L\left([v,\iota_u\varphi']-[\iota_u\varphi',v]-[u,\iota_v\varphi']+[\iota_v\varphi',u]\right)}\varphi,w\rangle
        \end{aligned}$$
        in particular, $$\begin{aligned}
        [\varphi,\varphi](u,v,w)=&2\langle [\iota_u\varphi,\iota_v\varphi],w\rangle+\langle \iota_{\pr_L\left([v,\iota_u\varphi]-[\iota_u\varphi,v]-[u,\iota_v\varphi]+[\iota_v\varphi,u]\right)}\varphi,w\rangle.
        \end{aligned}$$
        
    \end{enumerate}
\end{lemma}
\begin{proof}
See Section~\ref{appendix proof of lemma} in the appendix.   
\end{proof}

\subsubsection{Families of deformations}
Next, we will introduce \emph{smooth families of generalized almost complex deformations}. For this, let $S$ be an open neighborhood of $0$ in some finite-dimensional vector space.
\begin{definition}
    A \emph{(smooth) family of generalized almost complex deformations} of a generalized complex structure $L$ is a smooth map
    \begin{align*}
        \varphi\colon S\longrightarrow \Gamma(M,\Lambda^2 L^*)
    \end{align*}
    such that $\varphi(0)=0$ and such that for every $s\in S$, the form $\varphi(s)$ is a generalized almost complex deformation. The family $\{\varphi(s):s\in S\}$ is a \emph{(smooth) family of generalized complex deformations} if in addition for every $s\in S$, the form $\varphi(s)$ satisfies equation~\eqref{eq: Maurer-Cartan gc}.
\end{definition}

The group of autoequivalences $\Auteq(E)$ of a Courant algebroid $E$ acts on the space of generalized almost complex structures $\operatorname{GAC}(E)$ on $E$ via
\begin{align*}
\mu_{\mathrm{GAC}}\colon \Auteq(E)\times \operatorname{GAC}(E) &\longrightarrow \operatorname{GAC}(E)\\
    (F,\mathcal{J})&\longmapsto F\cdot \mathcal{J}:= F\circ \mathcal{J}\circ F^{-1}.
\end{align*}
If the (1-jet of the) autoequivalence $F$ is close enough to the identity in an appropriate sense, then $F\cdot \mathcal{J}$ can be described by a generalized (almost) complex deformation $\Theta(F)\in \Gamma(M,\Lambda^2 L^*)$ of $\mathcal{J}$ (see Section~\ref{sect: auteq acting on gc defs} for more details). Similarly, given a generalized almost complex deformation $\varphi\in\Gamma(M,\Lambda^2 L^*)$ of $\mathcal{J}$, for $F$ close enough to the identity, the generalized almost complex structure $F\cdot \mathcal{J}_{\varphi}$ is again described by a generalized almost complex deformation $\Theta(F,\varphi)\in \Gamma(M,\Lambda^2 L^*)$. We will also denote $F\cdot\varphi:=\Theta(F,\varphi)$. We will consider two generalized complex deformations equivalent if they are related by the action of some autoequivalence.
\begin{definition}
    A family $\varphi'\colon S'\longrightarrow \Gamma(M,\Lambda^2 L^*)$ of generalized complex deformations is \emph{obtainable from} $\varphi\colon S\longrightarrow \Gamma(M,\Lambda^2 L^*)$
    if there is a smooth map $\tau\colon S'\rightarrow S$ such that $\tau(0)=0$ and a family of autoequivalences $\{F_{s'}\}_{s'\in S'}$ depending smoothly on $s'$, with $F_0=\id$ such that for every $s'\in S'$
    \begin{align*}
        \varphi(\tau(s'))=F_{s'}\cdot\varphi'(s').
    \end{align*}
\end{definition}
\begin{definition}
    A family of deformations $\{\varphi(s):s\in S\}$ is called \emph{locally complete} if for every family of deformations the restriction to some neighborhood of $0$ is obtainable from $\{\varphi(s):s\in S\}$.
\end{definition}

\subsubsection{The Lie algebroid complex of a generalized complex structure}\label{sect: LA of GC}
Next, we will take a closer look at the Lie algebroid complex
\begin{align}
\label{eq: LA complex}
    \dd_L\colon\Gamma(\Lambda^k L^*)\longrightarrow\Gamma(\Lambda^{k+1} L^*).
\end{align}
We will see that if $L$ is a generalized complex structure on a transitive Courant algebroid, then this complex is elliptic. First, we will show the following Lemma:
\begin{lemma}
\label{lem: proj pi(L) to real part}
We denote by $\Re\colon TM_{\mathbb{C}}\to TM$ the projection to the real part of $TM_{\mathbb{C}}$. Suppose $E$ is a transitive Courant algebroid with a generalized complex structure $L\subset E_{\mathbb{C}}$. Then the map
    \begin{align*}
        \Re\circ \pi\colon L\to TM    
    \end{align*}
    is surjective.
\end{lemma}
\begin{proof}
Recall that $L\subset E_{\mathbb{C}}$ is the $\ii$-eigenbundle of the generalized almost complex structure $\mathcal{J}$ on $E$, which is given by
\begin{align*}
    L=\left\{ u - \ii\mathcal{J}u\,:\,u\in E\right\}.
\end{align*}
By the assumption that $E$ is transitive, we obtain 
\begin{align*}
    \Re\circ \pi(L)=\pi(E)=TM.  
\end{align*}  
\end{proof}

\begin{proposition}
    If $E$ is transitive, then the Lie algebroid complex $(\Gamma(\Lambda^{\bullet}L^*),\dd_L)$ is elliptic.
\end{proposition}
\begin{proof}
For $\alpha\in T^*M$, the symbol is given by
\begin{align*}
    s_{\alpha}(\dd_L)=\mathrm{pr}_{L^*}\left(\pi^*\alpha\right)\wedge\cdot\,.
\end{align*}
We choose $v\in L$ and compute
\begin{align*}
    \mathrm{pr}_{L^*}\left(\pi^*\alpha\right)(v)=\langle\pi^*\alpha,v\rangle=\alpha\circ\pi(v).
\end{align*}
Suppose $\alpha\circ\pi(v)=0$ for all $v\in L$. By decomposing the equation into real and imaginary parts and using Lemma~\ref{lem: proj pi(L) to real part}, this implies that $\alpha=0$. 

In particular, this shows that the symbol sequence is exact and therefore the complex $(\Gamma(\Lambda^{\bullet}L^*),\dd_L)$ is elliptic.
\end{proof}

\begin{corollary}
    Let $L$ be a generalized complex structure on a transitive Courant algebroid $E\to M$ over a compact manifold $M$. Then the cohomology $H^{\bullet}(M,L)$ of the complex~\eqref{eq: LA complex}, called \emph{Lie algebroid cohomology}, is a finite-dimensional graded ring.
\end{corollary}

For the remainder of this section, we will assume that $L$ is a generalized complex structure on a transitive Courant algebroid $E\to M$ over a compact orientable manifold $M$. We fix an auxiliary Riemannian metric on $M$ and a Hermitian bundle metric $h_L$ on $L$. These induce a scalar product $(\cdot,\cdot)$ on $\Gamma(M,\Lambda^{\bullet}L^*)$ via
\begin{align*}
    (\alpha,\beta):=\int_M h_L(\alpha,\beta) * 1.
\end{align*}
We denote by $\dd_L^*$ the formal adjoint of $\dd_L$ with respect to $(\cdot,\cdot)$ and define the Laplace operator 
\begin{align*}
\Delta_L:=\dd_L\dd_L^*+\dd_L^*\dd_L.    
\end{align*}
Moreover, we denote by  
\begin{align*}
    \mathcal{H}_L^{k}(M):=\ker \Delta_L|_{\Gamma(M,\Lambda^k L^*)}=\ker \dd_L|_{\Gamma(M,\Lambda^k L^*)}\cap \ker \dd_L^*|_{\Gamma(M,\Lambda^k L^*)}
\end{align*}
the $\Delta_L$-\emph{harmonic forms}.

\begin{theorem}[Hodge decomposition]
There exists a natural orthogonal decomposition 
\begin{align*}
    \Gamma(M,\Lambda^k L^*)=\dd_L\Gamma(M,\Lambda^{k-1}L^*)\oplus\mathcal{H}_L^k(M)\oplus \dd_L^*\Gamma(M,\Lambda^{k+1}L^*).
\end{align*}
Moreover, the canonical projection $\mathcal{H}^k_L(M)\rightarrow H^k(M,L)$ is an isomorphism.
\end{theorem}

We denote by $\HprojL\colon \Gamma(M,\Lambda^{\bullet}L^*)\rightarrow \mathcal{H}_L^{\bullet}(M)$ the orthogonal projection and by $\G_L\colon\Gamma(M,\Lambda^{\bullet}L^*)\rightarrow\Gamma(M,\Lambda^{\bullet}L^*)$ the \emph{Green's operator}, satisfying 
    \begin{align*}
        \im \G_L=\ker\HprojL,\quad\ker \G_L=\im\HprojL,\quad \HprojL+\G_L\Delta_L=\id.
    \end{align*}
Moreover, we define the operator $\Q_L:=\dd_L^*\G_L$. We note the following standard identities (see e.g.\ \cite{Wells2007-er}).
\begin{proposition}
    The operators $\HprojL$, $\G_L$ and $\Q_L$ satisfy 
    \begin{enumerate}
        \item $\G_L\dd_L=\dd_L \G_L$, $\G_L\dd_L^*=\dd_L^*\G_L$;
        \item $\HprojL+\dd_L \Q_L+\Q_L\dd_L=\id$;
        \item $\Q_L^2=\dd_L^*\Q_L=\Q_L\dd_L^*=\HprojL \Q_L=\Q_L\HprojL=0$.
    \end{enumerate}
\end{proposition}

\section{Autoequivalences as tame Lie groups}\label{sect: all autoeq}

In this section, we will endow the group of autoequivalences of a $\mathcal{G}$-flat transitive Courant algebroid with a tame Fréchet Lie group structure. Most of the required concepts and results needed from Hamilton--Nash--Moser theory are reviewed in~\cite[Sec~3]{paperoncomplexdefs}. Those, which are not covered and required for the present paper are summarized in Section~\ref{appendix HNM} of the appendix. For a comprehensive introduction to Hamilton--Nash--Moser theory we refer the reader to~\cite{hamilton82}.

We will construct the tame Lie group structure in two steps: First, we will consider the group of autoequivalences of an exact Courant algebroid in Section~\ref{sect: autoequ exact case}. We will see that these admit a tame Fréchet Lie group structure and there is a tame Fréchet Lie subgroup of so-called \emph{exact} autoequivalences, which will be relevant to the deformation theory of generalized complex structures in Section~\ref{defothm:sec}. Building on these results, we will show in Section~\ref{sect: auteq G-flat transitive} that the group of autoequivalences of every $\mathcal{G}$-flat transitive Courant algebroid is a tame Lie group and there is again a tame subgroup of exact autoequivalences. In Section~\ref{sect: auteq acting on gc defs} we will then study the action of (exact) autoequivalences on generalized almost complex structures and show that this action is smooth tame. 

\subsection{Autoequivalences of exact Courant algebroids}\label{sect: autoequ exact case}

As a first step, we will show that the autoequivalences of an exact Courant algebroid form a tame Fréchet Lie group. Moreover, we will see that there is a tame Lie subgroup of \emph{exact} autoequivalences, whose infinitesimal action on sections is by inner derivations.

\subsubsection{The group of autoequivalences as a tame Lie group}\label{sect: autoequ of exact CA are Lie group}

In the following, we will endow the group of autoequivalences of an exact Courant algebroid over a compact manifold with a tame Lie group structure. Note that it was already proved by Rubio and Tipler in~\cite{rubio20} that the group of autoequivalences of an exact Courant algebroid over a compact manifold is a strong inverse limit Hilbert (ILH) Lie group. 

Since the definition of strong ILH Lie groups is even more involved than that of tame Fréchet Lie groups, we will refer to~\cite{Omori2016} for an introduction to strong ILH Lie groups and simply note that these are special kinds of Fréchet Lie groups that are modeled on an ILH chain, a particular type of Fréchet space, such that in a local chart around the identity, the group multiplication and the inversion map satisfy certain properties. Below, we will follow a different approach to describe a (tame) Fr\'echet Lie group structure for the group of autoequivalences, which has already been laid out in the master thesis of Heck~\cite{jans_thesis}. We will then see that this Fr\'echet Lie group structure is equivalent to the one described in~\cite{rubio20}.

For the remainder of this section, let $E\rightarrow M$ be the generalized tangent bundle with $H$-twisted Dorfman bracket over a compact manifold $M$. We will denote by 
\begin{align*}
    \Diff(M)_{[H]}:=\{f\in\Diff(M):[f^*H-H]=0\}
\end{align*}
the group of diffeomorphisms preserving $[H]\in H^3_{\mathrm{dR}}(M)$. For $f\in \Diff(M)$ the class $[H-f^*H]\in H^3_{\mathrm{dR}}(M)$ only depends on the connected component of $f\in \Diff(M)$. We conclude that it is a tame Fréchet Lie subgroup of $\Diff(M)$ (containing the connected component of the identity in $\Diff(M)$).  

Furthermore, we will fix an auxiliary Riemannian metric $h$ on $M$. Recall the atlas for $\Diff(M)$ defined using the Riemannian exponential map $\exp^h$ of $h$ (see e.g. \cite[Sec 3.3.1]{paperoncomplexdefs}). This gives rise to an atlas of $\Diff(M)_{[H]}$, such that in a neighborhood of $f\in \Diff(M)_{[H]}$, the tame Fr\'echet Lie group $\Diff(M)_{[H]}$ is locally modeled on a neighborhood of 0 in $\Gamma(M,f^*TM)$.

The autoequivalences $\Auteq(E)$ are in bijection with $\Diff(M)_{[H]}\times \Omega^2_{\mathrm{cl}}(M)$ via 
\begin{align*}
    \Diff(M)_{[H]}\times \Omega^2_{\mathrm{cl}}(M)\ni (f,b)\longmapsto F_{f,B(f)+b}\in \Auteq(E)
\end{align*}
where $B(f):=\QdR (H-f^*H)$. Here $\QdR=\dd^*\G_{\mathrm{dR}}$, where $\G_{\mathrm{dR}}$ is the Greens operator of the Laplacian $\Delta_{\mathrm{dR}}=\dd\dd^*+\dd^*\dd$ and $\dd^*$ is the formal adjoint of the de Rham differential (with respect to the $L^2$-inner product). We have the following: 
\begin{theorem}[\cite{jans_thesis}]\label{thm: autoequivalences of exact CA}
    Let $E\rightarrow M$ be the generalized tangent bundle of a compact manifold $M$ with $H$-twisted Dorfman bracket. Then its group of autoequivalences is a tame Fréchet Lie group. The underlying manifold structure is given by $\Diff(M)_{[H]}\times \Omega^2_{\mathrm{cl}}(M)$ and the multiplication is given by
    \begin{align*}
        \left((f_1,b_1),(f_2,b_2)\right)\longmapsto (f_3,b_3),
    \end{align*}
    where $f_3=f_1\circ f_2$ and
    \begin{align*}
        b_3= f_2^*b_1+b_2+[f_2^*,\QdR](H-f_1^*H).
    \end{align*}
    The inversion map is given by
    \begin{align*}
        (f,b)\longmapsto \left(f^{-1},-(f^{-1})^*b+[\QdR,(f^{-1})^*](H-f^*H)\right).
    \end{align*}
    The Lie algebra $\auteq(E)$ of the group of autoequivalences is the space $\X(M)\oplus \Omega^2_{\mathrm{cl}}(M)$ with bracket given by
\begin{align*}
[(X,b),(Y,c)]=\left(\mathcal{L}_Y X, \mathcal{L}_Y b-\mathcal{L}_X c+(\mathcal{L}_X \QdR\mathcal{L}_Y-\mathcal{L}_Y\QdR\mathcal{L}_X)H-\QdR(\mathcal{L}_{\mathcal{L}_X Y}H)\right).
\end{align*}
\end{theorem}
We will omit the proof of Theorem~\ref{thm: autoequivalences of exact CA} since most of the computations and arguments necessary for the proof will be explained in the proof of the more general Theorem~\ref{thm: Auteq of flat std trans is tame Lie group}.

\begin{remark}\label{rk: from rubio-tipler to our tame structure}
    To argue that the Fréchet Lie group structure obtained in Theorem~\ref{thm: autoequivalences of exact CA} is indeed equivalent to that described in~\cite{rubio20}, let us briefly review the construction in \cite{rubio20}. 

    First, Rubio and Tipler consider the group $\Diff(M)\ltimes \Omega^2(M)$ with neutral element $(\id,0)$, multiplication given by \[(f_1,B_1)(f_2,B_2)=(f_1\circ f_2,f_2^*B_1+B_2)\] and inversion given by $(f,B)^{-1}=(f^{-1},-B)$. This is a strong ILH Lie group and in particular a Fréchet Lie group. Moreover, elements of this group are precisely equivalences of the generalized tangent bundle preserving both the anchor map and the scalar product. Then Rubio and Tipler argue that the group of autoequivalences of the generalized tangent bundle $E$ with $H$-twisted Dorfman bracket is a strong ILH Lie subgroup of $\Diff(M)\ltimes \Omega^2(M)$ by defining the map\footnote{In \cite{rubio20}, the map $\tilde{\rho}$ is defined with a minus sign in front of $\dd B$, due to a different sign convention for $H$.}
    \begin{align*}
        \tilde{\rho}\colon \Omega^3(M)\times (\Diff(M)\ltimes \Omega^2(M))&\longrightarrow \Omega^3(M)\\
        \left(h,(f,B)\right)&\longmapsto f^*h+\dd B,
    \end{align*}
    observing that $\Auteq(E)=\tilde{\rho}(H,\cdot)^{-1}(H)$ and using an implicit function theorem to show that this indeed has a strong ILH submanifold structure. Again, since a strong ILH Lie group structure induces a Fréchet Lie group structure, this proves that the autoequivalences of the generalized tangent bundle form a Fréchet Lie group.

    To see how their construction relates to ours, we consider the (tame) Fréchet Lie group $\Diff(M)\ltimes_H \Omega^2(M)$, whose underlying (tame) Fréchet manifold is $\Diff(M)\times \Omega^2(M)$ with neutral element $(\id,0)$, multiplication map \[(f_1,b_1)\circ (f_2,b_2)=(f_1\circ f_2,b_3),\]
    where \[ b_3= f_2^*b_1+b_2+[f_2^*,\QdR](H-f_1^*H)\] and with inversion map \[(f,b)^{-1}=\left(f^{-1},-(f^{-1})^*b+[\QdR,(f^{-1})^*](H-f^*H)\right).\]
    There is a smooth (tame) isomorphism of Lie groups
    \begin{align*}
        \Phi_H\colon \Diff(M)\ltimes \Omega^2(M)&\longrightarrow \Diff(M)\ltimes_H \Omega^2(M)\\
         (f,B)&\longmapsto \left(f, B-\QdR(H-f^*H)\right).
    \end{align*}
    We observe that
    \begin{align*}
        \Phi_H\left(\tilde{\rho}(H,\cdot)^{-1}(H)\right)=\left(\tilde{\rho}(H,\Phi_H^{-1}\cdot)\right)^{-1}(H)=\Diff(M)_{[H]}\times \Omega^2_{\mathrm{cl}}(M).
    \end{align*}
    Moreover, the map $\Phi_H^{-1}(f,b)=\left(f,b+\QdR(f^*H-H)\right)$ is precisely the bijection that we used in order to identify $\Auteq(E)$ with $\Diff(M)_{[H]}\times\Omega^2_{\mathrm{cl}}(M)$. This shows that the (tame) Fr\'echet Lie group structure described in Theorem~\ref{thm: autoequivalences of exact CA} is related to the one in \cite{rubio20} by a conjugation with $\Phi_H$. Since this is a smooth (tame) map, we conclude that both Fr\'echet Lie group structures are equivalent (see Section~\ref{appendix HNM} of the appendix for the facts that $\QdR$ is tame linear and that the pullback of forms is smooth tame). 
\end{remark}

Finally, note that by choosing a splitting, we can use Theorem~\ref{thm: autoequivalences of exact CA} in order to find that the group of autoequivalences of every exact Courant algebroid is a tame Fréchet Lie group. To see that this is independent of the choice of splitting, recall that different choices of splittings are related by conjugation with a $B$-field transformation. This gives rise to a tame Lie group isomorphism between the groups of autoequivalences corresponding to different choices of splittings.

\subsubsection{Exact autoequivalences of an exact Courant algebroid}\label{sect: exact autoequ of exact CA}
In this section, we will introduce a tame Lie subgroup of the Courant autoequivalences, which will be important in the proof of local completeness of the deformation family in Section~\ref{defothm:sec}. First, we will describe the corresponding Lie subalgebra. To this end, we consider the map
\begin{align*}
    q\colon \X(M)\oplus \Omega_{\mathrm{cl}}^2(M)&\longrightarrow \mathcal{H}^2(M)\\
    (X,b)&\longmapsto \Hproj(b+\iota_X H)= \Hproj(b-\QdR \mathcal{L}_X H+\iota_X H),
\end{align*}
where $\mathcal{H}^2(M)$ denotes the harmonic 2-forms on $M$ (with respect to the auxiliary Riemannian metric $h$) and $\pr_{\mathcal{H}}$ denotes the projection to harmonic forms.

We denote $\autex(E):=\ker q\subset \auteq(E)$. Note that since $\dd\QdR\mathcal{L}_XH=\dd\QdR \dd\iota_X H=\dd \iota_X H$, for $b\in \Omega^2_{\mathrm{cl}}(M)$, we have 
\[ \dd(b-\QdR \mathcal{L}_X H+\iota_X H)=0.\]
Hence, $(X,b)\in \ker q$ if and only if there is some $\alpha \in \Omega^1(M)$ such that 
\[b-\QdR \mathcal{L}_X H+\iota_X H=\dd\alpha.\]

\begin{proposition}
    The space $\autex(E):=\ker q\subset \auteq(E)$ is a tame Lie subalgebra of finite codimension. 
\end{proposition}
\begin{proof}
Since $\mathcal{H}^2(M)$ is finite-dimensional and $\auteq(E)=\autex(E)\oplus \mathcal{H}^2(M)$, it follows that $\autex(E)$ is a tame subspace of finite codimension. To see that it is a Lie subalgebra, we consider $(X,b),(Y,c)\in\autex(E)$. Then there are some $\alpha,\beta\in\Omega^1(M)$ such that 
\begin{align*}
    b=\dd \alpha+\QdR \mathcal{L}_X H-\iota_X H,\quad c=\dd \beta +\QdR \mathcal{L}_Y H-\iota_Y H.
\end{align*} 
We compute
\begin{align*}
    &[(X,b),(Y,c)]\\
    &=(\mathcal{L}_Y X,\mathcal{L}_Y(\dd\alpha +\QdR \mathcal{L}_X-\iota_X H)-\mathcal{L}_X(\dd\beta +\QdR \mathcal{L}_Y H-\iota_Y H)+\\&\quad+(\mathcal{L}_X \QdR\mathcal{L}_Y-\mathcal{L}_Y\QdR\mathcal{L}_X)H-\QdR(\mathcal{L}_{\mathcal{L}_XY}H))\\
    &=(\mathcal{L}_Y X,\mathcal{L}_Y \dd \alpha-\mathcal{L}_X \dd \beta-\mathcal{L}_X\iota_Y H+\mathcal{L}_Y \iota_X H-\QdR\mathcal{L}_{\mathcal{L}_XY}H)\\
    &=(\mathcal{L}_Y X,\dd(\iota_Y\dd \alpha-\iota_X\dd\beta+ \iota_X \iota_Y H)+\iota_X\mathcal{L}_Y H-\mathcal{L}_Y \iota_X H-\QdR\mathcal{L}_{\mathcal{L}_XY}H)\\
    &=(\mathcal{L}_Y X,\dd(\iota_X\dd\beta-\iota_Y\dd \alpha- \iota_X \iota_Y H)+\QdR\mathcal{L}_{\mathcal{L}_Y X}H-\iota_{\mathcal{L}_Y X}H),
\end{align*}
where we used the identities $\mathcal{L}_X=\dd \iota_X+\iota_X\dd$ and $[\mathcal{L}_Y,\iota_X]=\iota_{\mathcal{L}_Y X}$, as well as the fact that $H$ is closed. This shows that $[(X,b),(Y,c)]\in \autex(E)$ for all $(X,b),(Y,c)\in\autex(E)$ and hence $\autex(E)$ is a subalgebra.
\end{proof}
The Lie subalgebra $\autex(E)$ of $\auteq(E)$ in fact integrates to a tame Lie subgroup of $\Auteq(E)$.
\begin{theorem}[\cite{rubio20}]
There is a tame Lie subgroup of $\Auteq(E)$, the group of \emph{exact Courant autoequivalences} $\Autex(E)$ such that $\operatorname{Lie}(\Autex(E))=\autex(E)$.     
\end{theorem}
\begin{proof}
    In the setup of \cite{rubio20}, the Lie algebra $\autex(E)$ corresponds to the Lie (sub-)algebra $$\left\{(X,b)\in\X(M)\oplus\Omega^2(M):\iota_X H-b=\dd\alpha\text{ for some }\alpha\in\Omega^1(M)\right\}$$
    with Lie bracket
    \[[(X_1,b_1),(X_2,b_2)]=([X_2,X_1],\mathcal{L}_{X_2}b_1-\mathcal{L}_{X_1}b_2).\]
    It was proved in~\cite{rubio20} that this Lie algebra integrates to a strong ILH Lie subgroup of the autoequivalences. In particular, this strong ILH Lie subgroup is a Fréchet submanifold, which under conjugation with $\Phi_H$ (cf.\ Remark~\ref{rk: from rubio-tipler to our tame structure}) is mapped to a Fréchet submanifold $\Autex(E)$ of the tame Fréchet manifold $\Auteq(E)$. Since $\Autex(E)$ is of finite codimension it follows from~\cite[Lem 3.5]{freyn15} that it is a tame Fréchet submanifold and hence a tame Lie subgroup. 
\end{proof}

In the following, it will be useful for us to identify $\autex(E)$ with a tame subspace of $\Gamma(M,E)$.

\begin{proposition}\label{prop: hex isom to ad perp exact case}
    There is a tame isomorphism
    \begin{align*}
    \autex(E)\cong \{X+\alpha\in \Gamma(M,E):\alpha\in \dd^*\Omega^2(M)\}\subset \Gamma(M,E),
\end{align*}
    where $\dd^*$ is the formal adjoint of $\dd$ with respect to the $L^2$-inner product induced by the auxiliary metric $h$.
\end{proposition}
\begin{proof}
    We may think of $\autex(E)$ as the image of the tame linear map (see Section~\ref{appendix HNM} of the appendix for the fact that $\QdR$ is tame linear)
    \begin{align*}
        \tilde{q}: \Gamma(M,E)&\longrightarrow \auteq(E),\\
        X+\alpha &\longmapsto X+\dd \alpha +\QdR\mathcal{L}_X H-\iota_X H.
    \end{align*}
    We observe that $\ker \tilde{q}=\Omega^1_{\mathrm{cl}}(M)\subset \Gamma(M,E)$. Therefore, we have $$\im \tilde{q}\cong (\ker \tilde{q})^{\perp}=\X(M)\oplus \dd^*\Omega^2(M),$$ where the orthogonal complement is taken with respect to the $L^2$-inner product defined by the bundle metric $h\oplus h^*$ on $E= TM\oplus T^*M$. To see that the inverse map is tame, note that it is given by 
    \begin{align*}
        \im \tilde{q}\ni(X,b)&\longmapsto \left(X,\, \QdR(b-\QdR \mathcal{L}_XH+\iota_X H)\right)\in\X(M)\oplus \dd^*\Omega^2(M).
    \end{align*}
\end{proof}

In a neighborhood $\mathcal{V}$ of $\id \in \Autex(E)$, we consider a chart $\Phi$ taking values in some neighborhood $\mathcal{U}\subset\autex(E)$ around $0$. We may assume that the differential of this chart at the identity is the identity map on $\autex(E)$, $D\Phi(\id)=\id$. We denote by $F_v:=\Phi^{-1}(v)$ the exact autoequivalence in $\mathcal{V}$ corresponding to $v\in\mathcal{U}$ with respect to this chart. 

In the following we will see that the Lie algebra of the exact autoequivalences can in fact be identified with the Lie algebra of inner derivations. First, we consider the natural action of autoequivalences on sections $u\in\Gamma(M,E)$, given by
\begin{align*}
   \mu_E\colon(F,u)\longmapsto F\circ u\circ f^{-1},
\end{align*}
where $f$ is the diffeomorphism covered by $F$. 
\begin{proposition}
    The map $\mu_E\colon \Auteq(E)\times \Gamma(M,E)\rightarrow \Gamma(M,E)$ is smooth tame. 
\end{proposition}
\begin{proof}
    Let $F=(f,b)\in \Diff_{[H]}(M)\times \Omega^2_{\mathrm{cl}}(M)$, $u=X+\alpha\in \Gamma(M,E)$. We compute
    \begin{equation}\label{eq: muE}
    \begin{aligned}
        \mu_E(F,u)&=F_f\circ (X+\alpha+\iota_X(b+\QdR(H-f^*H)))\circ f^{-1}\\
        &=\dd f\circ X\circ f^{-1}+(\dd f^{-1})^*(\alpha+\iota_X (b+\QdR(H-f^*H)))\circ f^{-1}\\
        &=f_* X+{f^{-1}}^*(\alpha+\iota_X( b+\QdR(H-f^*H))).
    \end{aligned}
    \end{equation}
    This shows that $\mu_E$ is a composition of smooth tame maps and therefore smooth tame (see~\cite[Sec~3]{paperoncomplexdefs} and Proposition~\ref{prop: pullback} for the fact that the pushforward of a vector field and the pullback of a form are smooth tame). 
\end{proof}
For exact autoequivalences we find the following:
\begin{proposition}
\label{prop: mu for small exact autoeqs exact case}
Let $v\in \autex(E)$, $u\in\Gamma(M,E)$. Then we have 
\begin{align*}
    D_F\mu_E(F=\id,u)v=-[v,u].
\end{align*}
In particular, for some small exact autoequivalence $\Phi_v$, we have
\begin{align*}
    \mu_E(\Phi_v,u)=u-[v,u]+R_{\mu_E}(v)u,
\end{align*}
where $R_{\mu_E}(tv)=t^2\tilde{R}_{\mu_E}(v,t)$, for a real parameter $t$, where $\tilde{R}_{\mu_E}(v,t)$ depends smoothly on $t$ for small $t$.
\end{proposition}
\begin{proof}
    First, using equation~\eqref{eq: muE}, we observe that for $(Y,b)\in \auteq(E)$, $u=X+\alpha\in \Gamma(M,E)$, we have
    \begin{align*}
        D_F\mu_E(F=\id,u)(Y,b)=-\mathcal{L}_Y (X+\alpha)+\iota_X b-\iota_X\QdR\mathcal{L}_Y H.
    \end{align*}
    Now, if $b=\dd\beta+\QdR\mathcal{L}_Y H-\iota_YH$, for some $\beta\in \im(\QdR)$, i.e.~in the case that we are taking the derivative in the direction $v=(Y,\beta)\in \autex(E)$, we find
    \begin{align*}
        D_F\mu_E(F=\id,u)v&=-\mathcal{L}_Y(X+\alpha)+\iota_X(\dd\beta+\QdR\mathcal{L}_Y H-\iota_Y H)-\iota_X\QdR\mathcal{L}_Y H\\
        &=-\mathcal{L}_Y(X+\alpha)+\iota_X \dd\beta-\iota_X\iota_Y H\\
        &=-[v,u].
    \end{align*}
    The second statement follows from \cite[Thm I.3.5.6]{hamilton82} (see also~\cite[Prop 3.4]{paperoncomplexdefs}).
\end{proof}
The derivative at the identity of the action $\mu_E$ of autoequivalences on sections of $E$ yields an identification of elements in the Lie algebra $\auteq(E)$ of the autoequivalences with the derivations of $E$.
\begin{corollary}
    Under the identification of $\auteq(E)$ with the derivations of $E$, we have $\autex(E)=\im (-\ad:\Gamma(M,E)\rightarrow \operatorname{Der}(E))$, where $\ad(u):=[u,\cdot]$. 
\end{corollary}
Finally, we observe that the map $\tilde{q}$ considered in the proof of Proposition~\ref{prop: hex isom to ad perp exact case} is in fact the map $-\ad\colon \Gamma(E)\longrightarrow \Der(E)\cong \auteq(E)$ and hence what we constructed in Proposition~\ref{prop: hex isom to ad perp exact case} is a tame isomorphism $\auteq(E)\cong (\ker \ad)^{\perp}$, where the orthogonal complement is taken with respect to the $L^2$-inner product defined by the bundle metric $h\oplus h^*$.

\subsection{Autoequivalences of $\mathcal{G}$-flat transitive Courant algebroids} \label{sect: auteq G-flat transitive}

We will now focus on $\mathcal{G}$-flat transitive Courant algebroids in the sense of Definition~\ref{def: g-flat CA}. First, we fix the following $L^2$-inner products. We choose a positive definite bundle metric $\langle\cdot,\cdot\rangle_{\mathrm{Eucl}}$ on $\mathcal{G}$ and denote the $L^2$-inner products on $\Gamma(M,\mathcal{G})$ and $\Omega^k(M,\mathcal{G})$ defined by $\langle\cdot,\cdot\rangle_{\mathrm{Eucl}}$ and the auxiliary Riemannian metric by $(\cdot,\cdot)_{\mathcal{G}}$. Using $\langle\cdot,\cdot\rangle_{\mathrm{Eucl}}$ and the auxiliary Riemannian metric $h$, we fix a (positive definite) bundle metric on $\Gamma(E)= TM\oplus \mathcal{G}\oplus T^*M$ via $h\oplus \langle\cdot,\cdot\rangle_{\mathrm{Eucl}}\oplus h^*$ and denote the $L^2$-inner product on sections of $E$ defined by this bundle metric and $h$ by $(\cdot,\cdot)_E$.

Note that when $\mathcal{G}\cong M\times \R^K$, after choosing a global orthonormal frame $\{\mathbf{e}_i\}_{i=1}^K$, we may think of sections $\mathbf{r}\in \Gamma(M,\mathcal{G})$ as vectors of smooth functions $\mathbf{r}=(r_i)_{i=1}^K$, $r_i\in C^{\infty}(M)$ and of $p$-forms with values in $\mathcal{G}$ as vectors of $p$-forms on $M$. Given $\mathbf{A}=(A_i)_{i=1}^K\in \Omega^p(M,\mathcal{G})$ and $\mathbf{B}=(B_i)_{i=1}^K\in\Omega^q(M,\mathcal{G})$, we note that \[\langle \mathbf{A}\wedge \mathbf{B}\rangle_{\mathcal{G}}=\sum_{i=1}^K \epsilon_i\,A_i\wedge B_i,\]
where $\epsilon_i=\langle \mathbf{e}_i,\mathbf{e}_i\rangle_{\mathcal{G}}\in\{\pm1\}$. Moreover, we may assume that $\{\mathbf{e}_i\}$ is a global parallel orthonormal frame for the connection $\nabla$ on $\mathcal{G}$, since $\nabla$ has trivial holonomy and $\langle\cdot,\cdot\rangle_{\mathcal{G}}$ is parallel by the compatibility conditions. Then we have $\nabla \mathbf{r}=(\dd r_i)_{i=1}^K=:\dd\mathbf{r}$, and $\dn \mathbf{A}= (\dd A_i)_{i=1}^K=:\dd\mathbf{A}$ for $\mathbf{A}\in \Omega^p(M,\mathcal{G})$.

The differential operator $\dd\colon \Omega^{\bullet}(M,\mathcal{G})\longrightarrow \Omega^{\bullet+1}(M,\mathcal{G})$ defines an elliptic complex. We denote the formal adjoint of $\dd$ with respect to $(\cdot,\cdot)_{\mathcal{G}}$ by $\dd_{\nabla}^*$, its Laplace operator by $\Delta_{\nabla}=\dd\dd_{\nabla}^*+\dd_{\nabla}^*\dd$, the Green's operator by $\G_{\nabla}$ and by $\Q_{\nabla}:=\dd_{\nabla}^*\G_{\nabla}$. Note that the map $\Q_{\nabla}$ is tame linear (see Section~\ref{appendix HNM} of the appendix).

\begin{remark}\label{rk: coexact times parallel are top compl}
    Given a fixed parallel orthonormal frame $\{\mathbf{e}_i\}_{i=1}^K$ for $\mathcal{G}$, we may define a (constant) positive definite bundle metric $\langle\cdot,\cdot\rangle'_{\mathrm{Eucl}}$ on $\mathcal{G}$ by $\langle \mathbf{e}_i,\mathbf{e}_j\rangle'_{\mathrm{Eucl}}=\delta_{ij}$, i.e.\ $\{\mathbf{e}_i\}$ is an orthonormal basis for $\langle\cdot,\cdot\rangle'_{\mathrm{Eucl}}$. Together with the auxiliary Riemannian metric $h$, this bundle metric defines an $L^2$-inner product on $\Gamma(M,\mathcal{G})$ and $\Omega^{\bullet}(M,\mathcal{G})$, which we denote by $(\cdot,\cdot )'_{\mathcal{G}}$. The formal adjoint of $\dd\colon \Omega^k(M,\mathcal{G})\longrightarrow \Omega^{k+1}(M,\mathcal{G})$ with respect to $(\cdot,\cdot )'_{\mathcal{G}}$ is simply given by
$\dd_{\nabla}^*\mathbf{A}=\sum_{i=1}^K(\dd^*A_i)\mathbf{e}_i$ for $\mathbf{A}=\sum_{i=1}^KA_i\mathbf{e}_i$, where $\dd^*$ is the formal adjoint of $\dd$ with respect to the $L^2$-inner product induced by $h$. In particular, the space $\{\mathbf{A}=\sum_{i=1}^KA_i\mathbf{e}_i\in\Omega^k(M,\mathcal{G}): A_i\in\dd^*\Omega^{k+1}(M)\}$ is a topological complement of $\Omega_{\mathrm{cl}}^k(M,\mathcal{G})=\{\mathbf{A}\in\Omega^k(M,\mathcal{G}):
\dd \mathbf{A}=0\}$. 
\end{remark}

Finally, we note that if $\langle\cdot,\cdot\rangle_{\mathcal{G}}$ is of signature $(q,K-q)$, the autoequivalences of $\mathcal{G}$ are given by $\Diff(M)\times C^{\infty}(M,\O(q,K-q))$ acting on $(p,v)\in\mathcal{G}=M\times \R^K$ as \[(f,O)(p,v)=\left(f(p),O(p)v\right).\]
These act on sections $\mathbf{r}\in\Gamma(M,\mathcal{G})$ via \[(f,O)\cdot \mathbf{r}=(O\mathbf{r})\circ f^{-1},\] where $(Or)_i(p)=\sum_j O_{ij}(p)r_j(p)$.

 For $f\in \Diff(M)$ and $\mathbf{A}\in\Omega^k(M,\mathcal{G})$, we will denote by $f^*\mathbf{A}:=(f^*A_i)_{i=1}^K$. We observe that
 \begin{align}\label{eq: d and pullback commute on forms with values in G}
     \dd f^*\mathbf{A}=\sum_{i=1}^K\dd f^*A_i\,\mathbf{e}_i=\sum_{i=1}^K f^*\dd A_i\,\mathbf{e}_i=f^*\dd\mathbf{A}.
 \end{align}
 
\subsubsection{Autoequivalences of $\mathcal{G}$-flat standard transitive Courant algebroids}

In this section, we will see that the autoequivalences of a $\mathcal{G}$-flat transitive Courant algebroid form a tame Fr\'echet Lie group. First, we show the following. 
\begin{proposition}
\label{prop: flat autos are tame manifold}
    Let $E=TM\oplus\mathcal{G}\oplus T^*M$ be a $\mathcal{G}$-flat standard transitive Courant algebroid with $\mathcal{G}\cong M\times \R^K$ and $\langle\cdot,\cdot\rangle_{\mathcal{G}}$ of signature $(q,K-q)$. Then the group of autoequivalences $\Auteq(E)$ has a tame Fréchet manifold structure as a product
    \[\Auteq(E)=\Auteq(\mathbb{T}M,H)\times \O(q,K-q)\times \Omega^1_{\mathrm{cl}}(M,\mathcal{G}),\]
    where $\Auteq(\mathbb{T}M,H)$ denotes the autoequivalences of the $H$-twisted generalized tangent bundle and $\Omega^1_{\mathrm{cl}}(M,\mathcal{G})$ denotes all $\mathbf{A}\in\Omega^1(M,\mathcal{G})$ such that $\dd \mathbf{A}=0$. 
\end{proposition}
\begin{proof}
    Using Theorem \ref{thm: gen automs trans}, the fact that $[\cdot,\cdot]_{\mathcal{G}}\equiv 0$, $\mathbf{R}=0$ and that $\nabla$ is flat, we know that
    \begin{align*}
        \Auteq(E)= \{F_{\nu,\mathbf{A},B}:&\,\nu\in\Auteq(\mathcal{G})\text{ covers }f\in\mathrm{Diff}(M),\;\mathbf{A}\in\Omega^1(M,\mathcal{G}),\;B\in\Omega^2(M),\\
        &\nabla=\nu\cdot \nabla,\ 0=\nu\cdot \dd \mathbf{A},\ H-f^* H=\dd B+\langle \mathbf{A}\wedge \dd \mathbf{A}\rangle_{\mathcal{G}}\}.
    \end{align*}
    Note that the second equation implies that $\dd \mathbf{A}=0$. Therefore, the last equation becomes $H-f^*H=\dd B$ and hence $F_{f,B}\in \Auteq(\mathbb{T}M,H)$. Finally, let $O\in C^{\infty}(M,\O(q,K-q))$ such that $\nu=(f,O)$, $X\in \X(M)$ and $\mathbf{r}\in \Gamma(M,\mathcal{G})$. We compute
    \begin{align*}
        (\nu\cdot \nabla)_X\mathbf{r}&=\nu\cdot \left(\nabla_{f_*^{-1}X}(\nu^{-1}\cdot \mathbf{r})\right)\\
        &=\nu\cdot \dd(O^{-1}\mathbf{r}\circ f)(f_*^{-1}X)\\
        &=\nu\cdot (\dd(O^{-1}\mathbf{r})X)\circ f\\
        &=O(\dd(O^{-1}\mathbf{r})X)\\
        &=\iota_X\left[(OdO^{-1})\mathbf{r}+\dd \mathbf{r}\right].
    \end{align*}
    Hence, the condition $\nu\cdot \nabla=\nabla$ becomes $OdO^{-1}=0$, which implies that $dO=0$ and hence $O$ must be constant. This implies that as a manifold
    \[\Auteq(E)=\Auteq(\mathbb{T}M,H)\times \O(q,K-q)\times \Omega^1_{\mathrm{cl}}(M,\mathcal{G}).\]
    This finishes the proof.
\end{proof}
Recalling the bijection $\Auteq(\mathbb{T}M,H)\cong \Diff(M)_{[H]}\times \Omega^2_{\mathrm{cl}}(M)$ via 
$$
\Diff(M)_{[H]}\times \Omega^2_{\mathrm{cl}}(M)\ni (f,b)\longmapsto F_{f,b+B(f)},
$$ where $B(f)=\QdR(H-f^*H)$, which we used in order to define the manifold structure on $\Auteq(\mathbb{T}M,H)$, we find the following:
\begin{theorem}
\label{thm: Auteq of flat std trans is tame Lie group}
Let $E=TM\oplus\mathcal{G}\oplus T^*M$ be a $\mathcal{G}$-flat standard transitive Courant algebroid with $\mathcal{G}\cong M\times \R^K$ and three-form $H\in\Omega^3_{\mathrm{cl}}(M)$. Then $\Auteq(E)$ is a tame Fréchet Lie group with underlying manifold structure
$$ \Auteq(E)=\Diff(M)_{[H]}\times\Omega^2_{\mathrm{cl}}(M)\times \O(q,K-q)\times \Omega^1_{\mathrm{cl}}(M,\mathcal{G}).$$
The multiplication map is given by
\[(f_1,b_1,O_1,\mathbf{A}_1)\times (f_2,b_2,O_2,\mathbf{A}_2)\longmapsto (f_1\circ f_2,b_3,O_1O_2,\mathbf{A}_3),\]
where
\begin{align*}
    b_3&=f_2^*b_1+b_2+\langle O_2^{-1}f_2^*\mathbf{A}_1\wedge \mathbf{A}_2\rangle_{\mathcal{G}}+[f_2^*,\QdR](H-f_1^*H),\\
    \mathbf{A}_3&=O_2^{-1}f_2^*\mathbf{A}_1+\mathbf{A}_2,
\end{align*}
and the inversion map is given by
\begin{align*}
    (f,b,O,\mathbf{A})\longmapsto (f^{-1},-(f^{-1})^*b+[\QdR,(f^{-1})^*](H-f^*H),O^{-1},-O(f^{-1})^*\mathbf{A}).
\end{align*}
\end{theorem}
\begin{proof}
    It follows from the corresponding expressions in Theorem \ref{thm: gen automs trans} that multiplication on $\Auteq(E)=\Auteq(\mathbb{T}M,H)\times \O(q,K-q)\times \Omega^1_{\mathrm{cl}}(M,\mathcal{G})$ is given by
    \[\left(F_{f_1,B_1},O_1,\mathbf{A}_1\right)\circ \left(F_{f_2,B_2},O_2,\mathbf{A}_2\right)=\left(F_{f_3,B_3},O_3,\mathbf{A}_3\right),\]
    where $f_3=f_1\circ f_2$, $O_3=O_1O_2$ and
    \begin{align*}
        B_3&= f_2^*B_1+B_2+\langle O_2^{-1}(f_2^*\mathbf{A}_1)\wedge \mathbf{A}_2\rangle_{\mathcal{G}},\\
        \mathbf{A}_3&=O_2^{-1} f_2^*\mathbf{A}_1+\mathbf{A}_2
    \end{align*}
    and the identity element is $(\id,\id,0)$. The inversion map is given by
    \begin{align*}
    \left(F_{f,B},O,\mathbf{A}\right)\longmapsto\left(F_{f,B}^{-1},O^{-1},-O(f^{-1})^*\mathbf{A}\right). 
    \end{align*} 
    We only need to check the formulas for the multiplication and inversion map under $\Auteq(\mathbb{T}M,H)\cong \Diff(M)_{[H]}\times \Omega^2_{\mathrm{cl}}(M)$. In each case, the only non-trivial part is the component in $\Omega^2_{\mathrm{cl}}(M)$. For the multiplication, we find
    \begin{align*}
        b_3&=f_2^*\left(b_1+B(f_1)\right)+b_2+B(f_2)+\langle O_2^{-1}f_2^*\mathbf{A}_1\wedge \mathbf{A}_2\rangle_{\mathcal{G}}-B(f_1\circ f_2)\\
        &=f_2^*b_1+b_2+\langle O_2^{-1}f_2^*\mathbf{A}_1\wedge \mathbf{A}_2\rangle_{\mathcal{G}}+f_2^*\QdR(H-f_1^*H)\\&\quad+\QdR(H-f_2^*H)-\QdR(H-f_2^*f_1^*H)\\
        &=f_2^*b_1+b_2+\langle O_2^{-1}f_2^*\mathbf{A}_1\wedge \mathbf{A}_2\rangle_{\mathcal{G}}-[\QdR,f_2^*](H-f_1^*H).
    \end{align*}
    Similarly, we find that the $\Omega^2_{\mathrm{cl}}(M)$-component of the inverse of $(f,b,O,\mathbf{A})$ is given by
    \begin{align*}
        &-(f^{-1})^*\left(b+B(f)\right)-B(f^{-1})\\
        &=-(f^{-1})^*b-(f^{-1})^*\QdR(H-f^*H)-\QdR(H-(f^{-1})^*H)\\
        &=-(f^{-1})^*b+[\QdR,(f^{-1})^*](H-f^*H).
    \end{align*}
    Both the multiplication and the inversion map are compositions of smooth tame maps and hence smooth tame.
\end{proof}

\begin{remark}
    Note that for (not necessarily $\mathcal{G}$-flat) transitive Courant algebroids with $\operatorname{rank} \mathcal{G}=1$ a strong ILH Lie group structure of $\Auteq(E)$ has been constructed in~\cite{rubio20}.
\end{remark}

Using the atlas of $\Diff(M)_{[H]}$ in terms of the Riemannian exponential map of the auxiliary metric $h$, we find that in a neighborhood of the identity the autoequivalences of a $\mathcal{G}$-flat standard transitive Courant algebroid are modeled on the tame Fréchet space $\X(M)\oplus\Omega^2_{\mathrm{cl}}(M)\oplus \mathfrak{so}(q,K-q)\oplus \Omega^1_{\mathrm{cl}}(M,\mathcal{G})$.

\begin{proposition}
    The Lie algebra of $\Auteq(E)$, where $E$ is the $\mathcal{G}$-flat standard transitive Courant algebroid with $\mathcal{G}\cong M\times \R^K$ and three-form $H$, has Lie algebra $$\auteq(E)=\X(M)\oplus\Omega^2_{\mathrm{cl}}(M)\oplus \mathfrak{so}(q,K-q)\oplus \Omega^1_{\mathrm{cl}}(M,\mathcal{G})$$ with Lie bracket 
    \begin{align*}
        &[X_1+b_1+T_1+\mathbf{A}_1,\ X_2+b_2+T_2+\mathbf{A}_2]\\&\ =\mathcal{L}_{X_2}X_1+\mathcal{L}_{X_2}b_1-\mathcal{L}_{X_1}b_2+2\langle \mathbf{A}_1\wedge \mathbf{A}_2\rangle_{\mathcal{G}}+(\mathcal{L}_{X_1}\QdR \mathcal{L}_{X_2}-\mathcal{L}_{X_2}\QdR\mathcal{L}_{X_1})H\\
        &\quad-\QdR \mathcal{L}_{\mathcal{L}_{X _1}X_2}H +[T_1,T_2]-T_2 \mathbf{A}_1+\mathcal{L}_{X_2}\mathbf{A}_1+T_1 \mathbf{A}_2-\mathcal{L}_{X_1}\mathbf{A}_2.
        \end{align*}
\end{proposition}
\begin{proof}
    First, given $(f_1,\tilde{b}_1,O_1,\mathbf{\tilde{A}}_1),(f_2,\tilde{b}_2,O_2,\mathbf{\tilde{A}}_2)\in\Auteq(E)$, we determine $b_3\in\Omega_{\mathrm{cl}}^2(M)$  and $\mathbf{A}_3\in\Omega^1_{\mathrm{cl}}(M,\mathcal{G})$ such that
    \begin{align*}
        &(f_1,\tilde{b}_1,O_1,\mathbf{\tilde{A}}_1)(f_2,\tilde{b}_2,O_2,\mathbf{\tilde{A}}_2)(f_1,\tilde{b}_1,O_1,\mathbf{\tilde{A}}_1)^{-1}=(f_1f_2f_1^{-1},b_3,O_1O_2O_1^{-1},\mathbf{A}_3).
        \end{align*}
        For this, we first compute
        \begin{align*}
        &(f_2,\tilde{b}_2,O_2,\mathbf{\tilde{A}}_2)(f_1,\tilde{b}_1,O_1,\mathbf{\tilde{A}}_1)^{-1}\\&=
        (f_2,\tilde{b}_2,O_2,\mathbf{\tilde{A}}_2)(f_1^{-1},-(f_1^{-1})^*\tilde{b}_1+[\QdR,(f_1^{-1})^*](H-f_1^*H),O_1^{-1},\,-O_1(f_1^{-1})^*\mathbf{\tilde{A}}_1)\\
        &=(f_2f_1^{-1},\,(f_1^{-1})^*\tilde{b}_2-(f_1^{-1})^*\tilde{b}_1+[\QdR,(f_1^{-1})^*](H-f_1^*H)\\&\quad-\langle O_1(f_1^{-1})^*\mathbf{\tilde{A}}_2\wedge O_1(f_1^{-1})^*\mathbf{\tilde{A}}_1\rangle_{\mathcal{G}}  -[\QdR,(f_1^{-1})^*](H-f_2^*H),\,O_2O_1^{-1},O_1(f_1^{-1})^*(\mathbf{\tilde{A}}_2-\mathbf{\tilde{A}}_1)).
    \end{align*}
    Multiplying from the left by $(f_1,\tilde{b}_1,O_1,\mathbf{\tilde{A}}_1)$ yields
    \begin{align*}
        b_3&=(f_2f_1^{-1})^*\tilde{b}_1+(f_1^{-1})^*(\tilde{b}_2-\tilde{b}_1)-[\QdR,(f_1^{-1})^*](f_1^*H-f_2^*H)\\&\quad-[\QdR,(f_2f_1^{-1})^*](H-f_1^*H)+\langle O_1O_2^{-1}(f_2f_1^{-1})^*\mathbf{\tilde{A}}_1\wedge O_1(f_1^{-1})^*(\mathbf{\tilde{A}}_2-\mathbf{\tilde{A}}_1)\rangle_{\mathcal{G}}\\
        \mathbf{A}_3&=O_1O_2^{-1}(f_2f_1^{-1})^*\mathbf{\tilde{A}}_1+O_1(f_1^{-1})^*(\mathbf{\tilde{A}}_2-\mathbf{\tilde{A}}_1).
    \end{align*}
    Here, we used $\langle f^*\mathbf{A}\wedge f^*\mathbf{A'}\rangle_{\mathcal{G}}=f^*\langle\mathbf{A}\wedge \mathbf{A'}\rangle_{\mathcal{G}}$ and $\langle O\mathbf{A}\wedge O\mathbf{A'}\rangle_{\mathcal{G}}=\langle \mathbf{A}\wedge \mathbf{A'}\rangle_{\mathcal{G}}$. Setting $f_1(t)=\exp^h(t X_1)$, $\tilde{b}_1(t)=tb_1$, $\mathbf{\tilde{A}}_1(t)=t\mathbf{A}_1$ and $O_1(t)=\exp(t T_1)$ and similarly $f_2(s)=\exp^h(s X_2)$, $\tilde{b}_2(s)=sb_2$, $\mathbf{\tilde{A}}_2(s)=s\mathbf{A}_2$ and $O_2(s)=\exp(s T_2)$ and taking the derivative $\left.\frac{\dd}{\dd s}\right|_{s=0}\left.\frac{\dd}{\dd t}\right|_{t=0}$ yields the expression for the Lie bracket.
\end{proof}
Next, we identify a tame Lie subgroup of $\Auteq(E)$, which generalizes the group of exact autoequivalences of an exact Courant algebroid.
\begin{theorem}
    The tame submanifold $\Autex(\mathbb{T}M,H)\times \Omega^1_{\mathrm{ex}}(M,\mathcal{G})\subset \Auteq(E)$ inherits the structure of a tame Lie subgroup of $\Auteq(E)$, the tame Lie subgroup $\Autex(E)$ of \emph{exact autoequivalences} of $E$. 
\end{theorem}
\begin{proof}
    Let $(F_1,\mathbf{A}_1),(F_2,\mathbf{A}_2)\in \Autex(\mathbb{T}M,H)\times \Omega^1_{\mathrm{ex}}(M,\mathcal{G})$. Then applying the restriction of the multiplication map in $\Auteq(E)$ described in Theorem \ref{thm: Auteq of flat std trans is tame Lie group} yields
    \begin{align*}
       (F_1,\mathbf{A}_1)\times (F_2,\mathbf{A}_2)&\longmapsto \left(F_{\langle f_2^*\mathbf{A}_1\wedge \mathbf{A}_2\rangle_{\mathcal{G}}}F_1F_2,f_2^*\mathbf{A}_1+\mathbf{A}_2\right)\in \Auteq(\mathbb{T}M,H)\times \Omega^1_{\mathrm{cl}}(M,\mathcal{G}), 
    \end{align*}
    where $f_2$ is the diffeomorphism covered by $F_2$. Clearly, $f_2^*\mathbf{A}_1+\mathbf{A}_2\in \Omega^1_{\mathrm{ex}}(M,\mathcal{G})$. Moreover, $\langle f_2^*\mathbf{A}_1\wedge \mathbf{A}_2\rangle_{\mathcal{G}}\in \Omega^2_{\mathrm{ex}}(M)$ and hence $F_{\langle f_2^*\mathbf{A}_1\wedge \mathbf{A}_2\rangle_{\mathcal{G}}}$ is an exact $B$-field transformation and in particular an exact autoequivalence of $\mathbb{T}M$. Since $\Autex(\mathbb{T}M,H)$ is closed under multiplication, we conclude that $\Autex(E)$ is closed under multiplication. Since it is closed under inversion too and clearly contains the identity, it follows that it is a subgroup of $\Auteq(E)$. As it is also a tame submanifold it follows that it is a tame Lie subgroup.
\end{proof}
In the following, we describe the Lie subalgebra of the exact autoequivalences, which again can be identified with a tame subspace of $\Gamma(M,E)$.
\begin{proposition}\label{prop: hex in flat std isom to ad perp}
    The tame Lie subalgebra corresponding to $\Autex(E)$ is given by
    \[\autex(E)=\{X+\dd \alpha+\QdR \mathcal{L}_X H-\iota_X H+\mathbf{A}:X\in \X(M),\, \alpha\in \Omega^1(M),\, \mathbf{A}\in \Omega^1_{\mathrm{ex}}(M,\mathcal{G})\}.\]
    It is tamely isomorphic to the tame Fréchet space 
    \begin{align*}
        \autex(E)\cong\X(M)\oplus \dd^*\Omega^2(M)\oplus \dd_{\nabla}^*\Omega^1(M,\mathcal{G})\subset \Gamma(M,E).
    \end{align*}
\end{proposition}
\begin{proof}
    Recall that the tangent space at the identity of $\Autex(\mathbb{T}M,H)$ is given by 
    \begin{align*}
        T_{\id}\Autex(\mathbb{T}M,H)&=\{X+\dd \alpha+\QdR \mathcal{L}_X H-\iota_X H:X\in \X(M),\, \alpha\in \Omega^1(M)\}.
    \end{align*}
    It follows that 
    \begin{align*}
        T_{\id}\Autex(E)&=\{X+\dd \alpha+\QdR \mathcal{L}_X H-\iota_X H+\dd \mathbf{r}:X\in \X(M),\ \alpha\in \Omega^1(M),\ \mathbf{r}\in \Omega^0(M,\mathcal{G})\}.
    \end{align*}
    We may therefore think of $\autex(E)=T_{\id}\Autex(E)$ as the image of the tame linear map
    \begin{align*}
        p\colon \Gamma(M,E)&\longrightarrow \auteq(E)\\
        X+\mathbf{r}+\alpha&\longmapsto X+\dd \alpha+\QdR \mathcal{L}_X H-\iota_X H+\dd \mathbf{r}.
    \end{align*}
    We observe that $\ker p=\Omega_{\mathrm{cl}}^1(M)\oplus \Omega^0_{\mathrm{cl}}(M,\mathcal{G})$. It follows from the Hodge decomposition that with respect to the $L^2$-inner product $(\cdot,\cdot)_E$, the orthogonal complement of the kernel of $p$ is given by $(\ker p)^{\perp}=\X(M)\oplus\dd^*\Omega^2(M)\oplus \dd_{\nabla}^*\Omega^1(M,\mathcal{G})$ and hence
    \begin{align*}
        \autex(E)=\im p\cong (\ker p)^{\perp}=\X(M)\oplus\dd^*\Omega^2(M)\oplus \dd_{\nabla}^*\Omega^1(M,\mathcal{G}).
    \end{align*}
    To see that this is a tame isomorphism, note that its inverse map  is given by
    \begin{align*}
        \autex(E)\ni X+b+\mathbf{A}\longmapsto X+\QdR (b-(\QdR \mathcal{L}_X H-\iota_X H))+ \Q_{\nabla} \mathbf{A}.    \end{align*}
\end{proof}
Like in the case of exact Courant algebroids, we consider a local chart neighborhood $\mathcal{V}$ of $\id \in \Autex(E)$, with a chart $\Phi$ taking values in some neighborhood $\mathcal{U}\subset\autex(E)$ around $0$ and assume that $D\Phi(\id)=\id$. We will again denote by $F_v:=\Phi^{-1}(v)$ the exact autoequivalence in $\mathcal{V}$ corresponding to $v\in\mathcal{U}$. 

Like for exact Courant algebroids, the Lie algebra of exact autoequivalences can be identified with the Lie algebra of inner derivations. First, we consider the natural action of autoequivalences on sections $u\in\Gamma(M,E)$, which we denote again by $\mu_E$.
\begin{proposition}
\label{prop: action of ex auteq is smooth tame flat trans}
    The map $\mu_E\colon \Auteq(E)\times \Gamma(M,E)\rightarrow \Gamma(M,E)$ is smooth tame. 
\end{proposition}
\begin{proof}
    Given $(f,b,O,\mathbf{A})\in \Auteq(E)$ and $u=X+\mathbf{r}+\alpha\in \Gamma(M,E)$, we compute
    \begin{align*}
        &(f,b,O,\mathbf{A})\circ u\circ f^{-1}\\
        &=F_{f,O}\circ F_A\circ F_{b+B(f)}(X+\mathbf{r}+\alpha)\circ f^{-1}\\
        &=F_{f,O}\circ F_{\mathbf{A}}\left(X+\mathbf{r}+\alpha+\iota_Xb+\iota_X \QdR(H-f^*H)\right)\circ f^{-1}\\
        &=F_{f,O}\circ \left(X+\mathbf{A}(X)+\mathbf{r}+\alpha+\iota_X b+\iota_X\QdR(H-f^*H)-\iota_X\langle \mathbf{A},\mathbf{A}\rangle_{\mathcal{G}}-2\langle \mathbf{A},\mathbf{r}\rangle_{\mathcal{G}}\right)\circ f^{-1}\\
        &=f_*X+O(\mathbf{A}(X)+\mathbf{r})\circ f^{-1}+(f^{-1})^*\left(\alpha+\iota_Xb+\iota_X\QdR(H-f^*H)-\iota_X\langle \mathbf{A},\mathbf{A}\rangle_{\mathcal{G}}-2\langle \mathbf{A},\mathbf{r}\rangle_{\mathcal{G}}\right),
    \end{align*}
    which is smooth tame as a composition of smooth tame maps.
\end{proof}
\begin{proposition}
\label{prop: mu for small exact autoeqs}
Let $v\in \autex(E)$, $u\in\Gamma(M,E)$. Then we have 
\begin{align*}
    D_F\mu_E(F=\id,u)v=-[v,u].
\end{align*}
In particular, for some small exact autoequivalence $\Phi_v$, we have
\begin{align*}
    \mu_E(\Phi_v,u)=u-[v,u]+R_{\mu_E}(v)u,
\end{align*}
where $R_{\mu_E}(tv)=t^2\tilde{R}_{\mu_E}(v,t)$, for a real parameter $t$, where $\tilde{R}_{\mu_E}(v,t)$ depends smoothly on $t$ for small $t$.
\end{proposition}
\begin{proof}
    First, using the formula from the proof of Proposition~\ref{prop: action of ex auteq is smooth tame flat trans}, for $(Y+b+T+\mathbf{A})\in \auteq(E)$ and $u=X+\mathbf{r}+\alpha$, we find
    \begin{align*}
        &D_F\mu_E(F=\id,u)\{Y+b+T+\mathbf{A}\}\\&=-\mathcal{L}_YX+\mathbf{A}(X)+T\mathbf{r}-\mathcal{L}_Y\mathbf{r}-\mathcal{L}_Y\alpha+\iota_Xb-\iota_X\QdR\mathcal{L}_YH-2\langle \mathbf{A},\mathbf{r}\rangle_{\mathcal{G}}.
    \end{align*}
    Setting $T=0$, $\mathbf{A}=\dd \mathbf{s}$ for some $\mathbf{s}\in \Omega^0(M,\mathcal{G})$, $b=\dd \beta+\QdR \mathcal{L}_Y H-\iota_YH$, and using $\mathcal{L}_Y \mathbf{r}=\dd \mathbf{r}(Y)$, we find
    \begin{align*}
        &D_F\mu_E(F=\id,u)\{Y+\dd \beta+\QdR \mathcal{L}_Y H-\iota_YH+\dd \mathbf{s}\}\\
        &=-\mathcal{L}_YX+\dd \mathbf{s}(X)-\dd \mathbf{r}(Y)-\mathcal{L}_Y \alpha+\iota_X\dd \beta -\iota_X\iota_Y H-2\langle \dd \mathbf{s},\mathbf{r}\rangle_{\mathcal{G}}\\
        &=-[Y+\mathbf{s}+\beta,X+\mathbf{r}+\alpha].
    \end{align*}
\end{proof}
Recall again that the derivative at the identity of $\mu_E$ gives rise to an identification of $\auteq(E)$ with the derivations of $E$.
\begin{corollary}
    Under the identification $\Der(E)\cong\auteq(E)$, we have $\autex(E)=\im (-\ad)$.
\end{corollary}

\begin{remark}\label{rk: hex id with orthog complement of ad for std flat}
Analogous to the exact case we observe that Proposition~\ref{prop: hex in flat std isom to ad perp} in fact states that $\autex(E)$ is tamely isomorphic to $(\ker \ad)^{\perp}$, where the orthogonal complement is taken with respect to the $L^2$-inner product $(\cdot,\cdot)_E$ on $\Gamma(M,E)$ defined by the bundle metric $h\oplus \langle\cdot,\cdot\rangle_{\mathrm{Eucl}}\oplus h^*$.\end{remark}

The tame Fréchet Lie group structure of the standard $\mathcal{G}$-flat transitive Courant algebroid gives rise to a tame Fréchet Lie group structure on a general $\mathcal{G}$-flat transitive Courant algebroid. Given a $\mathcal{G}$-flat transitive Courant algebroid, there is no unique standard $\mathcal{G}$-flat transitive Courant algebroid it is isomorphic to: A $B$-field transformation only changes $H\to H+\dd B$ and hence maps standard $\mathcal{G}$-flat transitive Courant algebroids to standard $\mathcal{G}$-flat transitive Courant algebroids. Similarly, elements in $\Aut(\mathcal{G})= C^{\infty}(M,\O(q,K-q))$ are simply changes of trivializations and again yield a standard $\mathcal{G}$-flat transitive Courant algebroid.

In contrast, when acting with $\mathbf{A}\in \Omega^1(M,\mathcal{G})$, in general $\mathbf{R}$ no longer vanishes but rather $\mathbf{R}=-\dd \mathbf{A}$. Later, we will require more flexibility in the choice of bisection and need to also consider those bisections in which $\mathbf{R}$ does not vanish. Therefore, we will be interested in an explicit description of the group of autoequivalences of a standard transitive Courant algebroid obtained from a standard $\mathcal{G}$-flat Courant algebroid by an $\mathbf{A}$-field transformation.

\subsubsection{Autoequivalences in a general choice of bisection}
Suppose we change the bisection from a flat standard transitive Courant algebroid by acting with $-\mathbf{S}\in \Omega^1(M,\mathcal{G})$. Since the action by a closed $\mathbf{A}$-field transform on a standard $\mathcal{G}$-flat transitive Courant algebroid will again yield a standard $\mathcal{G}$-flat transitive Courant algebroid, we will assume without loss of generality that $\mathbf{S}=\sum_{i=1}^KS_i\mathbf{e}_i$, with $S_i\in\dd^*\Omega^2(M)$ (recall that by Remark~\ref{rk: coexact times parallel are top compl} the $\mathbf{S}$ of this kind form a topological complement to $\Omega^1_{\mathrm{cl}}(M,\mathcal{G})$).

Then the data defining the bracket is as follows: the connection on $\mathcal{G}$ is still flat with trivial holonomy, $\mathbf{R}=\dd \mathbf{S}$ and the three-form $H$ satisfies $H-\langle \mathbf{S}\wedge \dd \mathbf{S}\rangle_{\mathcal{G}}\in \Omega^3_{\mathrm{cl}}(M)$. The group of autoequivalences of this Courant algebroid is the subgroup of the group of orthogonal equivalences conjugate to the autoequivalences of the standard $\mathcal{G}$-flat transitive Courant algebroid by the action of $F_{-\mathbf{S}}$. In particular, it will still have a tame Fréchet Lie group structure, which we explicitly describe in the following theorem. 

\begin{theorem}\label{thm: auteq for R=dS}
    Let $E\cong TM\oplus \mathcal{G}\oplus T^*M$ be a standard transitive Courant algebroid with $\mathcal{G}=M\times \R^K$, $\nabla$ a flat connection on $\mathcal{G}$ with trivial holonomy, $\mathbf{R}=\dd \mathbf{S}$ for some $\mathbf{S}=\sum_{i=1}^K S_i\mathbf{e}_i$, $S_i\in\dd^*\Omega^2(M)$ and $H\in \Omega^3(M)$ such that $H-\langle \mathbf{S}\wedge \dd \mathbf{S}\rangle_{\mathcal{G}}\in \Omega^3_{\mathrm{cl}}(M)$. Then
    \begin{align*}
        \Auteq(E)&=\{F_{\nu,\mathbf{A},B}:\, \nu=(f,O)\text{ s.t. } f\in \Diff_{[H-\langle \mathbf{S}\wedge \dd \mathbf{S}\rangle_{\mathcal{G}}]}(M),\, O\in \O(q,K-q),\\
        &\quad\quad \mathbf{a}:=\mathbf{A}-\mathbf{S}+O^{-1}f^*\mathbf{S}\in \Omega^1_{\mathrm{cl}}(M,\mathcal{G}),\\
        &\quad\quad b:=B-\QdR\left(H-\langle \mathbf{S}\wedge \dd \mathbf{S}\rangle_{\mathcal{G}}-f^*(H-\langle \mathbf{S}\wedge \dd \mathbf{S}\rangle_{\mathcal{G}})\right)\\&\quad \quad\quad-\langle \mathbf{a}\wedge (\mathbf{S}+O^{-1}f^*\mathbf{S})\rangle_{\mathcal{G}}+\langle O^{-1}f^*\mathbf{S}\wedge \mathbf{S}\rangle_{\mathcal{G}}\in \Omega^2_{\mathrm{cl}}(M)\}\\
        &\cong \Diff(M)_{[H-\langle \mathbf{S}\wedge \dd \mathbf{S}\rangle_{\mathcal{G}}]}\times \Omega^2_{\mathrm{cl}}(M)\times \O(q,K-q)\times \Omega^1_{\mathrm{cl}}(M,\mathcal{G}).
    \end{align*}
    The multiplication map is given by \[(f_1,b_1,O_1,\mathbf{a}_1)\circ (f_2,b_2,O_2,\mathbf{a}_2)\longmapsto (f_1f_2,b_3,O_1O_2,\mathbf{a}_3),\]
    where 
    \begin{align*}
        \mathbf{a}_3&=O_2^{-1}f_2^*\mathbf{a}_1+\mathbf{a}_2\\
        b_3&=f_2^*b_1+b_2+\langle O_2^{-1}f_2^*\mathbf{a}_1\wedge \mathbf{a}_2\rangle_{\mathcal{G}}+[f_2^*,\QdR]\left(H-\langle \mathbf{S}\wedge \dd \mathbf{S}\rangle_{\mathcal{G}}-f_1^*(H-\langle \mathbf{S}\wedge \dd \mathbf{S}\rangle_{\mathcal{G}})\right),
    \end{align*}
    and the inversion map is given by
    \begin{align*}
        (f,b,O,\mathbf{a})\mapsto(f^{-1},b',O^{-1},-O(f^{-1})^*\mathbf{a}),
    \end{align*}
    where \[b'=-(f^{-1})^*b+[\QdR,(f^{-1})^*]\left(H-\langle \mathbf{S}\wedge \dd \mathbf{S}\rangle_{\mathcal{G}}-f^*(H-\langle \mathbf{S}\wedge \dd \mathbf{S}\rangle_{\mathcal{G}})\right).\]
\end{theorem}
\begin{proof}
    Like in the proofs of Proposition~\ref{prop: flat autos are tame manifold} and Theorem~\ref{thm: Auteq of flat std trans is tame Lie group}, we will first take a closer look at the constraints in Theorem~\ref{thm: gen automs trans} on $\nu=(f,O)\in\Auteq(\mathcal{G})$, $\mathbf{A}\in\Omega^1(M,\mathcal{G})$ and $B\in \Omega^2(M)$ to define an autoequivalence $F_{\nu,\mathbf{A},B}$ of $E$.

    By the same argument as in Proposition~\ref{prop: flat autos are tame manifold}, the first condition implies that $O$ must be constant, i.e.\ $O\in\O(q,K-q)$. The second condition becomes
    \begin{align*}
        O^{-1}f^*\dd\mathbf{S}=\dd\mathbf{S}-\dd\mathbf{A},
    \end{align*}
    which, by equation~\eqref{eq: d and pullback commute on forms with values in G} and the constancy of $O$, is equivalent to
    \begin{align*}
        \dd(\mathbf{A}+O^{-1}f^*\mathbf{S}-\mathbf{S})=0.
    \end{align*}
    Henceforth, we will write
    \begin{align*}
        \mathbf{A}=\mathbf{a}+\mathbf{S}-O^{-1}f^*\mathbf{S}\text{ for some }\mathbf{a}\in\Omega^1_{\mathrm{cl}}(M,\mathcal{G}).
    \end{align*}
    The last condition in Theorem~\ref{thm: gen automs trans} becomes
    \begin{align*}
        H-f^*H&=\dd B+ 2\langle\mathbf{A}\wedge \dd\mathbf{S}\rangle_{\mathcal{G}}-\langle\mathbf{A}\wedge\dd\mathbf{A}\rangle_{\mathcal{G}}\\
        &=\dd B+2\left\langle (\mathbf{a}+\mathbf{S}-O^{-1}f^*\mathbf{S})\wedge\dd\mathbf{S}\right\rangle_{\mathcal{G}}-\left\langle(\mathbf{a}+\mathbf{S}-O^{-1}f^*\mathbf{S})\wedge (\dd \mathbf{S}-O^{-1}f^*\dd\mathbf{S})\right\rangle_{\mathcal{G}}.
    \end{align*}
    Observing that for general $\mathbf{A'}\in\Omega^k(M,\mathcal{G})$ and $\mathbf{B'}\in\Omega^l(M,\mathcal{G})$, we have
    \begin{align*}
        \left\langle O^{-1}\mathbf{A'}\wedge O^{-1}\mathbf{B'}\right\rangle_{\mathcal{G}}&=\left\langle\mathbf{A'}\wedge \mathbf{B'}\right\rangle_{\mathcal{G}},\\
        \left\langle f^*\mathbf{A'}\wedge f^*\mathbf{B'}\right\rangle_{\mathcal{G}}&=f^*\left\langle\mathbf{A'}\wedge \mathbf{B'}\right\rangle_{\mathcal{G}},\\
        \dd\langle\mathbf{A'}\wedge\mathbf{B'}\rangle_{\mathcal{G}}&=\langle \dd\mathbf{A'}\wedge \mathbf{B'}\rangle_{\mathcal{G}}+(-1)^k\langle \mathbf{A'}\wedge \dd\mathbf{B'}\rangle_{\mathcal{G}},
    \end{align*}
    we find
    \begin{align*}
        H-f^*H=\dd B-\dd\left\langle \mathbf{a}\wedge(\mathbf{S}+O^{-1}f^*\mathbf{S})\right\rangle_{\mathcal{G}}-\dd\left\langle \mathbf{S}\wedge O^{-1}f^*\mathbf{S}\right\rangle_{\mathcal{G}}+\langle\mathbf{S}\wedge\dd\mathbf{S}\rangle_{\mathcal{G}}-f^*\langle\mathbf{S}\wedge\dd\mathbf{S}\rangle_{\mathcal{G}}.
    \end{align*}
    In particular, this shows that $f\in\Diff(M)_{[H-\langle \mathbf{S}\wedge \dd \mathbf{S}\rangle_{\mathcal{G}}]}$ and that
    \begin{align*}
        B=b+B(f,\mathbf{a})\text{ for some }b\in\Omega^2_{\mathrm{cl}}(M),
    \end{align*}
    where we defined 
    \begin{align*}
        B(f,\mathbf{a})=\QdR\left(H-\langle \mathbf{S}\wedge \dd \mathbf{S}\rangle_{\mathcal{G}}\right)+\left\langle \mathbf{a}\wedge(\mathbf{S}+O^{-1}f^*\mathbf{S})\right\rangle_{\mathcal{G}}+\left\langle \mathbf{S}\wedge O^{-1}f^*\mathbf{S}\right\rangle_{\mathcal{G}}.
    \end{align*}
    In total, we obtain a bijection
    \begin{align*}
        \Diff(M)_{[H-\langle \mathbf{S}\wedge \dd \mathbf{S}\rangle_{\mathcal{G}}]}\times \Omega^2_{\mathrm{cl}}(M)\times \O(q,K-q)\times \Omega^1_{\mathrm{cl}}(M,\mathcal{G})&\xlongrightarrow{\cong} \Auteq(E)\\
        (f,b,O,\mathbf{a})&\longmapsto F_{(f,O),\,\mathbf{a}+\mathbf{S}-O^{-1}f^*\mathbf{S},\,b+B(f,\mathbf{a})}.
    \end{align*}
    To express the multiplication in $\Auteq(E)$ in terms of this bijection, we write
    \begin{align*}
        (f_1,\,b_1,\,O_1,\,\mathbf{a}_1)(f_2,\,b_2,\,O_2,\,\mathbf{a}_2)=(f_1f_2,\,b_3,\,O_1O_2,\mathbf{a}_3).
    \end{align*}
    Then, using the formulas for the multiplication in Theorem~\ref{thm: gen automs trans}, we find that $\mathbf{a}_3$ is given by
    \begin{align*}
        \mathbf{a}_3&=O_2^{-1}f_2^*(\mathbf{a}_1+\mathbf{S}-O_1^{-1}f_1^*\mathbf{S})+\mathbf{a}_2+\mathbf{S}-O_2^{-1}f_2^*\mathbf{S}-\left(\mathbf{S}-(O_1O_2)^{-1}(f_1f_2)^*\mathbf{S}\right)\\
        &=O_2^{-1}f_2^*\mathbf{a}_1+O_2^{-1}f_2^*\mathbf{S}-(O_1O_2)^{-1}(f_1f_2)^*\mathbf{S}+\mathbf{a}_2+\mathbf{S}-O_2^{-1}f_2^*\mathbf{S}-\left(\mathbf{S}-(O_1O_2)^{-1}(f_1f_2)^*\mathbf{S}\right)\\
        &=O_2^{-1}f_2^*\mathbf{a}_1+\mathbf{a}_2.
    \end{align*}
    Similarly, 
    \begin{align*}
        b_3&=f_2^*\left(b_1+B(f_1,\mathbf{a}_1)\right)+\left(b_2+B(f_2,\mathbf{a}_2)\right)\\&\quad+\left\langle O_2^{-1}f_2^*\left(\mathbf{a}_1+\mathbf{S}-O_1^{-1}f_1^*\mathbf{S}\right)\wedge\left(\mathbf{a}_2+\mathbf{S}-O_2^{-1}f_2^*\mathbf{S}\right) \right\rangle_{\mathcal{G}}-B(f_1f_2,O_2^{-1}f_2^*\mathbf{a}_1+\mathbf{a}_2).
    \end{align*}
    To simplify this expression, we first compute
    \begin{align*}
        &f_2^*\left(b_1+B(f_1,\mathbf{a}_1)\right)+\left(b_2+B(f_2,\mathbf{a}_2)\right)-B(f_1f_2,O_2^{-1}f_2^*\mathbf{a}_1+\mathbf{a}_2)\\
        &=[f_2^*,\QdR]\left(H-\langle \mathbf{S}\wedge \dd\mathbf{S}\rangle_{\mathcal{G}}-f_1^*\left(H-\langle\mathbf{S}\wedge\dd\mathbf{S}\rangle_{\mathcal{G}}\right)\right)+\left\langle f_2^*\mathbf{a}_1\wedge f_2^*\mathbf{S}\right\rangle_{\mathcal{G}}\\
        &\quad -\left\langle O_1^{-1}(f_1f_2)^*\mathbf{S}\wedge f_2^*\mathbf{S}\right\rangle_{\mathcal{G}}+\left\langle\mathbf{a}_2\wedge O_2^{-1}f_2^*\mathbf{S}\right\rangle_{\mathcal{G}}-\left\langle O_2^{-1}f_2^*\mathbf{S}\wedge\mathbf{S}\right\rangle_{\mathcal{G}}\\
        &\quad-\left\langle O_2^{-1}f_2^*\mathbf{a}_1\wedge \mathbf{S}\right\rangle_{\mathcal{G}}-\left\langle\mathbf{a}_2\wedge (O_1O_2)^{-1}(f_1 f_2)^*\mathbf{S}\right\rangle_{\mathcal{G}}+\left\langle (O_1O_2)^{-1}(f_1f_2)^*\mathbf{S}\wedge\mathbf{S}\right\rangle_{\mathcal{G}}.
    \end{align*}
    Expanding and adding the remaining terms $$f_2^*b_1+b_2+\left\langle O_2^{-1}f_2^*\left(\mathbf{a}_1+\mathbf{S}-O_1^{-1}f_1^*\mathbf{S}\right)\wedge\left(\mathbf{a}_2+\mathbf{S}-O_2^{-1}f_2^*\mathbf{S}\right) \right\rangle_{\mathcal{G}}$$
    we find that
    \[b_3=f_2^*b_1+b_2+\langle O_2^{-1}f_2^*\mathbf{a}_1\wedge \mathbf{a}_2\rangle_{\mathcal{G}}+[f_2^*,\QdR]\left(H-\langle \mathbf{S}\wedge \dd \mathbf{S}\rangle_{\mathcal{G}}-f_1^*(H-\langle \mathbf{S}\wedge \dd \mathbf{S}\rangle_{\mathcal{G}})\right).\]
    Finally, using these formulas for the multiplication, we confirm that the inversion map is given by the expression stated in the theorem.
\end{proof}

Since $\Auteq(E)$ is conjugate to the group of autoequivalences of a flat standard transitive Courant algebroid by some $\mathbf{A}$-field transformation, and since the action of $\mathbf{A}$-field transformations on $\Gamma(M,E)$ is smooth tame, we conclude from Proposition~\ref{prop: action of ex auteq is smooth tame flat trans} that the action of $\Auteq(E)$ on $\Gamma(M,E)$ is smooth tame. 

Moreover, $\Auteq(E)$ has a tame Fréchet Lie subgroup, the group $\Autex(E)$ of exact autoequivalences, which is conjugate to the group of exact autoequivalences of some standard $\mathcal{G}$-flat transitive Courant algebroid. Its Lie algebra is again the image of $-\ad\colon \Gamma(M,E)\rightarrow \auteq(E)$. Using this fact, we will give an explicit description of the underlying tame Fréchet space below.

\begin{proposition}
\label{prop: autex for R exact}
Let $E$ be as in Theorem~\ref{thm: auteq for R=dS}. Then the Lie algebra of $\Autex(E)$ is given by
    \begin{align*}
        \autex(E)&=\{X+b+\mathbf{a}\in\auteq(E):\exists\, \alpha\in \Omega^1(M),\,\exists\,\mathbf{r}\in\Gamma(M,\mathcal{G})\text{ s.t.\ }\mathbf{a}=\dd \mathbf{r}+\dd \iota_X\mathbf{S},\\
        &\qquad b=\dd(\alpha-\langle 2\mathbf{r}+\iota_X\mathbf{S},\mathbf{S}\rangle_{\mathcal{G}})+\QdR\mathcal{L}_X(H-\langle\mathbf{S}\wedge \dd\mathbf{S}\rangle_{\mathcal{G}})-\iota_X(H-\langle\mathbf{S}\wedge \dd\mathbf{S}\rangle_{\mathcal{G}})\},
    \end{align*}
    where as above $\mathbf{S}=\sum_{i=1}^K S_i\mathbf{e}_i$ is such that $\mathbf{R}=\dd\mathbf{S}$, $S_i\in\dd^*\Omega^2(M)$.
\end{proposition}
\begin{proof}
    This can be proved by first computing $D_F\mu_E(F=\id,Y+\mathbf{s}+\beta)\{X+b+T+\mathbf{a}\}$ for general $X+b+T+\mathbf{a}\in\auteq(E)$ and then comparing to the expression of $-[X+\mathbf{r}+\alpha,Y+\mathbf{s}+\beta]$. For both expressions to agree with each other, we must have $T=0$, $\mathbf{a}=\dd\mathbf{r}+\dd\iota_X\mathbf{S}$ and $$b=\dd(\alpha-\langle 2\mathbf{r}+\iota_X\mathbf{S},\mathbf{S}\rangle_{\mathcal{G}})+\QdR\mathcal{L}_X(H-\langle\mathbf{S}\wedge \dd\mathbf{S}\rangle_{\mathcal{G}})-\iota_X(H-\langle\mathbf{S}\wedge \dd\mathbf{S}\rangle_{\mathcal{G}}).$$
\end{proof}

As before, we will identify $\autex (E)$ with a tame Fréchet subspace of $\Gamma(M,E)$. 

\begin{proposition}
\label{prop: hex as orthog complement of ad}
    Let $\{\mathbf{e}_i\}_{i=1}^K$ be some parallel orthonormal frame of $\mathcal{G}$ and $S_i\in\dd^*\Omega^2(M)$ such that $\mathbf{S}=\sum_i S_i\mathbf{e}_i$, where $\mathbf{S}$ is such that $\mathbf{R}=\dd\mathbf{S}$. We denote by $(G_{ij})\in\R^{K^2}$ the (positive definite symmetric) matrix defined by $G_{ij}=(\mathbf{e}_i,\mathbf{e}_j)_{\mathcal{G}}$, where $(\cdot,\cdot)_{\mathcal{G}}$ denotes the $L^2$-inner product on $\mathcal{G}$ defined in the beginning of Section~\ref{sect: auteq G-flat transitive}. There is a tame isomorphism 
    \begin{align*}
        \autex(E)&\cong (\ker\ad)^{\perp}\\&=\left\{X+\mathbf{r}-2\sum_{i,j}G_{ji}^{-1}(\epsilon_iS_i,\alpha)\mathbf{e}_j+\alpha:X\in \X(M),\,\mathbf{r}\in\dd_{\nabla}^*\Omega^1(M,\mathcal{G}),\,\alpha\in\dd^*\Omega^2(M)\right\},
    \end{align*}
    where the orthogonal complement of $\ker\ad$ is taken with respect to the $L^2$-inner product $(\cdot,\cdot)_E$ on $\Gamma(M,E)$ (defined in the beginning of Section~\ref{sect: auteq G-flat transitive}) and $(\cdot,\cdot)$ denotes the $L^2$-inner product induced by the auxiliary Riemannian metric $h$. 
\end{proposition}
\begin{proof}
    First, we observe by inspecting the expression of the Dorfman bracket that
    \begin{align*}
        Y+\mathbf{s}+\beta\in\ker(-\ad)=\ker(\ad)\iff Y=0,\, \dd\mathbf{s}=0,\,\beta-2\langle \mathbf{s},\mathbf{S}\rangle_{\mathcal{G}}\in \Omega^1_{\mathrm{cl}}(M). 
    \end{align*}
    We claim that the orthogonal complement of $\ker (\ad)$ with respect to $(\cdot,\cdot)_E$ exists and that it is given by 
    \begin{align*}
        (\ker\ad)^{\perp}=\left\{X+\mathbf{r}-2\sum_{i,j}G_{ji}^{-1}(\epsilon_iS_i,\alpha)\mathbf{e}_j+\alpha:X\in \X(M),\,\mathbf{r}\in\dd_{\nabla}^*\Omega^1(M,\mathcal{G}),\,\alpha\in\dd^*\Omega^2(M)\right\}.
    \end{align*}
    To see that it is orthogonal, we observe that for $\mathbf{s}+\beta\in\ker (\ad)$ and $X+\mathbf{r}-2\sum_{i,j}G_{ji}^{-1}(\epsilon_iS_i,\alpha)\mathbf{e}_j+\alpha$ with $\mathbf{r}\in\dd_{\nabla}^*\Omega^1(M,\mathcal{G})$, $\alpha\in\dd^*\Omega^2(M)$ we have 
    \begin{align*}
        \left(\mathbf{s}+\beta,\,X+\mathbf{r}-2\sum_{i,j}G_{ji}^{-1}(\epsilon_iS_i,\alpha)\mathbf{e}_j+\alpha\right)=-2\sum_is_i(\epsilon_iS_i,\alpha)+2(\alpha,\langle \mathbf{s},\mathbf{S}\rangle_{\mathcal{G}})=0,
    \end{align*}
    where we used that $s_i$ is constant for every $i$ since $\dd \mathbf{s}=0$ and $\pr_{\im\dd^*}\beta=2\pr_{\im\dd^*}\langle \mathbf{s},\mathbf{S}\rangle_{\mathcal{G}}$.

    To see that it is an algebraic complement, we need to show that for every $\mathbf{t}+\gamma\in \Gamma(M,\mathcal{G})\oplus \Omega^1(M)$ there exist $\mathbf{s}\in \Omega^0_{\mathrm{cl}}(M,\mathcal{G})$, $\mathbf{r}\in\dd_{\nabla}^*\Omega^1(M,\mathcal{G})$, $\tilde{\beta}\in\Omega^1_{\mathrm{cl}}(M)$ and $\alpha\in \dd^*\Omega^2(M)$ such that
    \begin{align*}
        \mathbf{t}+\gamma=\mathbf{s}+\tilde{\beta}+2\langle \mathbf{s},\mathbf{S}\rangle_{\mathcal{G}}+\mathbf{r}+\alpha-2\sum_{i,j}G_{ji}^{-1}(\epsilon_iS_i,\alpha)\mathbf{e}_j.
    \end{align*}
    Writing $\mathbf{t}=\mathbf{t}_0+\mathbf{t}_1$ with $\mathbf{t}_0\in\Omega^0_{\mathrm{cl}}(M,\mathcal{G})$, $\mathbf{t}_1\in\dd_{\nabla}^*\Omega^1(M,\mathcal{G})$ and $\gamma=\gamma_0+\gamma_1$ with $\gamma_0\in\Omega^1_{\mathrm{cl}}(M)$, $\gamma_1\in\dd^*\Omega^2(M)$, this is equivalent to the system
    \begin{equation}
    \label{eq: system for complement of ker ad}
        \begin{aligned}
            \gamma_0&=\tilde{\beta},\quad \gamma_1=\alpha + 2\langle\mathbf{s},\mathbf{S}\rangle_{\mathcal{G}}, \quad \mathbf{t}_1=\mathbf{r},\\
            \mathbf{t}_0&=\mathbf{s}-2\sum_{i,j}G_{ji}^{-1}(\epsilon_iS_i,\alpha)\mathbf{e}_j=\mathbf{s}-2\sum_{i,j=1}^KG^{-1}_{ji}(\epsilon_iS_i,\gamma_1-2\langle \mathbf{s},\mathbf{S}\rangle_{\mathcal{G}})\mathbf{e}_j,
        \end{aligned}
    \end{equation}
   where we substituted the equation for $\gamma_1$ in the second line and used that $\pr_{\im \dd^*}\langle \mathbf{s},\mathbf{S}\rangle_{\mathcal{G}}=\langle \mathbf{s},\mathbf{S}\rangle_{\mathcal{G}}$, since $S_i\in \dd^*\Omega^2(M)$ and $s_i$ are constant. Then~\eqref{eq: system for complement of ker ad} has a (unique) solution if and only if
   \begin{equation}\label{eq: complement of ker ad aux equation}\Omega^0_{\mathrm{cl}}(M,\mathcal{G})\ni\mathbf{s}\longmapsto \mathbf{s}+4\sum_{i,j=1}^KG^{-1}_{ji}(\epsilon_iS_i,\langle\mathbf{s},\mathbf{S}\rangle_{\mathcal{G}})\mathbf{e}_j\in \Omega^0_{\mathrm{cl}}(M,\mathcal{G})\end{equation}
   is invertible. This is equivalent to the invertibility of the matrix $A\in\R^{K^2}$ with
   \[A_{ik}=\delta_{ik}+4\sum_{j}G^{-1}_{ij}(\epsilon_jS_j,\epsilon_kS_k).\]
   Note that the matrix defined by $(\epsilon_iS_i,\epsilon_jS_j)$ is positive semi-definite. Indeed, given $(\lambda_i)\in\R^K$, we have 
   \begin{align*}
       \sum_{i,j}\lambda_i(\epsilon_iS_i,\epsilon_jS_j)\lambda_j=\left(\sum_i \lambda_i \epsilon_iS_i,\sum_j\lambda_j\epsilon_j S_j\right)\geq0.
   \end{align*}
   Since $G$ is positive-definite, it follows that $GA$ is invertible as a sum of a positive-definite and positive semi-definite matrix. Hence, $A$ is invertible.
   
   Note that both $\ker \ad$ and $(\ker \ad)^{\perp}$ are closed subspaces of $\Gamma(M,E)$ and therefore graded Fréchet spaces. Moreover, the projections of $\Gamma(M,E)$ to the respective subspaces are tame linear maps, which can be seen from the fact that the solution $(\mathbf{s},\alpha,\mathbf{r},\tilde{\beta})$ of \eqref{eq: system for complement of ker ad}  depends tame linearly on $\mathbf{t}+\gamma$ (note that the inverse map of \eqref{eq: complement of ker ad aux equation} is automatically tame linear, being a linear map between finite dimensional vector spaces). This shows that $\ker \ad$ and $(\ker \ad)^{\perp}$ are tame direct summands of $\Gamma(M,E)$ and thus tame Fréchet spaces.

   We conclude that
   \begin{align*}
       \autex(E)&\cong (\ker \ad)^{\perp}\\&=\left\{X+\mathbf{r}-2\sum_{i,j}G_{ji}^{-1}(\epsilon_iS_i,\alpha)\mathbf{e}_j+\alpha:X\in \X(M),\,\mathbf{r}\in\dn^*\Omega^1(M,\mathcal{G}),\,\alpha\in\dd^*\Omega^2(M)\right\},
   \end{align*}
   where the map $(\ker \ad)^{\perp}\longrightarrow \autex(E)$ is given by the restriction of $\ad\colon \Gamma(M,E)\longrightarrow \autex(E)$. To see that this is a tame isomorphism, using Proposition~\ref{prop: autex for R exact} we observe that its inverse is given by
   \[\autex(E)\ni X+b+\mathbf{a}\longmapsto X+\mathbf{r}(X,\mathbf{a})-2\sum_{i,j}G_{ji}^{-1}\left(\epsilon_iS_i,\alpha(X,b,\mathbf{a})\right)\mathbf{e}_j+\alpha(X,b,\mathbf{a}),\]
   where 
   \begin{align*}
       \mathbf{r}(X,\mathbf{a})&=\Q_{\nabla}(\mathbf{a}-\dd \iota_X \mathbf{S}),\\
       \alpha(X,b,\mathbf{a})&=\QdR\left(b+\dd\langle 2\,\mathbf{r}(X,\mathbf{a})+\iota_X\mathbf{S},\mathbf{S}\rangle_{\mathcal{G}}-\QdR\mathcal{L}_X(H-\langle\mathbf{S}\wedge \dd\mathbf{S}\rangle_{\mathcal{G}})+\iota_X(H-\langle\mathbf{S}\wedge \dd\mathbf{S}\rangle_{\mathcal{G}})\right).
   \end{align*}
   Here we use that for $\mathbf{A}\in\dd\Omega^0(M,\mathcal{G})$, we have $\dd\Q_{\nabla}\mathbf{A}=\mathbf{A}$ and similarly, for $B\in\dd\Omega^1(M)$, we have $\dd\QdR B=B$.
\end{proof}

In view of Proposition~\ref{prop: hex as orthog complement of ad} (and Remark~\ref{rk: hex id with orthog complement of ad for std flat}) we will from now on always identify 
\begin{align}\label{eq: hex as a subspace of sections of E}
  \autex(E)\cong (\ker\ad)^{\perp}\subset \Gamma(M,E). 
\end{align}
Last, we observe that given some generalized complex structure $L$, we may use the isomorphism of real vector bundles
\begin{equation}\label{eq: definition of Psi}\begin{aligned}
    \Psi\colon \bar{L}&\longrightarrow E\\
    \xi&\longmapsto \xi+\bar{\xi}\end{aligned}
\end{equation}
to describe exact autoequivalences by elements in $\Psi^{-1}(\autex(E))\subset\Gamma(M,L^*)$. For $\xi\in \Psi^{-1}(\autex(E))$, we will then also denote $F_{\xi}:=F_{\Psi\xi}$. 

\subsubsection{Choices of bisection and bundle metrics}

Next, we consider a $\mathcal{G}$-flat transitive Courant algebroid with a generalized complex structure. Note that the existence of a generalized complex structure forces $\langle\cdot,\cdot\rangle$ to be of signature $(2k,2\ell)$, where $2k=\dim M+q$ and $2\ell = \dim M +K-q$. In other words, it forces $K$ and $\dim M +q$ to be even.

Recall that a \emph{generalized (Riemannian) metric}  on a Courant algebroid $E$ is a positive definite scalar product $\mathfrak H$ on $E$ such that the endomorphism $\mathfrak H^{\mathrm{end}}$ defined by $\mathfrak H=\langle \mathfrak H^{\mathrm{end}}\cdot , \cdot \rangle$ is an involution. We denote by $E_\pm\subset E$ the $\pm 1$-eigenbundles of $\mathfrak H^{\mathrm{end}}$.

\begin{lemma} \label{lem: ex gen herm}
    Given a generalized complex structure $\mathcal{J}$ on a Courant algebroid $E$, there exists a generalized Hermitian metric with respect to $\mathcal{J}$, i.e.\ a generalized metric $\mathfrak{H}$ such that $\mathfrak{H}(\mathcal{J}\cdot,\mathcal{J}\cdot)=\mathfrak{H}(\cdot,\cdot)$.
\end{lemma}
\begin{proof}
    This follows from the fact that every $\mathrm{U}(k,\ell )$-structure admits a reduction to the maximal compact subgroup $\mathrm{U}(k)\times \mathrm{U}(\ell)\subset \mathrm{U}(k,\ell )$.
\end{proof}
\begin{remark}\label{rk: construction generalized metric} 
    Explicitly, a generalized Hermitian metric can be obtained by constructing $\mathcal{J}$-invariant subspaces $E_{\pm}$ in the following way: Take an orthonormal vector $u_1\in E$, say $\langle u_1,u_1\rangle=\epsilon\in \{\pm 1\}$. Then 
    \[\langle \mathcal{J}u_1,\mathcal{J}u_1\rangle=\langle u_1,u_1\rangle=\epsilon.\]
    Hence, $\operatorname{span}\{u_1,\mathcal{J}u_1\}$ is either a positive or a negative subbundle of $E$.

    Next, we pick some vector $u_2\in \left(\operatorname{span}\{u_1,\mathcal{J}u_1\}\right)^{\perp}$ and normalize such that $|\langle u_2,u_2\rangle|=1$. Proceeding as before and iterating the procedure yields two subbundles $E_{\pm}$ that are $\mathcal{J}$-invariant and define a generalized metric that is in particular a generalized Hermitian metric.
\end{remark}

A statement similar to the following lemma already appeared in~\cite{Juro2016} but with a different proof.

\begin{lemma}\label{lem: block form}
    Let $\mathfrak{H}$ be a generalized metric on a transitive Courant algebroid $E$. Then there is a bisection of $E$ in which $\mathfrak{H}=\frac12 g\oplus\langle\cdot,\cdot\rangle_{\mathrm{Eucl}}\oplus \frac12 g^*$, where $g$ is a Riemannian metric on $M$ and $\langle\cdot,\cdot\rangle_{\mathrm{Eucl}}$ is a positive definite scalar product on $\mathcal{G}$.
\end{lemma}
\begin{proof}
    We apply similar ideas as in \cite{GarciaFernandez2014}. Since $\mathfrak{H}$ is positive definite, we must have \[E_+\cap T^*M=\{0\}=E_-\cap T^*M\]
    and it follows that the restrictions
    \[\pr_{\mathcal{G}}|_{\ker\pi\cap E_{\pm}}\colon \ker\pi\cap E_{\pm}\longrightarrow \mathcal{G}\]
    are injective. We denote $\mathcal{G}_{\pm}=\pr_{\mathcal{G}}(E_{\pm}\cap\ker\pi)$. Note that $\langle\cdot,\cdot\rangle_{\mathcal{G}}|_{\mathcal{G}_+}>0$ and $\langle\cdot,\cdot\rangle_{\mathcal{G}}|_{\mathcal{G}_-}<0$. It follows that
    \[\mathcal{G}_+\cap\mathcal{G}_-=\{0\},\quad \operatorname{rank} \mathcal{G}_+\leq q,\quad \operatorname{rank} \mathcal{G}_-\leq K-q.\]
    Moreover, we denote by $(E_+\cap\ker\pi)^{\perp}\subset E_+$ the orthogonal complement of $E_+\cap\ker\pi$. Then the anchor map defines an injective vector bundle morphism
    \[\pi_{\perp} =\pi|_{(E_{+}\cap\ker \pi)^{\perp}}\colon (E_+\cap\ker \pi)^{\perp}\longrightarrow TM.\]
    In particular, $\operatorname{rank}(E_+\cap\ker\pi)^{\perp}\leq\dim M$ and we find 
    \[\dim M\geq \operatorname{rank}  (E_+\cap\ker\pi)^{\perp}=\operatorname{rank} E_+-\operatorname{rank} \mathcal{G}_+\geq \dim M+q-q=\dim M\]
    and hence 
    \[\operatorname{rank}\mathcal{G}_+=q,\quad\operatorname{rank} (E_+\cap\ker\pi)^{\perp}=\dim M.\]
    A similar argument shows that $\operatorname{rank} \mathcal{G}_-=K-q$. Hence $\mathcal{G}=\mathcal{G}_+\oplus \mathcal{G}_-$ and $\pi_{\perp}$ is in fact an isomorphism. Moreover, we obtain isomorphisms
\[\pr_{\mathcal{G}_{\pm}}=\pr_\mathcal{G}|_{\ker\pi\cap E_{\pm}}\colon \ker\pi\cap E_{\pm}\longrightarrow \mathcal{G}_{\pm}.\]

    The isomorphism $\pi_{\perp}$ determines a Riemannian metric on $TM$ via 
    \[g(\cdot,\cdot):=\mathfrak{H}(\pi_{\perp}^{-1}\cdot,\pi_{\perp}^{-1}\cdot )=\langle \pi_{\perp}^{-1}\cdot,\pi_{\perp}^{-1}\cdot \rangle.\]
    We define the map
    \begin{align*}
        \lambda\colon TM &\longrightarrow E\\
        X &\longmapsto \pi_{\perp}^{-1}(X)-\frac12\pi^*g(X).
    \end{align*}
    Then clearly $\pi\circ\lambda=\id$ and 
    \begin{align*}
        \langle \lambda(X),\lambda(Y)\rangle&=\langle\pi_{\perp}^{-1}(X)-\frac12\pi^*g(X),\pi_{\perp}^{-1}(Y)-\frac12\pi^*g(Y)\rangle\\
        &=g(X,Y)-\frac12  g(X)(Y)-\frac12 g(Y)(X)\\
        &=g(X,Y)-g(X,Y)\\
        &=0.
    \end{align*}

    Moreover, for every $\mathbf{r}\in\mathcal{G}$ we write $\mathbf{r}=\mathbf{r}_+ +\mathbf{r}_-$ and define
    \begin{align*}
        \sigma\colon \mathcal{G}&\longrightarrow \ker\pi\\
        \mathbf{r}&\longmapsto \pr_{\mathcal{G}_+}^{-1}(\mathbf{r}_+)+\pr_{\mathcal{G}_-}^{-1}(\mathbf{r}_-).
    \end{align*}
    Then for every $X\in TM$ and $\mathbf{r}\in\mathcal{G}$, we find
    \begin{align*}
        \langle\sigma(\mathbf{r}),\lambda(X)\rangle=\langle \pr_{\mathcal{G}_+}^{-1}(\mathbf{r}_+),\pi_{\perp}^{-1}(X)\rangle=0.
    \end{align*}

    Therefore, $(\lambda,\sigma)$ define a bisection of $E$. Since
    \begin{align*}
     E_+&=\{\pi_{\perp}^{-1}(X)+\pr_{\mathcal{G}_+}^{-1}(\mathbf{r}_+):X\in TM,\, \mathbf{r}_+\in \mathcal{G}_+\}\\
     &=\{\lambda(X)+\frac12\pi^*g(X)+\pr_{\mathcal{G}_+}^{-1}(\mathbf{r}_+):X\in TM,\, \mathbf{r}_+\in \mathcal{G}_+\},
    \end{align*}
    we observe that in the bisection defined by $(\lambda,\sigma)$ we have
    \begin{align*}
        E_+&=\{X+g(X)+\mathbf{r}_+:X\in TM,\, \mathbf{r}_+\in \mathcal{G}_+\},\\
        E_-&=\{X-g(X)+\mathbf{r}_-:X\in TM,\, \mathbf{r}_-\in \mathcal{G}_-\}.
    \end{align*}
   It follows that in this splitting $\mathfrak{H}$ takes the form $\mathfrak{H}=\frac12 g\oplus\langle\cdot,\cdot\rangle_{\mathrm{Eucl}}\oplus \frac12 g^*$, where $$\langle\cdot,\cdot\rangle_{\mathrm{Eucl}}=\langle\cdot,\cdot\rangle_{\mathcal{G}}|_{\mathcal{G}_+}- \langle\cdot,\cdot\rangle_{\mathcal{G}}|_{\mathcal{G}_-}.$$
\end{proof}

Note that in the case where $\langle\cdot,\cdot\rangle_{\mathcal{G}}$ is positive definite we must have $\langle\cdot,\cdot\rangle_{\mathrm{Eucl}}=\langle\cdot,\cdot\rangle_{\mathcal{G}}$. Combining Lemma~\ref{lem: ex gen herm} and Lemma~\ref{lem: block form}, we obtain the following.
\begin{corollary}
    Suppose that $E$ is a $\mathcal{G}$-flat transitive Courant algebroid with a generalized complex structure $\mathcal{J}$. Then there exists a generalized Hermitian metric $\mathfrak{H}$, a Riemannian metric $h$ on $M$ and a choice of bisection such that $\mathfrak{H}=h\oplus \langle\cdot,\cdot\rangle_{\mathrm{Eucl}}\oplus h^*$.
\end{corollary}
\begin{proof}
    This is a direct consequence of Lemma~\ref{lem: ex gen herm} and Lemma~\ref{lem: block form} by defining $h:=\tfrac{1}{2}g$.
\end{proof}

Given a $\mathcal{G}$-flat transitive Courant algebroid with a generalized complex structure, we will fix a bisection and bundle metrics on $TM$, $L$ and $\bar{L}$ in the following way: First, we choose a generalized Hermitian metric $\mathfrak{H}$ on $E$. Then we choose a bisection $E\cong TM\oplus\mathcal{G}\oplus T^*M$ and a Riemannian metric $h$ on $M$ such that the generalized Hermitian metric is given by $\mathfrak{H}=h\oplus\langle\cdot,\cdot\rangle_{\mathrm{Eucl}}\oplus h^*$. Restricting the complex sesqui-linear extension of $\mathfrak{H}$ to the subbundles $L$ and $\bar{L}$ yields Hermitian bundle metrics on these subbundles. 

In the following, we will always assume that the $L^2$-inner products on the spaces $\Gamma(M,\Lambda^{\bullet}L^*)$, as well as on the spaces $\Omega^{\bullet}(M)$ and $\X(M)$ are defined in terms of these choices of bundle metrics and with respect to the Riemannian metric $h$. Moreover, the Riemannian metric defining the charts for $\Diff(M)_{[H]}$ will be chosen to be $h$.

Like in the case of classical complex deformations~\cite{kuranishi64,paperoncomplexdefs}, when proving the local completeness of the deformation family, we will use Hodge theory in order to argue why a given smooth tame family of linear maps is invertible in an open neighborhood. In order to make sense of this, it is necessary to argue that the subspace $\dd_L^*\left(\Gamma(M,\Lambda^2L^*)\right)\subset \Gamma(M,L^*)$ is contained in the subspace $\Psi^{-1}\autex(E)\subset \Gamma(M,L^*)$ that the Lie algebra of the exact autoequivalences is identified with. The following lemma shows that this is the case for our choice of bundle metrics.

Note that this technicality does not arise in the case of (classical) complex deformations, where the Lie algebra of the diffeomorphisms is identified with the entire space of sections $\Gamma(M,T^{1,0}M)$ and therefore trivially contains the $\dbar^*$-exact elements.

\begin{lemma}
\label{lem: autex contains coexact}
    Let $E$ be a $\mathcal{G}$-flat transitive Courant algebroid with a generalized complex structure $L$. With the above choices of bundle metrics defining the $L^2$-inner products and the identification $\autex(E)\cong (\ker\ad)^{\perp}$ as in~\eqref{eq: hex as a subspace of sections of E}, it holds that $$\Psi^{-1}\autex(E)\supset\dd_L^* \left(\Gamma(M,\Lambda^2 L^*)\right).$$
    Moreover, $\dd_L^* \left(\Gamma(M,\Lambda^2 L^*)\right)$ has finite codimension in $\Psi^{-1}\autex(E)$ and if $H^1(M,L)=0$, then 
    $$\Psi^{-1}\autex(E)=\dd_L^* \left(\Gamma(M,\Lambda^2 L^*)\right).$$
\end{lemma}
\begin{proof}
Under the identification $\autex(E)\cong (\ker\ad)^{\perp}$, we have an orthogonal decomposition
\begin{align*}
\Gamma(M,E)=\autex(E)\oplus \ker(\ad).    
\end{align*}
Since $\Psi^{-1}=\pr_{\bar{L}}:\Gamma(M,E)\longrightarrow \Gamma(M,\bar{L})$, with our choices of bundle metrics $\Psi^{-1}$ is in fact a homothety and we obtain an orthogonal decomposition
\begin{align*}
    \Gamma(M,\bar{L})=\Psi^{-1}\autex(E)\oplus \Psi^{-1}\ker (\ad).
\end{align*}

On the other hand recall that by the Hodge decomposition $$\left(\dd_L^* \left(\Gamma(M,\Lambda^2 L^*)\right)\right)^{\perp}=\ker \left(\dd_L:\Gamma(M,L^*)\longrightarrow \Gamma(M,\Lambda^2L^*)\right).$$
We consider $\xi\in\Gamma(M,L^*)\cong\Gamma(M,\bar{L})$ such that $\Psi\xi=\xi+\bar{\xi}\in\ker \ad$, and $u,v\in \Gamma(M,L)$. Then by Lemma~\ref{lem: different formulas for dL and brackets} we have
\begin{align*}
    \dd_L\xi(u,v)=-\langle [\xi,u],v\rangle=-\langle [\xi+\bar{\xi},u],v\rangle=0.
\end{align*}
It follows that $\xi\in\ker \dd_L$ and therefore $\ker \left(\dd_L:\Gamma(M,L^*)\longrightarrow \Gamma(M,\Lambda^2L^*)\right)\supset \Psi^{-1}\ker(\ad)$. This implies 
\begin{align*}
   \dd_L^*\left(\Gamma(M,\Lambda^2L^*)\right)=\left(\ker \left(\dd_L:\Gamma(M,L^*)\longrightarrow \Gamma(M,\Lambda^2L^*)\right)\right)^{\perp}\subset \left(\Psi^{-1}\ker \ad\right)^{\perp}=\Psi^{-1}\autex(E).
\end{align*}
We claim that $\dd_L C^{\infty}(M)\subset \Psi^{-1}\ker (\ad).$ Given $f\in C^{\infty}(M)$ and $u\in\Gamma(M,L)$, we compute
\begin{align*}
    \dd_Lf(u)=\pi(u)[f]=\pr_{\bar{L}}\pi^*\dd f(u).
\end{align*}
Since $$\pr_{\bar{L}}\pi^*\dd f=\Psi^{-1}(\pi^*\dd f)\in\Psi^{-1}\ker(\ad),$$
we conclude that $$\dd_L C^{\infty}(M)\subset \Psi^{-1}\ker (\ad)\quad\implies \Psi^{-1}\autex(E)\cap \dd_L C^{\infty}(M)=\{0\}.$$
Hence 
$$\Psi^{-1}\autex(E)\cap\ker \left(\dd_L:\Gamma(M,L^*)\longrightarrow \Gamma(M,\Lambda^2L^*)\right)\subset \mathcal{H}_L^1(M).$$
\end{proof}

\subsection{Exact autoequivalences acting on generalized complex deformations}
\label{sect: auteq acting on gc defs}
In the following, let $E$ be a $\mathcal{G}$-flat transitive Courant algebroid over a compact manifold. Generalized almost complex structures on $E$ can be seen as sections to a fiber subbundle of $\End E$, whose fibers are isomorphic to $\O(2k,2\ell)/\mathrm{U}(k,\ell)$. In particular, the space of generalized almost complex structures on $E$ is the space of all sections $\mathrm{GAC}(E)$ of this fiber bundle and thus a tame Fréchet manifold~\cite[Thm II.2.3.1]{hamilton82}.

 Recall that the group of autoequivalences acts on the tame Fréchet manifold of generalized almost complex structures via
\begin{align*}
    \Auteq(E)\times \operatorname{GAC}(E)\ni (F,\mathcal{J})&\longmapsto F\cdot \mathcal{J}:= F\circ \mathcal{J}\circ F^{-1}.
\end{align*} 
For exact $E$, it was already shown in \cite{jans_thesis} that this action is smooth tame. In the following, we will see this more generally in the case of $\mathcal{G}$-flat transitive Courant algebroids. For this, we need the following lemma.
\begin{lemma}
\label{lem: action auteq on end}
    The map
    \begin{align*}
        \mu_{\mathrm{End}}\colon\Auteq(E)\times\Gamma(M,\End E)&\longrightarrow \Gamma(M,\End E)\\
        (F,A)&\longmapsto F\circ A\circ F^{-1}
    \end{align*}
    is smooth tame.
\end{lemma}
\begin{proof}
    We consider the auxiliary map
    \begin{align*}
        \Auteq(E)\times\Gamma(M,\End E)\times \Gamma(M,E)&\longrightarrow \Gamma(M,E)\\
        (F,A,u)&\longmapsto F\circ A\circ F^{-1}u.
    \end{align*}
    Then we observe that this map is smooth tame since
    \begin{align*}
        F\circ A\circ F^{-1}u&=F\circ A\circ F^{-1}\circ u\circ f\circ f^{-1}\\
        &=\mu_E(F,A\mu_E(F^{-1},u)),
    \end{align*}
    where $f$ is the diffeomorphism covered by $F$. In a locally trivial neighborhood $U\subset M$ for $E$, the map $\mu_{\mathrm{End}}$ can be recovered from the auxiliary map using a local frame $\{e_i\}$ for $E$ and its dual local frame $\{\alpha_i\}$ for $E^*$ via
    \begin{align*}
        \mu_{\mathrm{End}}(F,A)|_U=\sum_i \mu_E(F,A\mu_E(F^{-1},e_i))\otimes \alpha_i.
    \end{align*}
    These patch together to the global section $\mu_{\mathrm{End}}(F,A)$ of $\End E$. The patching can be made explicit using a partition of unity. In total, this yields a smooth tame map.
\end{proof}

\begin{proposition}
    The action of $\Auteq(E)$ on $\mathrm{GAC}(E)$ is smooth tame.
\end{proposition}
\begin{proof}
    This follows from Lemma~\ref{lem: action auteq on end} since $\mathrm{GAC}(E)$ is a tame Fr\'echet submanifold of $\Gamma(M,\End E)$.
\end{proof}

Consider now a generalized almost complex structure $\mathcal{J}$ on $E$ and an autoequivalence $F$. Then we have
\begin{align*}
    L_p(F\cdot\mathcal{J})&=\left\{F u_{f^{-1}(p)}:u_{f^{-1}(p)}\in L_{f^{-1}(p)}\right\}\\
    &=\left\{F\circ u\circ f^{-1}(p):u\in\Gamma(M,L)\right\}\\
    &=\left\{[\mu_E(F,u)]_p:u\in \Gamma(M,L)\right\}.
\end{align*}
For $F$ in a small enough $C^1$-neighborhood $\mathcal{V}^{(1)}$ of $\id\in \Auteq(E)$, we find that
\begin{align*}
    L(F\cdot \mathcal{J})\cap \bar{L}=\{0\}
\end{align*}
and hence there is some $\Theta(F)\in \Gamma(M,\Lambda^2 L^*)$ such that 
\begin{align*}
    L(F\cdot \mathcal{J})=\operatorname{graph}(\Theta(F)).
\end{align*}
In particular, $\Theta(F)$ is the generalized (almost) complex deformation corresponding to $F\cdot \mathcal{J}$. Similarly, possibly after shrinking $\mathcal{V}^{(1)}$ further, and for some small enough $C^0$-neighborhood $\mathcal{W}^{(0)}$ of $0$ in $\Gamma(M,\Lambda^2 L^*)$, for every $F\in \mathcal{V}^{(1)}$, $\varphi\in\mathcal{W}^{(0)}$, the generalized almost complex structure $F\cdot \mathcal{J}_{\varphi}$ corresponds to a generalized almost complex deformation $\Theta(F,\varphi)$. Since we can think of $\Theta$ as the local expression for the action of $\Auteq(E)$ on $\mathrm{GAC}(E)$ in a local chart around $\mathcal{J}$, it is a smooth tame map. We will state this in the following proposition.

\begin{proposition}
\label{prop:Theta smooth tame}
The map 
\begin{align*}
    \Theta\colon \mathcal{V}^{(1)}\times \mathcal{W}^{(0)}&\longrightarrow \Gamma(M,\Lambda^2 L^*)\\
    (F,\varphi)&\longmapsto \Theta(F,\varphi)
\end{align*}
is smooth tame. It is given by the formula
\begin{align*}
        \Theta(F,\varphi)=(F_{\bar{L}L}+F_{\bar{L}\bar{L}}\circ \varphi )(F_{LL}+F_{L\bar{L}}\circ\varphi)^{-1},
    \end{align*}
    where we wrote
     \begin{align*}
    F=\begin{pmatrix}
        F_{LL} & F_{L\bar{L}}\\
        F_{\bar{L}L}& F_{\bar{L}\bar{L}}
    \end{pmatrix}\qquad \text{w.r.t.}\quad E=L\oplus \bar{L}.
    \end{align*}
\end{proposition}
\begin{proof}
    We only need to prove the formula for $\Theta(F,\varphi)$.
    First, we write
    \begin{align*}
        L_p(F\cdot \mathcal{J}_{\varphi})&=\left\{F(u+\iota_u\varphi)_{f^{-1}(p)}:u\in\Gamma(M,L)\right\}\\
        &=\left\{(F_{LL}+F_{L\bar{L}}\circ\varphi)u_{f^{-1}(p)}+(F_{\bar{L}L}+F_{\bar{L}\bar{L}}\circ \varphi )u_{f^{-1}(p)}:u\in\Gamma(M,L)\right\}.
    \end{align*}
    Note that for $F$ in a small $C^1$-neighborhood $\mathcal{V}^{(1)}$ of the identity and $\varphi$ in a small $C^0$-neighborhood $\mathcal{W}^{(0)}$ of $0$, the map
    \begin{align*}
        F_{LL}+F_{L\bar{L}}\circ\varphi\colon L\longrightarrow L
    \end{align*}
    is invertible. Therefore, for every $u\in \Gamma(M,L)$ there is some $v\in \Gamma(M,L)$ such that for every $p\in M$ we have
    \begin{align*}
    v_p=(F_{LL}+F_{L\bar{L}}\circ\varphi)u_{f^{-1}(p)}.    
    \end{align*}
    It follows that
    \begin{align*}
     L_p(F\cdot \mathcal{J}_{\varphi})&=\left\{v_p+(F_{\bar{L}L}+F_{\bar{L}\bar{L}}\circ \varphi )(F_{LL}+F_{L\bar{L}}\circ\varphi)^{-1}v_p:v\in\Gamma(M,L)\right\}    
    \end{align*}
    and thus we find
    \begin{align*}
        \Theta(F,\varphi)&=(F_{\bar{L}L}+F_{\bar{L}\bar{L}}\circ \varphi )(F_{LL}+F_{L\bar{L}}\circ\varphi)^{-1}.
    \end{align*}
\end{proof}
From now on, we will also denote $\Theta(F,\varphi):=F\cdot \varphi$.
\begin{lemma}
    $\Theta(F_1\circ F_2,\varphi)=F_1\cdot \Theta(F_2,\varphi)$.
\end{lemma}
\begin{proof}
    This follows from 
    \begin{align*}
       (F_1\circ F_2)\cdot \mathcal{J}&=(F_1\circ F_2)\circ \mathcal{J}\circ (F_1\circ F_2)^{-1}\\
       &=F_1\circ (F_2\circ \mathcal{J}\circ F_2^{-1})\circ F_1^{-1}\\
       &=F_1\cdot (F_2\cdot\mathcal{J}). 
    \end{align*}
\end{proof}

We denote by 
\begin{align*}
    \mathcal{W}^{(k)}_{\varepsilon}&:=\{\varphi\in \Gamma(M,\Lambda^2L^*):\Vert \varphi\Vert_{k,\infty}<\varepsilon\},\\\mathcal{U}^{(k)}_{\varepsilon}&:=\left\{\xi\in \Psi^{-1}(\autex(E))\subset\Gamma(M,L^*)\colon \Vert \xi \Vert_{k,\infty}<\varepsilon\right\},
\end{align*}
where $\|\cdot\|_{k,\infty}$ denotes the $C^k$-norm on the respective spaces of sections. Then for $\varepsilon>0$ small enough, restricting $\Theta$ to exact autoequivalences and concatenating with the local chart $\Phi$ of $\Autex(E)$ yields a smooth tame map
\begin{align*}
    \mathcal{E}\colon \mathcal{U}_{\varepsilon}^{(1)}\times \mathcal{W}_{\varepsilon}^{(0)}\longrightarrow \Gamma(M,\Lambda^2 L^*).
\end{align*}
For our proof of the local completeness of the deformation family in Section~\ref{sec: local cplt gc}, we need to compute the partial derivative $D_{\xi}\mathcal{E}(\xi,\varphi)$. First, we prove the following proposition.
\begin{proposition}
\label{prop: taylor E}
    Let $\varphi\in \mathcal{W}^{(0)}_{\varepsilon}$ be a generalized almost complex deformation, $F_{\xi}$ an exact autoequivalence, where $\xi\in \mathcal{U}^{(1)}_{\varepsilon}$. Then for $\varepsilon>0$ small enough, $F_{\xi}\cdot\varphi$ is well-defined and we have
    \begin{align*}
        \mathcal{E}(\xi,\varphi):=F_{\xi}\cdot \varphi=\varphi+\dd_L(a_{\varphi}\xi)+[\varphi,a_{\varphi}\xi]-\iota_{\bar{\xi}}\left(\dd_L\varphi+\tfrac{1}{2}[\varphi,\varphi]\right)+R(\varphi,\xi),
    \end{align*}
    where $a_{\varphi}\xi:=\xi-\iota_{\bar{\xi}}\varphi$ and $R(\varphi,t\xi)=t^2 \tilde{R}(\varphi,\xi,t)$ for a real parameter $t$ with $\tilde{R}(\varphi,\xi,t)$ depending smoothly on $t$ for small $t$.
\end{proposition}
\begin{proof}
    We know that 
    \begin{align*}
        L_{F_{\xi}\cdot \varphi,p}&=\left\{ \mu_E(F_{\xi},u+\iota_u\varphi)_p : u\in \Gamma(M,L) \right\}\\
        &=\left\{u_p+\varphi(u_p)-[\Psi\xi,u+\varphi(u)]_p+R_{\mu_E}(\Psi\xi)(u+\iota_u\varphi)_p:u\in\Gamma(M,L)\right\},
    \end{align*}
    where we used Proposition~\ref{prop: mu for small exact autoeqs} and $R_{\mu_E}$ is as defined there. It follows that
    \begin{align*}
        \pr_L( L_{F_{\xi}\cdot \varphi,p})&=\left\{ u_p-\pr_L[\bar{\xi}+\xi,u+\iota_u\varphi]_p+\pr_LR_{\mu_E}(\Psi\xi)(u+\iota_u\varphi)_p: u\in \Gamma(M,L)\right\}\\
        &=\left\{ u_p-\pr_L[\bar{\xi}+\xi,u]_p-\pr_L [\bar{\xi},\iota_u\varphi]_p +\pr_LR_{\mu_E}(\Psi\xi)(u+\iota_u\varphi)_p : u\in \Gamma(M,L)\right\}.
    \end{align*}
    
    We define $P(\xi,\varphi)\colon \Gamma(M,L)\longrightarrow \Gamma(M,L)$ by
    \begin{align*}
        P(\xi,\varphi)u&:= \pr_L\left(\mu_E(F_{\xi}, u+\iota_u\varphi)\right)\\
        &\;=u-\pr_L[\bar{\xi}+\xi,u]-\pr_L [\bar{\xi},\iota_u\varphi] +\pr_LR_{\mu_E}(\Psi\xi)(u+\iota_u\varphi).
    \end{align*}
    Then $P(\xi,\varphi)$ is invertible for $(\xi,\varphi)\in \mathcal{U}_{\varepsilon}^{(1)}\times \mathcal{W}_{\varepsilon}^{(0)}$ with inverse 
    \begin{align*}
        P(\xi,\varphi)^{-1}u=u+[\bar{\xi},u]+\pr_L[\xi,u]+\pr_L[\bar{\xi},\iota_u\varphi]+\text{ higher orders of }\xi,
    \end{align*}
    where we used that the family of inverse maps $P(\xi,\varphi)^{-1}u$ depends smoothly on $\xi$. (To see that it is smooth in $\xi$, note that the inverse of $\pr_L\colon L_{F_{\xi}\cdot \varphi}\rightarrow L$ is given by $\id+\mathcal{E}(\xi,\varphi)\colon L\to L_{F_{\xi}\cdot\varphi}\subset E$, which we already know depends smoothly on $\xi$, and the inverse of $u\mapsto\mu_E(F_\xi,u+\iota_u\varphi)$ is given by $v\mapsto\pr_L\mu_E(F_{\xi}^{-1},v)$.)

    It follows that 
    \begin{align*}
        L_{F_{\xi}\cdot \varphi,p}&=\left\{ u_p+\varphi\left(P(\xi,\varphi)^{-1}u\right)_p-\pr_{\bar{L}}[\Psi\xi,P(\xi,\varphi)^{-1}u+\varphi\left(P(\xi,\varphi)^{-1}u\right)]_p\right.\\&\quad\quad\left.+\text{ higher orders of }\xi : u\in \Gamma(M,L) \right\}\\
        &=\left\{u_p+\varphi\left(u_p+[\bar{\xi},u]_p+\pr_L[\xi,u]_p+\pr_L[\bar{\xi},\iota_u\varphi]_p\right)-\pr_{\bar{L}}[\xi,u]_p-\pr_{\bar{L}}[\xi+\bar{\xi},\iota_u\varphi]_p\right.\\&\quad\quad\left.+\text{ higher orders of }\xi:u\in\Gamma(M,L)\right\}.
    \end{align*}
   This shows that
    \begin{align*}
        \mathcal{E}(\xi,\varphi)(u)&=\iota_{u+[\bar{\xi},u]+\pr_L[\xi,u]+\pr_L[\bar{\xi},\iota_u\varphi]}\varphi-\pr_{\bar{L}}[\xi,u]-\pr_{\bar{L}}[\xi+\bar{\xi},\iota_u\varphi]+\text{ higher orders of }\xi.
    \end{align*}

    On the other hand, using Lemma~\ref{lem: different formulas for dL and brackets}, we find that 
\begin{align*}
    &\iota_u\left(\varphi+\dd_L(\xi-\iota_{\bar{\xi}}\varphi )+[\varphi,\xi-\iota_{\bar{\xi}}\varphi]-\iota_{\bar{\xi}}(\dd_L\varphi+\tfrac{1}{2} [\varphi,\varphi])\right)\\
    &=\iota_u\varphi -\pr_{\bar{L}}[\xi,u]+\pr_{\bar{L}}[\iota_{\bar{\xi}}\varphi,u]+[\iota_u\varphi,\xi-\iota_{\bar{\xi}}\varphi]+\iota_{\pr_L[\xi-\iota_{\bar{\xi}}\varphi,u]}\varphi\\
    &\quad-\pr_{\bar{L}}[\iota_{\bar{\xi}}\varphi,u]-\pr_{\bar{L}}[\bar{\xi},\iota_u\varphi]+\iota_{[\bar{\xi},u]}\varphi-[\iota_{\bar{\xi}}\varphi,\iota_u\varphi]\\
    &\quad+\tfrac{1}{2}\,\iota_{\pr_L([\bar{\xi},\iota_u\varphi]-[\iota_u\varphi,\bar{\xi}]-[u,\iota_{\bar{\xi}}\varphi]+[\iota_{\bar{\xi}}\varphi,u])}\varphi\\
    &=\iota_{u+[\bar{\xi},u]+\pr_L[\xi,u]}\varphi-\pr_{\bar{L}}[\xi,u]-\pr_{\bar{L}}[\xi+\bar{\xi},\iota_u\varphi]\\
    &\quad+
    \tfrac{1}{2}\,\iota_{\pr_L([\bar{\xi},\iota_u\varphi]-[\iota_u\varphi,\bar{\xi}]-[u,\iota_{\bar{\xi}}\varphi]-[\iota_{\bar{\xi}}\varphi,u])}\varphi.    
\end{align*}
We need to show that
\begin{align*}
    -\pr_L([\iota_u\varphi,\bar{\xi}]+[u,\iota_{\bar{\xi}}\varphi]+[\iota_{\bar{\xi}}\varphi,u])=\pr_L[\bar{\xi},\iota_u\varphi].
\end{align*}
For this, for some arbitrary $\psi\in \bar{L}$, we compute
\begin{align*}
    &\langle [\bar{\xi},\iota_u\varphi]+[\iota_u\varphi,\bar{\xi}],\psi\rangle+\langle[u,\iota_{\bar{\xi}}\varphi]+[\iota_{\bar{\xi}}\varphi,u],\psi\rangle\\
    &=\pi(\psi)\langle \bar{\xi},\iota_u\varphi\rangle+\pi(\psi)\langle u,\iota_{\bar{\xi}}\varphi\rangle\\
    &=\pi(\psi)\left( \varphi(u,\bar{\xi})+\varphi(\bar{\xi},u)\right)=0.
\end{align*}
This concludes the proof.
\end{proof}

\begin{lemma}
\label{lem: properties a_epsilon}
    For some $\varepsilon>0$, $\varphi\in\mathcal{W}^{(0)}_{\varepsilon}$, let
    \begin{align*}
        a_{\varphi}\colon \Gamma(M,L^*)&\longrightarrow \Gamma(M,L^*)\\
        \xi &\longmapsto \xi-\iota_{\bar{\xi}}\varphi.
    \end{align*}
    Then the family of linear maps 
    \begin{align*}
        a\colon \mathcal{W}^{(0)}_{\varepsilon}\times \Gamma(M,L^*)&\longrightarrow \Gamma(M,L^*)\\
        (\varphi,\xi)&\longmapsto a_{\varphi}\xi
    \end{align*}
    is smooth tame. Moreover, for $\varepsilon>0$ small enough, $a_{\varphi}$ is invertible for every $\varphi\in \mathcal{W}^{(0)}_{\varepsilon}$ and the family of inverses $a_{\varphi}^{-1}$ is smooth tame.
\end{lemma}
\begin{proof}
    Note that $a$ is a non-linear vector bundle operator and therefore smooth tame by~\cite[Thm II.2.2.6]{hamilton82} (see also~\cite[Prop 3.16]{paperoncomplexdefs}). Moreover, observe that $a_{\varphi}$ is invertible for $\Vert \varphi\Vert_{0,\infty}$ small enough. We define $\tilde{\varphi}\colon \bar{L}\rightarrow \bar{L}$ via $\tilde{\varphi}(\xi):=\varphi(\bar{\xi})$. Then
    \begin{align*}
        a_{\varphi}^{-1}=\sum_{k=0}^{\infty}\tilde{\varphi}^{\,k},
    \end{align*}
        where we used that the sum converges fiberwise for small enough $\Vert\varphi\Vert_{0,\infty}$.
    This shows that, when it is defined, the family $a_{\varphi}^{-1}\xi$ is a non-linear vector bundle operator and hence smooth tame.
\end{proof}
Using Proposition~\ref{prop: taylor E}, we will compute $D_{\xi}\mathcal{E}(\xi,\varphi)$ for $(\xi,\varphi)\in \mathcal{U}_{\varepsilon}^{(1)}\times \mathcal{W}_{\varepsilon}^{(0)}$.
\begin{proposition}
\label{prop: DE}
    Let $\varphi\in \mathcal{W}^{(0)}_{\varepsilon}$ be a generalized almost complex deformation. Then
    \begin{align*}
        D_{\xi}\mathcal{E}(\xi,\varphi)\eta=\dd_L(b_{(\xi,\varphi)}\eta)+[\mathcal{E}(\xi,\varphi),b_{(\xi,\varphi)}\eta]-\iota_{\overline{c_{\xi}\eta}}\left(\dd_L\mathcal{E}(\xi,\varphi)+\tfrac{1}{2}[\mathcal{E}(\xi,\varphi),\mathcal{E}(\xi,\varphi)]\right),
    \end{align*}
    where $b_{(\xi,\varphi)}\eta=a_{\mathcal{E}(\xi,\varphi)}c_{\xi}\eta$, with $a$ as defined in Proposition~\ref{prop: taylor E} and 
    \begin{align*}
    c_{\xi}\eta=DC_{\xi}(0)\eta,\quad C_{\xi}\colon \mathcal{U}\ni\tilde{\eta}\mapsto \Psi^{-1}\Phi\left(F_{\xi+\tilde{\eta}} F_{\xi}^{-1}\right)\in \Psi^{-1}\autex(E),
    \end{align*}
    for a sufficiently small $C^1$-neighborhood $\mathcal{U}\subset \Psi^{-1}\autex(E)\subset\Gamma(M,L^*)$ of zero.
\end{proposition}
\begin{proof}
    Let $\xi\in \mathcal{U}_{\varepsilon}^{(1)}$, $\varphi\in\mathcal{W}_{\varepsilon}^{(0)}$. Let $\mathcal{U}\subset \mathcal{U}_{\varepsilon}^{(1)}$ an open $C^1$-neighborhood of $0\in\Gamma(M,L^*)$ such that $\xi+\tilde{\eta}\in\mathcal{U}_{\varepsilon}^{(1)}$ for every $\tilde{\eta}\in\mathcal{U}$ and such that $F_{\xi+\tilde{\eta}}\circ F_{\xi}^{-1}\in \Phi^{-1}\circ\Psi(\mathcal{U}_{\varepsilon}^{(1)})$ for every $\tilde{\eta}\in\mathcal{U}$. Then we define
    \begin{align*}
        C_{\xi}\colon \mathcal{U}&\longrightarrow \mathcal{U}_{\varepsilon}^{(1)}\\
        \tilde{\eta}&\longmapsto \Psi^{-1}\circ\Phi(F_{\xi+\tilde{\eta}}\circ F_{\xi}^{-1}).
    \end{align*}
    Then
    \begin{align*}
        \mathcal{E}(\xi+\tilde{\eta},\varphi)&=F_{\xi+\tilde{\eta}}\cdot\varphi\\
        &=(F_{C_{\xi}(\tilde{\eta})}\circ F_{\xi})\cdot \varphi\\
        &=F_{C_{\xi}(\tilde{\eta})}\cdot \mathcal{E}(\xi,\varphi).
    \end{align*}
    By Proposition~\ref{prop: taylor E} and since $C_{\xi}(0)=0$, we find
    \begin{align*}
        D_{\xi}\mathcal{E}(\xi,\varphi)&=\lim_{t\to0} \frac{1}{t}\left\{ \mathcal{E}(\xi+t\eta,\varphi)-\mathcal{E}(\xi,\varphi)\right\}\\
        &=\lim_{t\to 0} \frac{1}{t}\cdot[\mathcal{E}\left(C_{\xi}(t\eta),\mathcal{E}(\xi,\varphi)\right)-\mathcal{E}(\xi,\varphi)]
        \\
        &=\lim_{t\to 0}\frac{1}{t}\left\{\dd_L\left(a_{\mathcal{E}(\xi,\varphi)}C_{\xi}(t\eta)\right)+[\mathcal{E}(\xi,\varphi),a_{\mathcal{E}(\xi,\varphi)}C_{\xi}(t\eta)]\right.\\
        &\qquad\quad \left.-\iota_{\overline{C_{\xi}(t\eta)}}\left(\dd_L\mathcal{E}(\xi,\varphi)+\tfrac{1}{2}[\mathcal{E}(\xi,\varphi),\mathcal{E}(\xi,\varphi)]\right)\right\}\\
        &=\dd_L (a_{\mathcal{E}(\xi,\varphi)}c_{\xi}\eta)+[\mathcal{E}(\xi,\varphi),a_{\mathcal{E}(\xi,\varphi)}c_{\xi}\eta]-\iota_{\overline{c_{\xi}\eta}}\left(\dd_L\mathcal{E}(\xi,\varphi)+\tfrac{1}{2}[\mathcal{E}(\xi,\varphi),\mathcal{E}(\xi,\varphi)]\right),
    \end{align*}
    where $c_{\xi}\eta=DC_{\xi}(0)\eta$. Setting $b_{(\xi,\varphi)}\eta:=a_{\mathcal{E}(\xi,\varphi)}c_{\xi}\eta$ we arrive at the expression stated above.
\end{proof}

\begin{lemma}
\label{lem: properties b and c}
    For some $\varepsilon>0$, let $\mathcal{W}^{(0)}_{\varepsilon}$, $\mathcal{U}^{(1)}_{\varepsilon}$ as above. We assume that $\varepsilon>0$ is small enough so that $c_{\xi}$ and $b_{(\xi,\varphi)}$ as in Proposition~\ref{prop: DE} are defined for $\xi\in\mathcal{U}^{(1)}_{\varepsilon}$ and $\varphi\in\mathcal{W}^{(0)}_{\varepsilon}$. Then the families of linear maps
    \begin{align*}
        c\colon \mathcal{U}^{(1)}_{\varepsilon}\times \Psi^{-1}\autex(E) &\longrightarrow \Psi^{-1}\autex(E)\\
        (\xi,\eta)&\longmapsto c_{\xi}\eta
    \end{align*}
    and 
    \begin{align*}
        b\colon  \mathcal{U}^{(1)}_{\varepsilon}\times\mathcal{W}^{(0)}_{\varepsilon}\times\Psi^{-1}\autex(E) &\longrightarrow \Gamma(M,L^*)\\
        (\xi,\varphi,\eta)&\longmapsto b_{(\xi,\varphi)}\eta
    \end{align*}
    are smooth tame. Moreover, for $\varepsilon>0$ small enough, $c_{\xi}$ is invertible for every $\xi\in \mathcal{U}^{(1)}_{\varepsilon}$. The family of inverses $c_{\xi}^{-1}$ is smooth tame. 
\end{lemma}
\begin{proof}
    Recall that
    \begin{align*}
    c_{\xi}\eta=DC_{\xi}(0)\eta,\quad C_{\xi}\colon \mathcal{U}\ni\tilde{\eta}\mapsto \left(\Phi^{-1}\circ \Psi\right)^{-1}\left(F_{\xi+\tilde{\eta}} F_{\xi}^{-1}\right)\in \Psi^{-1}\autex(E),
    \end{align*}
   where $\mathcal{U}\subset \mathcal{U}_{\varepsilon}^{(1)}$ is an open neighborhood of $0$ such that $\xi+\tilde{\eta}\in\mathcal{U}_{\varepsilon}^{(1)}$ for every $\tilde{\eta}\in\mathcal{U}$ and $F_{\xi+\tilde{\eta}}\circ F_{\xi}^{-1}\in \Phi^{-1}(\mathcal{U}_{\varepsilon}^{(1)})$. Note that $C_{\xi}$ can be extended to a $C^0$-neighborhood $\tilde{\mathcal{U}}$ of $0$ such that $\xi+\tilde{\eta}\in\Psi^{-1}\circ\Phi(\mathcal{V}_{\id})$, and $F_{\xi+\tilde{\eta}}F_{\xi}^{-1}\in\mathcal{V}_{\id}$ for every $\tilde{\eta}\in\tilde{\mathcal{U}}$, where $\mathcal{V}_{\id}$ denotes the local chart neighborhood of $\id\in\Autex(E)$. We may furthermore assume, possibly after shrinking $\varepsilon>0$, that $\tilde{\mathcal{U}}$ can be chosen independently of $\xi$, such that we obtain a well-defined map
   \begin{align*}
       C\colon \mathcal{U}_{\varepsilon}^{(1)}\times \tilde{\mathcal{U}}&\longrightarrow\Psi^{-1}\autex(E)\\
       (\xi,\tilde{\eta})&\longmapsto C_{\xi}(\tilde{\eta}).
   \end{align*}
   Then $C_{\xi}(\tilde{\eta})$ is smooth tame in $\xi$ and $\tilde{\eta}$ as a composition of smooth tame maps and invertible for every small enough $\xi$. To compute the inverse, we note that
   \begin{align*}
       C_{\xi}(\tilde{\eta})=(\Phi^{-1}\circ \Psi)^{-1}\left(R_{F_{\xi}^{-1}}\circ\Phi^{-1}\circ \Psi(\xi+\tilde{\eta})\right).
   \end{align*}
   Since $R_{F_{\xi}^{-1}}^{-1}=R_{F_{\xi}}$, we find that
    \begin{align*}
        C_{\xi}^{-1}(\chi)&=(\Phi^{-1}\circ\Psi)^{-1}\circ R_{F_{\xi}}\circ \Phi^{-1}\circ \Psi(\chi)-\xi
        \\&=(\Phi^{-1}\circ \Psi)^{-1}\left(F_{\chi}F_{\xi}\right)-\xi,
    \end{align*}
    which is smooth tame. It follows that the family $c$ of linear maps is smooth tame and $c_{\xi}$ is invertible for small enough $\xi$ and the family of inverses $c_{\xi}^{-1}$ is again smooth tame. 

    For the statements about $b$, we observe that by Proposition~\ref{prop:Theta smooth tame} and since $c$ is smooth tame, $b$ is a composition of smooth tame maps and hence a smooth tame family of linear maps.
\end{proof}

\section{The deformation theorem}
\label{defothm:sec}

We now state the main theorem of this work.
\begin{theorem}
\label{thm: deformation theorem}
Let $L$ be a generalized complex structure on a $\mathcal{G}$-flat transitive Courant algebroid $E\rightarrow M$. There exists an open neighborhood $\mathcal{W}\subset H^2(M,L)$ of $0$, a family $\tilde{\mathcal{M}}=\{\varphi_t:t\in \mathcal{W}\}$ of generalized almost complex
deformations of $L$, and an analytic obstruction map 
\begin{equation*}
\phi\colon\mathcal{W}\rightarrow H^3(M,L)  \end{equation*}
such that the generalized almost complex deformations $\mathcal{M}:=\{\varphi_t\in\tilde{\mathcal{M}}:\phi(t)=0\}$ are precisely the integrable ones. Any sufficiently small generalized complex deformation of $L$ is equivalent to at least one member
of the family $\mathcal{M}$. In the case that the obstruction map vanishes, $\mathcal{M}$ is a locally complete family of generalized complex deformations.    
\end{theorem}

    We will call the family $\mathcal{M}$ in Theorem~\ref{thm: deformation theorem} the \emph{Kuranishi family} and the set $\phi^{-1}(0)\subset\mathcal{W}$ the \emph{Kuranishi moduli space}.

The proof of Theorem~\ref{thm: deformation theorem} is split into two parts. First, we will construct the family and the obstruction map in Section~\ref{descr_fam:sec}. Then we will prove that the family is locally complete in Section~\ref{sec: local cplt gc}. In Section~\ref{descr_fam:sec}, we will closely follow Gualtieri's proof in~\cite{gualtieri2011}. Since Gualtieri followed Kuranishi's proof in~\cite{kuranishi64}, there are similarities with this work, too, as well as with~\cite{paperoncomplexdefs}, which set the foundation for the proofs in Section~\ref{sec: local cplt gc}.

\subsection{Description of the family}
\label{descr_fam:sec}
Let
\begin{align*}
    \tilde{\mathcal{V}}_1:=\left\{\varphi\in\Gamma(M,\Lambda^2L^*):\dd_L\varphi+\tfrac{1}{2}[\varphi,\varphi]=0=\dd_L^*\varphi\right\}\subset \Gamma(M,\Lambda^2L^*).
\end{align*}
We will see that as a set, the generalized complex deformations $\mathcal{M}$ in the family $\tilde{\mathcal{M}}$ are contained in $\tilde{\mathcal{V}}_1$. Our strategy will be the same as in~\cite{gualtieri2011}: First we show that $\tilde{\mathcal{V}}_1$ is contained in another set $\tilde{\mathcal{V}}_2$ and that for small enough $\varepsilon>0$ and large enough $k\in \N_0$, the intersection $\mathcal{W}_{\varepsilon}^{(k)}\cap\tilde{\mathcal{V}}_2$ is parametrized by a neighborhood of 0 in $\mathcal{H}_L^2(M)$. Then we will find an analytic map $$\phi\colon\mathcal{W}_{\varepsilon}^{(k)}\cap\tilde{\mathcal{V}}_2\longrightarrow \mathcal{H}_L^{3}(M)$$ such that $\phi(\varphi)=0$ if and only if $\varphi$ is integrable and hence contained in $\mathcal{M}$.

In the following lemma, we define the set $\tilde{\mathcal{V}}_2$ and prove that it contains the set $\tilde{\mathcal{V}}_1$ as defined above.
\begin{lemma}[\cite{Gualtieri:2003dx,gualtieri2011}]
    $\tilde{\mathcal{V}}_1\subset\tilde{\mathcal{V}}_2:=\left\{\varphi\in\Gamma(M,\Lambda^2 L^*):\varphi+\tfrac{1}{2}\Q_L[\varphi,\varphi]\in\mathcal{H}^2_L(M)\right\}$.
\end{lemma}
\begin{proof}
    This exact statement has already been proved in~\cite{Gualtieri:2003dx,gualtieri2011}, but for the sake of completeness we will repeat the argument here. Let $\varphi\in\tilde{\mathcal{V}}_1$. Then
    \begin{align*}
        \Delta_L\varphi=(\dd_L\dd_L^*+\dd_L^*\dd_L)\varphi=\dd_L^*\dd_L\varphi=-\tfrac{1}{2}\dd_L^*[\varphi,\varphi],
    \end{align*}
    where we used $\dd_L^*\varphi=0$ for the second equality and $\dd_L\varphi+\tfrac{1}{2}[\varphi,\varphi]=0$ for the last equality. Hence,
    \begin{align*}
        0=\G_L(\Delta_L\varphi+\tfrac{1}{2}\dd_L^*[\varphi,\varphi])=\varphi-\HprojL\varphi+\tfrac{1}{2}\Q_L[\varphi,\varphi],
    \end{align*}
    where we used that $\HprojL+\G_L\Delta_L=\id$. It follows that
    \begin{align*}
        \varphi+\tfrac{1}{2}\Q_L[\varphi,\varphi]=\HprojL\varphi\in\mathcal{H}_L^2(M)
    \end{align*}
    and thus $\varphi\in \tilde{\mathcal{V}}_2$.
\end{proof}

We consider the smooth tame map
\begin{align*}
    \sigma\colon \Gamma(M,\Lambda^2L^*)&\longrightarrow \Gamma(M,\Lambda^2L^*),\\
    \varphi &\longmapsto \varphi+\tfrac{1}{2}\Q_L[\varphi,\varphi].
\end{align*}
\begin{proposition}
\label{prop: sigma invertible with smooth tame inverse}
    For $\varepsilon>0$ small enough and large enough $k\in \N_0$, $\sigma$ maps $\mathcal{W}_{\varepsilon}^{(k)}$ bijectively onto some neighborhood $\tilde{\mathcal{W}}_{\varepsilon}\subset \Gamma(M,\Lambda^2 L^*)$ such that 
    \begin{align*}
        \sigma^{-1}\colon \tilde{\mathcal{W}}_{\varepsilon}\longrightarrow \mathcal{W}_{\varepsilon}^{(k)}
    \end{align*}
    is smooth tame.
\end{proposition}
To prove Proposition~\ref{prop: sigma invertible with smooth tame inverse}, we need the following lemma:
\begin{lemma}
\label{lem: existence Vsigma}
    For $\varepsilon>0$ small enough and large enough $k\in\N_0$, the equation $D\sigma(\varphi)\alpha=\beta$ has a unique solution $\alpha$ for every $\varphi\in \mathcal{W}^{(k)}_{\varepsilon}$, $\beta\in \Gamma(M,\Lambda^2 L^*)$. Moreover, the family $V(\varphi)\beta=\alpha$ of solutions is a smooth tame family of linear maps.
\end{lemma}
\begin{proof}
    Let $\varphi\in \mathcal{W}^{(k)}_{\varepsilon}$, $\beta\in \Gamma(M,\Lambda^2 L^*)$. We need to find some $\alpha\in \Gamma(M,\Lambda^2 L^*)$ such that
\begin{align}
    \label{eq: surjective Dsigma}
\beta=\alpha-\Q_L[\alpha,\varphi]=\alpha-\dd_L^*\G_L[\alpha,\varphi].
\end{align} 
The statement follows now from Proposition~\ref{prop: inverse family from elliptic}, since $\alpha\mapsto\Delta_L\dd_L^*\G_L[\alpha,\varphi]=\dd_L^*[\alpha,\varphi]$ is a linear partial differential operator of order two.
\end{proof}

\begin{proof}[Proof of Proposition~\ref{prop: sigma invertible with smooth tame inverse}]
This follows from the Nash--Moser Theorem~\cite[Thm III.1.1.1]{hamilton82} (see also~\cite[Thm 3.18]{paperoncomplexdefs}) and Lemma~\ref{lem: existence Vsigma}.    
\end{proof}

Set
\begin{align*}
 \mathcal{W}_{\varepsilon}:=\tilde{\mathcal{W}}_{\varepsilon}\cap\mathcal{H}^2_L(M),\quad \mathcal{V}_{\varepsilon}:=\tilde{\mathcal{V}}_2\cap\mathcal{W}_{\varepsilon}^{(k)}.   
\end{align*}
Then $\sigma$ yields a smooth tame map
\begin{align*}
    \varphi:= \sigma^{-1}|_{\mathcal{W}_{\varepsilon}}\colon \mathcal{W}_{\varepsilon}\longrightarrow \mathcal{V}_{\varepsilon}.
\end{align*}

\begin{corollary}
\label{cor: family of defs}
The family $\{\varphi(t):t\in\mathcal{W}_{\varepsilon}\}$ is a smooth family of generalized almost complex deformations.    
\end{corollary}
To identify which elements in the family are integrable, i.e.~contained in $\tilde{\mathcal{V}}_1$, we note the following:
\begin{lemma}[\cite{Gualtieri:2003dx,gualtieri2011}]
For $\varepsilon>0$ small enough, $\varphi\in \mathcal{V}_{\varepsilon}$ is in $\mathcal{S}_{\varepsilon}:=\tilde{\mathcal{V}}_1\cap \mathcal{W}_{\varepsilon}^{(k)}$ if and only if
\begin{align*}
    \phi(\varphi):=\HprojL[\varphi,\varphi]=0.
\end{align*}
\end{lemma}
\begin{proof} This exact statement has already been proved in~\cite{Gualtieri:2003dx,gualtieri2011}, but for the sake of completeness we repeat the argument here.
    We first note that for any $\varphi\in \mathcal{V}_{\varepsilon}$, we have 
    \begin{align*}
        \dd_L\varphi+\tfrac{1}{2}[\varphi,\varphi]=&-\tfrac{1}{2}\dd_L \Q_L[\varphi,\varphi]+\tfrac{1}{2}[\varphi,\varphi]\\
        =&\tfrac{1}{2}\Q_L \dd_L [\varphi,\varphi]+\tfrac{1}{2}\HprojL[\varphi,\varphi],
    \end{align*}
    where we used $\varphi\in \mathcal{V}_{\varepsilon}$ for the first equality and $\HprojL+\dd_L\Q_L+\Q_L\dd_L=\id$ for the second equality. Since the images of $\Q_L$ and $\HprojL$ are orthogonal to each other, it follows that 
    \begin{align}
    \label{eq: integrability}    \dd_L\varphi+\tfrac{1}{2}[\varphi,\varphi]=0\quad \iff \quad \Q_L\dd_L[\varphi,\varphi]=0\text{ and }\HprojL[\varphi,\varphi]=0.
    \end{align}
    We claim that for $\varepsilon>0$ small enough, $\HprojL [\varphi,\varphi]=0$ already implies $\Q_L\dd_L[\varphi,\varphi]=0$. To see this, suppose $\varphi\in\mathcal{V}_{\varepsilon}$ such that $\HprojL[\varphi,\varphi]=0$. Then,
    \begin{align*}
        \dd_L\Q_L[\varphi,\varphi]=-\Q_L\dd_L[\varphi,\varphi]+[\varphi,\varphi].
    \end{align*}
    Using the compatibility of the bracket $[\cdot,\cdot]$ and $\dd_L$ as stated in Lemma~\ref{lem: bracket and dL} and $\varphi\in \tilde{\mathcal{V}}_2$, we compute
    \begin{align*}
        \Q_L\dd_L[\varphi,\varphi]=2\Q_L[\dd_L\varphi,\varphi]=-\Q_L[\dd_L\Q_L[\varphi,\varphi],\varphi]=\Q_L[\Q_L\dd_L[\varphi,\varphi],\varphi]-\Q_L[[\varphi,\varphi],\varphi].
    \end{align*}
    By the Jacobi identity (see Lemma~\ref{lem: bracket and dL}), we must have $\left[[\varphi,\varphi],\varphi\right]=0$. Therefore, setting $\eta:=\Q_L\dd_L[\varphi,\varphi]$, we find
    \begin{align*}
        \eta=\Q_L[\eta,\varphi].
    \end{align*}
    It follows from the discussion in Section~\ref{appendix HNM} of the appendix that for $k\geq \dim M/2+1$ we have
    \begin{align*}
        \Vert \eta\Vert_{k,2}=\Vert \Q_L[\eta,\varphi]\Vert_{k,2}\leq C\Vert\eta\Vert_{k,2}\Vert\varphi\Vert_{k,2}\leq \tilde{C}\varepsilon\Vert\eta\Vert_{k,2},
    \end{align*}
    where $\|\cdot\|_{k,2}$ denotes the $W^{k,2}$-Sobolev norm.
    In particular, for $\varepsilon <1/\tilde{C}$, this implies that $\eta=0$ and hence $\Q_L\dd_L[\varphi,\varphi]=0$. By equation~\eqref{eq: integrability} this shows that $\varphi\in\mathcal{S}_{\varepsilon}$.
\end{proof}

\subsection{Local completeness of the family}\label{sec: local cplt gc}
In the last step of the proof of Theorem \ref{thm: deformation theorem}, we need to show that every sufficiently small generalized complex deformation $\varphi$ can be obtained from $\mathcal{M}$ via an (exact) autoequivalence, i.e.\ there is some small $\xi\in \Gamma(M,L^*)$ such that $\dd_L^*\mathcal{E}(\xi,\varphi)=0$. Moreover, given a smooth family of small generalized (almost) complex deformations $\varphi_t$, we need to show that there is a family $\{\xi_t\}\subset \Gamma(M,L^*)$ depending smoothly on $t$ such that $\dd_L^*\mathcal{E}(\xi_t,\varphi_t)=0$.

Recall that by Lemma~\ref{lem: autex contains coexact} we have
$$\Psi^{-1}\autex(E)\supset\dd_L^* \left(\Gamma(M,\Lambda^2 L^*)\right).$$
\begin{proposition}
\label{prop: locally complete}
    There are neighborhoods $\mathcal{W}\subset \Gamma(M,\Lambda^2 L^*)$ of $0\in \Gamma(M,\Lambda^2 L^*)$, $\mathcal{U}\subset \Gamma(M,L^*)$ of $0\in\Gamma(M,L^*)$ such that for every $\varphi \in \mathcal{W}$ there is a smooth tame family of solutions $\xi=\xi(\varphi)\in \mathcal{U}$ of the equation
    \begin{align*}
        \dd_L^*\mathcal{E}(\xi,\varphi)=0.
    \end{align*}
\end{proposition}
For the proof of Proposition~\ref{prop: locally complete} we wish to apply Hamilton's Nash--Moser implicit function theorem~\cite[Thm III.3.3.1]{hamilton82}. For some $\varepsilon>0$, $k\in\N$, let ${\mathcal{W}}^{(k)}_{\varepsilon}$ and $\mathcal{U}^{(k)}_{\varepsilon}$ be as above. For $\varepsilon$ small enough and large enough $k$ the map
\begin{align*}
    A\colon \mathcal{U}^{(k)}_{\varepsilon}\times{\mathcal{W}}^{(k)}_{\varepsilon}&\longrightarrow \dd_L^* \left(\Gamma(M,\Lambda^2 L^*)\right)\\
    (\xi,\varphi)&\longmapsto \Q_L\mathcal{E}(\xi,\varphi)
\end{align*}
is well-defined and smooth tame as a composition of smooth tame maps. We observe that $\dd_L^*\mathcal{E}(\xi,\varphi)=0$ if and only if $A(\xi,\varphi)=0$. Therefore, our strategy will be to apply Hamilton's Nash--Moser implicit function theorem to solve the equation $A(\xi,\varphi)=0$. To do so, we will need the following two lemmas proved below.

\begin{lemma}\label{lem: aux lemma for invertibility of DA}
    For $\varepsilon>0$ small enough, the restriction 
    \begin{align*}
        \Q_L\dd_L a_{\mathcal{E}(\xi,\varphi)}\colon \dd_L^* \left(\Gamma(M,\Lambda^2 L^*)\right)\longrightarrow\dd_L^* \left(\Gamma(M,\Lambda^2 L^*)\right)
    \end{align*}
    is invertible for every $(\xi,\varphi)\in \mathcal{U}^{(k)}_{\varepsilon}\times \mathcal{W}^{(k)}_{\varepsilon}$. The family of inverse maps $p_{(\xi,\varphi)}$ is smooth tame.
\end{lemma}
\begin{proof}
    Let $\eta\in \dd_L^* \left(\Gamma(M,\Lambda^2 L^*)\right)$. Then
    \begin{align*}
        \Q_L\dd_L a_{\mathcal{E}(\xi,\varphi)}\eta=\Q_L\dd_L\eta-\Q_L\dd_L\iota_{\bar{\eta}}\mathcal{E}(\xi,\varphi)=\eta-\Q_L\dd_L\iota_{\bar{\eta}}\mathcal{E}(\xi,\varphi).
    \end{align*}
   Since
    \begin{align*}
        \Delta_L\Q_L\dd_L\iota_{\bar{\eta}}\mathcal{E}(\xi,\varphi)=\dd_L^*\dd_L\iota_{\bar{\eta}}\mathcal{E}(\xi,\varphi),
    \end{align*}
    for $\eta,\tilde{\beta}\in \dd_L^* \left(\Gamma(M,\Lambda^2 L^*)\right)$, the equation has a unique solution by Proposition~\ref{prop: inverse family from elliptic} and the solution map is a smooth tame family of linear maps.
\end{proof}
\begin{lemma}
\label{lem: DA invertible, locally complete}
    For $\varepsilon>0$ small enough and large enough $k$, the equation $D_{\xi}A(\xi,\varphi)\alpha=\beta$ has a solution $\alpha$ for every $(\xi,\varphi)\in \mathcal{U}_{\varepsilon}^{(k)}\times {\mathcal{W}}_{\varepsilon}^{(k)}$, $\beta\in \dd_L^* \left(\Gamma(M,\Lambda^2 L^*)\right)$. Moreover, there is a smooth tame family of linear maps $V(\xi,\varphi)\colon \dd_L^* \left(\Gamma(M,\Lambda^2 L^*)\right)\rightarrow \Psi^{-1}\autex(E)$ such that $D_{\xi}A(\xi,\varphi)V(\xi,\varphi)\beta=\beta$ for every $\beta\in \dd_L^* \left(\Gamma(M,\Lambda^2 L^*)\right)$.
\end{lemma}
\begin{proof}
Using Proposition~\ref{prop: DE}, for $(\xi,\varphi)\in \mathcal{U}_{\varepsilon}^{(k)}\times {\mathcal{W}}_{\varepsilon}^{(k)}$, $\alpha\in \Psi^{-1}\autex(E)$ we compute
\begin{align*}
    D_{\xi}A(\xi,\varphi)\alpha=&\Q_L D_{\xi}\mathcal{E}(\xi,\varphi)\alpha
    \\=&\Q_L\dd_L(b_{(\xi,\varphi)}\alpha)+\Q_L[\mathcal{E}(\xi,\varphi),b_{(\xi,\varphi)}\alpha]-\Q_L i_{\overline{c_{\xi}\alpha}}\left(\dd_L\mathcal{E}(\xi,\varphi)+\tfrac{1}{2}[\mathcal{E}(\xi,\varphi),\mathcal{E}(\xi,\varphi)]\right),
\end{align*}    
where as before $b_{(\xi,\varphi)}\alpha=a_{\mathcal{E}(\xi,\varphi)}c_{\xi}\alpha=c_{\xi}\alpha-\iota_{\overline{c_{\xi}\alpha}}\mathcal{E}(\xi,\varphi)$.

Now suppose that $c_{\xi}\alpha\in \dd_L^* \left(\Gamma(M,\Lambda^2 L^*)\right)$  and let $\eta\in \dd_L^* \left(\Gamma(M,\Lambda^2 L^*)\right)$ be such that
\begin{align*}
    \eta=\Q_L\dd_L a_{\mathcal{E}(\xi,\varphi)}c_{\xi}\alpha,
\end{align*}
i.e.~$\alpha=c_{\xi}^{-1}p_{(\xi,\varphi)}\eta$, where $p_{(\xi,\varphi)}$ is as in Lemma~\ref{lem: aux lemma for invertibility of DA}. If we can find $\eta\in \dd_L^* \left(\Gamma(M,\Lambda^2 L^*)\right)$ such that
\begin{align*}
    \eta+\Q_L[\mathcal{E}(\xi,\varphi),b_{(\xi,\varphi)}c_{\xi}^{-1}p_{(\xi,\varphi)}\eta]-\Q_L\iota_{\overline{p_{(\xi,\varphi)}\eta}}\left(\dd_L\mathcal{E}(\xi,\varphi)+\tfrac{1}{2}[\mathcal{E}(\xi,\varphi),\mathcal{E}(\xi,\varphi)]\right)=\beta
\end{align*}
this will provide us with a solution to $D_{\xi}A(\xi,\varphi)\alpha=\beta$. We observe by Proposition~\ref{prop: inverse family from elliptic}, that this equation
has a unique solution $\eta\in \dd_L^* \left(\Gamma(M,\Lambda^2 L^*)\right)$ for every $\beta\in \dd_L^* \left(\Gamma(M,\Lambda^2 L^*)\right)$, provided that $(\xi,\varphi)\in \mathcal{U}_{\varepsilon}^{(k)}\times \mathcal{W}_{\varepsilon}^{(k)}$ for small enough $\varepsilon>0$ and large enough $k$, and the solution map $s_{(\xi,\varphi)}\colon \dd_L^* \left(\Gamma(M,\Lambda^2 L^*)\right)\rightarrow \dd_L^* \left(\Gamma(M,\Lambda^2 L^*)\right)$ is smooth tame. This gives rise to a smooth tame family  of linear maps
\begin{align*}
    V(\xi,\varphi):=c_{\xi}^{-1}p_{(\xi,\varphi)}s_{(\xi,\varphi)}\colon \dd_L^* \left(\Gamma(M,\Lambda^2 L^*)\right)\longrightarrow \Psi^{-1}\autex(E),
\end{align*}
such that $D_{\xi}A(\xi,\varphi)V(\xi,\varphi)\beta=\beta$ for every $\beta\in \dd_L^* \left(\Gamma(M,\Lambda^2 L^*)\right)$.
\end{proof}

\begin{proof}[Proof of Proposition~\ref{prop: locally complete}]
By Lemma~\ref{lem: DA invertible, locally complete}, the assumptions of~\cite[Thm III.3.3.1]{hamilton82} (see also~\cite[Thm 3.20]{paperoncomplexdefs}) hold. Since $A(0,0)=0$, we conclude that there is a neighborhood $\mathcal{W}$ of $0$ in $\Gamma(M,\Lambda^2 L^*)$ and a neighborhood $\mathcal{U}$ of $0$ in $\Psi^{-1}\autex(E)$ such that for every $\varphi\in\mathcal{W}$ we can find some $\xi=\xi(\varphi)\in\mathcal{U}$ such that $\dd_L^*\mathcal{E}(\xi,\varphi)=0$. Moreover, the solution $\xi(\varphi)$ is defined by a smooth tame map. This finishes the proof of Proposition~\ref{prop: locally complete}.
\end{proof}

\begin{proof}[Proof of Theorem~\ref{thm: deformation theorem}]
By Corollary~\ref{cor: family of defs}, the family $\tilde{\mathcal{M}}=\{ \varphi(t): t\in \mathcal{W}_{\varepsilon}\}$ is a smooth family of generalized almost complex deformations. The integrable deformations are precisely those for which the analytic map $\phi(\varphi):=\HprojL[\varphi,\varphi]$ vanishes. They are contained in a neighborhood of 0 in $\tilde{\mathcal{V}}_1$. Moreover, by Proposition~\ref{prop: locally complete}, we know that any sufficiently small generalized complex deformation is equivalent to a generalized complex deformation in a neighborhood of 0 of $\tilde{\mathcal{V}}_1$. 

Furthermore, if $\phi\equiv 0$, then 
$$\mathcal{M}=\{ \varphi(t)\in\tilde{\mathcal{M}}:  \phi(\varphi(t))=0\}=\tilde{\mathcal{M}}$$
 is a smooth family of generalized complex deformations. Given any other smooth family of generalized complex deformations $\{\varphi'(t'):t'\in \mathcal{W}'
\}$, it follows again from Proposition~\ref{prop: locally complete} that (possibly after shrinking $\mathcal{W}'$) there is a family of exact autoequivalences $\{F_{t'}\}_{t'\in \mathcal{W}'}$, depending smoothly on $t'$, and a smooth map $\tau\colon \mathcal{W}'\rightarrow \mathcal{W}_{\varepsilon}$ such that for every $t'\in \mathcal{W}'$, we have
\begin{align*}
    F_{t'}\cdot\varphi'(t') =\varphi(\tau(t')).
\end{align*}
Hence, the family is locally complete.
\end{proof}

\section{Deformations of a class of generalized complex structures}\label{example of deformation}
Let $E=TM\oplus \mathcal{G}\oplus T^*M$ be a standard $\mathcal{G}$-flat Courant algebroid over a compact complex manifold $(M,J)$. We will furthermore assume that $H=0$. 

Let $I\in \End(\mathcal{G})$ such that $I^2=-\id$ and $$\langle I\cdot,\cdot\rangle_{\mathcal{G}}=-\langle \cdot,I\cdot\rangle_{\mathcal{G}}.$$ In particular, the $\pm\ii$-eigenbundles $\mathcal{G}^{1,0}$ and $\mathcal{G}^{0,1}$ are isotropic for the $\C$-linear extension of $\langle\cdot,\cdot\rangle_{\mathcal{G}}$. Moreover, suppose that $\nabla I=0$ for the flat connection $\nabla$ on $\mathcal{G}$ appearing in the Dorfman bracket of $E$. In particular $$\nabla_X\mathbf{r}\in \Gamma(M,\mathcal{G}^{1,0})\quad\forall \;X\in \X(M),\, \mathbf{r}\in \Gamma(M,\mathcal{G}^{1,0})$$
and similarly for $\mathcal{G}^{0,1}$.

Then we consider 
\begin{equation}\label{eq: example GC}
\mathcal{J}=\begin{pmatrix}
    -J &0 &0 \\
    0&-I&0\\
    0&0&J^*
\end{pmatrix}\in\End(E).    
\end{equation}
We claim that this defines a generalized complex structure. Clearly, it squares to $-\id$ and is skew with respect to $\langle\cdot,\cdot\rangle$. To see that it is integrable, note that the $\ii$-eigenbundle is given by
$$L=T^{0,1}M\oplus\mathcal{G}^{0,1}\oplus \Lambda^{1,0}T^*M.$$
Given $X+\mathbf{r}+\alpha\in\Gamma(M,L)$ and $Y+\mathbf{s}+\beta\in \Gamma(M,L)$, we compute 
\begin{align*}
    [X+\mathbf{r}+\alpha,Y+\mathbf{s}+\beta]&=\mathcal{L}_X Y+\nabla_X\mathbf{s}-\nabla_Y\mathbf{r}+\mathcal{L}_X\beta-\iota_Y\dd\alpha+2\langle\mathbf{s},\nabla\mathbf{r}\rangle_{\mathcal{G}}\\
    &=\mathcal{L}_X Y+\nabla_X\mathbf{s}-\nabla_Y\mathbf{r}+\mathcal{L}_X\beta-\iota_Y\dd\alpha.
\end{align*}
Hence $[X+\mathbf{r}+\alpha,Y+\mathbf{s}+\beta]\in\Gamma(M,L)$, using that $J$ is integrable and $\mathcal{G}^{0,1}$ is parallel.\footnote{More generally, the generalized almost complex structure in \eqref{eq: example GC} is integrable if and only if $J$ is integrable and $\nabla_X \mathbf{r}\in \Gamma(M,\mathcal{G}^{0,1})$ for $X\in \Gamma(M,T^{0,1}M)$ and $\mathbf{r}\in\Gamma(M,\mathcal{G}^{0,1})$.}

Next, we would like to decribe the Lie algebroid complex of $L$. First of all we observe that since $\nabla$ is flat and preserves $\mathcal{G}^{1,0}$, $\mathcal{G}^{1,0}$ is a holomorphic vector bundle with Dolbeault operator determined by $\nabla$. Since
$$L^*\cong\bar{L}=T^{1,0}M\oplus \mathcal{G}^{1,0}\oplus \Lambda^{0,1}T^*M,$$
we may identify $$\Gamma(M,\Lambda^kL^*)\cong\bigoplus_{p+q=k}\Omega^{0,p}(M,\Lambda^q(T^{1,0}M\oplus \mathcal{G}^{1,0})).$$
We claim that the Lie algebroid differential is given by the Dolbeault operator of $\Lambda^q(T^{1,0}M\oplus \mathcal{G}^{1,0})$. It is enough to check this on functions and sections of $\Gamma(M,L^*)$. For $f\in C^{\infty}(M,\C)$, $Y+\mathbf{s}+\beta\in \Gamma(M,L)$, we find
$$ \dd_L f(Y+\mathbf{s}+\beta)=Y[f]=\dbar f(Y).$$
For $X+\mathbf{r}+\alpha\in\Gamma(M,L^*)$ and $Y+\mathbf{s}+\beta$, $Z+\mathbf{t}+\gamma\in \Gamma(M,L)$, we have
\begin{align*}
    &\dd_L(X+\mathbf{r}+\alpha)(Y+\mathbf{s}+\beta,Z+\mathbf{t}+\gamma)\\&\quad=Y\langle X+\mathbf{r}+\alpha,Z+\mathbf{t}+\gamma\rangle-Z\langle X+\mathbf{r}+\alpha,Y+\mathbf{s}+\beta\rangle\\ &\qquad-\langle X+\mathbf{r}+\alpha,[Y+\mathbf{s}+\beta,Z+\mathbf{t}+\gamma]\rangle\\
    &\quad=Y\left(\frac{1}{2}(\iota_X\gamma+\iota_Z\alpha)+\langle \mathbf{r},\mathbf{t}\rangle_{\mathcal{G}}\right)-Z\left(\frac{1}{2}(\iota_X\beta+\iota_Y\alpha)+\langle\mathbf{r},\mathbf{s}\rangle_{\mathcal{G}}\right)\\ &\qquad -\langle X+\mathbf{r}+\alpha,\mathcal{L}_Y Z+\nabla_Y\mathbf{t}-\nabla_Z\mathbf{s}+\mathcal{L}_Y\gamma-\iota_Z\dd \beta\rangle\\
    &\quad=\frac{1}{2}\left(\mathcal{L}_Y\iota_X\gamma+\mathcal{L}_Y\iota_Z\alpha-\mathcal{L}_Z\iota_X\beta-\mathcal{L}_Z\iota_Y\alpha-\iota_X\mathcal{L}_Y\gamma+\iota_X\iota_Z\dd\beta-\iota_{\mathcal{L}_Y Z}\alpha\right)\\
    &\qquad +Y\langle \mathbf{r},\mathbf{t}\rangle_{\mathcal{G}}-Z\langle\mathbf{r},\mathbf{s}\rangle_{\mathcal{G}}-\langle \mathbf{r},\nabla_Y\mathbf{t}\rangle+\langle\mathbf{r},\nabla_Z\mathbf{s}\rangle_{\mathcal{G}}.
\end{align*}
Note that since $Z\in \Gamma(M,T^{0,1}M)$ and $\beta\in \Omega^{1,0}(M)$, we have
$$\iota_Z\dd\beta=\mathcal{L}_Z\beta-\dd\iota_Z\beta=\mathcal{L}_Z\beta.$$
Using this, as well as the identity $[\mathcal{L}_{X_1},\iota_{X_2}]=\iota_{\mathcal{L}_{X_1}X_2}$ and the fact that the scalar product on $\mathcal{G}$ is parallel for $\nabla$, we find
\begin{align*}
    \dd_L(X+\mathbf{r}+\alpha)(Y+\mathbf{s}+\beta,Z+\mathbf{t}+\gamma)&=\frac{1}{2}\left(\iota_{\mathcal{L}_Y X}\gamma-\iota_{\mathcal{L}_Z X}\beta\right)+\langle\nabla_Y\mathbf{r},\mathbf{t}\rangle_{\mathcal{G}}-\langle \nabla_Z\mathbf{r},\mathbf{s}\rangle_{\mathcal{G}}\\
    &=\langle\dbar_{T^{1,0}M}X(Y),\gamma\rangle-\langle \dbar_{T^{1,0}M} X(Z),\beta\rangle\\ &\quad +\langle\nabla\mathbf{r}(Y),\mathbf{t}\rangle-\langle\nabla\mathbf{r}(Z),\mathbf{s}\rangle\\
    &=(\dbar_{T^{1,0}M}X+\nabla\mathbf{r})(Y+\mathbf{s}+\beta,Z+\mathbf{t}+\gamma),
\end{align*}
where we interpreted the vector valued $(0,1)$-forms as sections of $\Lambda^2L^*$. 

Compared to the deformation space of the generalized complex structure on the generalized tangent bundle of a compact complex manifold discussed in~\cite[Sect.\ 5.3]{gualtieri2011}, which was presented as the sum of $H^0(M, \Lambda^2T^{1,0})$, 
$H^1(M,T^{1,0})$ and $H^2(M,\mathcal O)$, there are three additional components in $H^2(M,L)$: 
$$H^0(M,\Lambda^2\mathcal{G}^{1,0}),\quad H^0(M,T^{1,0}\wedge \mathcal{G}^{1,0}),\quad H^1(M,\mathcal{G}^{1,0}).$$  
We will consider deformations of $\mathcal{J}$ in a single one of such components, respectively, with all other components set to zero. To this end, we fix a parallel frame $\{\mathbf{e}_i\}$ of $\mathcal{G}^{1,0}$. Then $[\mathbf{e}_i,\mathbf{e}_j]=0$ and $\dd_L\mathbf{e}_i=\nabla\mathbf{e}_i=0$ for all $i,j$.

Consider now a deformation $\varphi\in \Gamma(M,\Lambda^2 \mathcal{G}^{1,0})$. Then $\varphi=\sum_{i<j}\varphi_{ij}\mathbf{e}_i\wedge \mathbf{e}_j$ for some $\varphi_{ij}\in C^{\infty}(M,\C)$. Since for all $f\in C^{\infty}(M,\C)$
$$[\mathbf{e}_i,f]=-[f,\mathbf{e}_i]=\pi(\mathbf{e}_i)[f]=0,$$
it follows automatically that $[\varphi,\varphi]=0$. Hence, $\varphi$ is integrable if and only if
\begin{align*}
    0=\dd_L\varphi=\sum_{i<j}\dbar\varphi_{ij}\,\mathbf{e}_i\wedge\mathbf{e}_j,
\end{align*}
i.e.\ if and only if $\dbar\varphi_{ij}=0$ for all $i,j$, i.e.\ $\varphi\in H^0(M,\Lambda^2\mathcal{G}^{1,0})$.

If $\varphi\in \Gamma(M,T^{1,0}M\wedge \mathcal{G}^{1,0})$ we may write $\varphi=\sum_i X_i\wedge \mathbf{e}_i$ for some $X_i\in\Gamma(M,T^{1,0}M)$. Then we compute
\begin{align*}
    \dd_L\varphi&=\sum_i(\dbar_{T^{1,0}M}X_i)\wedge \mathbf{e}_i\quad \in \Omega^{0,1}(M,T^{1,0}M\wedge \mathcal{G}^{1,0})\\
    [\varphi,\varphi]&=\sum_{i,j}[X_i\wedge \mathbf{e}_i,X_j\wedge \mathbf{e}_j]\\
    &=\sum_{i,j}[X_i,X_j]\wedge \mathbf{e}_i\wedge \mathbf{e}_j \quad \in \Omega^0(M,T^{1,0}M\wedge \mathcal{G}^{1,0}\wedge\mathcal{G}^{1,0}).
\end{align*}
Hence, $\varphi$ is integrable if and only if $$ \dbar_{T^{1,0}M}X_i=0,\quad \mathcal{L}_{X_i}X_j=0\qquad \forall\; i,j.$$

Finally, suppose $\varphi\in \Omega^{0,1}(M,\mathcal{G}^{1,0})$. Then $\varphi=\sum \varphi_i\wedge \mathbf{e}_i$ for some $\varphi_i\in \Omega^{0,1}(M)$. Then $[\varphi,\varphi]=0$ follows automatically and the Maurer--Cartan equation becomes
$$0=\dd_L\varphi=\sum_i(\dbar\varphi_i)\wedge \mathbf{e}_i,$$
i.e.\ $\varphi\in H^{1}(M,\mathcal{G}^{1,0})$.
\appendix

\section{The Lie algebroid differential and Schouten bracket}\label{appendix proof of lemma}
In this section, we will prove the identities for the Lie algebroid differential and the Schouten bracket stated in Lemma~\ref{lem: different formulas for dL and brackets} for the Lie bialgebroid of a generalized complex structure.

For the first equation, we plug in the definition of $\dd_L$ and compute
    \begin{align*}
        \dd_L\xi(u,v)&=\pi(u)\xi(v)-\pi(v)\xi(u)-\xi([u,v])\\
        &=\pi(u)\langle\xi,v\rangle-\pi(v)\langle\xi,u\rangle -\langle\xi,[u,v]\rangle\\
        &=\langle [u,\xi],v\rangle-\pi(v)\langle\xi,u\rangle\\&=-\langle[\xi,u],v\rangle,
    \end{align*}
    where we used the second and the third axiom in Definition~\ref{def: courant algebroid}. 
    
    Similarly, for the second identity, we compute
    \begin{align*}
        \dd_L\varphi(u,v,w)&=\pi(u)\left(\varphi(v,w)\right)-\pi(v)\left(\varphi(u,w)\right)+\pi(w)\left(\varphi(u,v)\right)\\
        &\quad -\varphi([u,v],w)+\varphi([u,w],v)-\varphi([v,w],u)\\
        &=\pi(u)\langle\iota_v\varphi,w\rangle-\pi(v)\langle\iota_u\varphi,w\rangle+\pi(w)\langle\iota_u\varphi,v\rangle\\
        &\quad -\langle\iota_{[u,v]}\varphi,w\rangle-\langle\iota_{v}\varphi,[u,w]\rangle+\langle\iota_{u}\varphi,[v,w]\rangle\\
        &=\langle [u,\iota_v\varphi]+[\iota_u\varphi,v]-\iota_{[u,v]}\varphi,w\rangle.
    \end{align*}

    Note that to show the third equation, we may assume that $\varphi=\eta\wedge\psi$ for some $\eta,\psi\in \Gamma(U,L^*)$. Using the definition of the Schouten bracket we compute
    \begin{align*}
        [\eta\wedge \psi,\xi](u,v)&=\left([\eta,\xi]\wedge \psi-[\psi,\xi]\wedge\eta\right)(u,v)\\
        &=\langle [\eta,\xi],u\rangle\langle\psi,v\rangle-\langle[\eta,\xi],v\rangle\langle\psi,u\rangle\\
        &\quad -\langle[\psi,\xi],u\rangle\langle\eta,v\rangle+\langle[\psi,\xi],v\rangle\langle\eta,u\rangle\\
        &=\left(-\pi(\xi)\langle\eta,u\rangle+\langle\eta,[\xi,u]\rangle\right)\langle\psi,v\rangle+\left(\pi(\xi)\langle\eta,v\rangle-\langle\eta,[\xi,v]\rangle\right)\langle\psi,u\rangle\\
        &\quad+\left(\pi(\xi)\langle\psi,u\rangle-\langle\psi,[\xi,u]\rangle\right)\langle\eta,v\rangle+\left(-\pi(\xi)\langle\psi,v\rangle+\langle\psi,[\xi,v]\rangle\right)\langle\eta,u\rangle\\
        &=\pi(\xi)\left(-\langle\eta,u\rangle\langle\psi,v\rangle+\langle\eta,v\rangle\langle\psi,u\rangle\right)+\eta\wedge\psi (\pr_L[\xi,u],v)-\eta\wedge\psi(\pr_L[\xi,v],u).
    \end{align*}
    Observing that
    \begin{align*}
        \langle[\iota_u(\eta\wedge\psi),\xi],v\rangle&=\left\langle \left[\langle\eta,u\rangle\psi-\langle\psi,u\rangle\eta,\xi\right],v\right\rangle\\
        &=-\pi(\xi)\left\langle\langle\eta,u\rangle\psi-\langle\psi,u\rangle\eta,v\right\rangle+\left\langle\langle\eta,u\rangle\psi-\langle\psi,u\rangle\eta,[\xi,v]\right\rangle\\
        &=\pi(\xi)\left(-\langle\eta,u\rangle\langle\psi,v\rangle+\langle\eta,v\rangle\langle\psi,u\rangle\right)-\eta\wedge\psi(\pr_L[\xi,v],u)
    \end{align*}
    we obtain the third identity.

    For the last identity we may again assume that $\varphi=\eta\wedge \psi$, $\varphi'=\eta'\wedge\psi'$ for $\eta,\eta',\psi,\psi'\in\Gamma(M,L^*)$. We start computing
    \begin{align*}
        &\left\langle \left[\iota_u(\eta\wedge \psi),\iota_v(\eta'\wedge\psi')\right],w\right\rangle\\
        &=\left\langle \left[\langle\eta,u\rangle\psi-\langle\psi,u\rangle\eta,\langle\eta',v\rangle\psi'-\langle\psi',v\rangle\eta'\right],w\right\rangle\\
        &=\pi(\psi)\left(\langle\eta',v\rangle\right)\langle\psi',w\rangle\langle\eta,u\rangle-\pi(\psi)\left(\langle\psi',v\rangle\right)\langle\eta,u\rangle\langle\eta',w\rangle\\
        &\quad-\pi(\eta)\left(\langle\eta',v\rangle\right)\langle\psi',w\rangle\langle\psi,u\rangle+\pi(\eta)\left(\langle\psi',v\rangle\right)\langle\eta',w\rangle\langle\psi,u\rangle\\
        &\quad-\pi(\psi')\left(\langle\eta,u\rangle\right)\langle\psi,w\rangle\langle\eta',v\rangle+\pi(\psi')\left(\langle\psi,u\rangle\right)\langle\eta,w\rangle\langle\eta',v\rangle\\
        &\quad+\pi(\eta')\left(\langle\eta,u\rangle\right)\langle\psi,w\rangle\langle\psi',v\rangle-\pi(\eta')\left(\langle\psi,u\rangle\right)\langle\eta,w\rangle\langle\psi',v\rangle\\
        &\quad+\langle \eta,u\rangle\langle\eta',v\rangle\left\langle [\psi,\psi'],w\right\rangle-\langle\eta,u\rangle\langle\psi',v\rangle\left\langle[\psi,\eta'],w\right\rangle\\
        &\quad -\langle\psi,u\rangle\langle\eta',v\rangle\left\langle[\eta,\psi'],w\right\rangle+\langle \psi,u\rangle\langle \psi',v\rangle\left\langle[\eta,\eta'],w\right\rangle,
    \end{align*}
    where we used the Leibniz rule for the bracket (Proposition~\ref{prop: extra properties of dorfman bracket}) in combination with the skew-symmetry of the bracket when restricted to sections of the isotropic bundle $\bar L$. 
    Using the second axiom in Definition~\ref{def: courant algebroid}, we find
    \begin{align*}
         &\left\langle \left[\iota_u(\eta\wedge \psi),\iota_v(\eta'\wedge\psi')\right],w\right\rangle\\
         &=-\left\langle [\psi,\psi'],v\right\rangle\langle\eta,u\rangle \langle\eta',w\rangle+\left\langle[\psi',\psi],u\right\rangle\langle\eta,w\rangle\langle\eta',v\rangle+\left\langle[\psi,\psi'],w\right\rangle\langle\eta,u\rangle\langle\eta',v\rangle\\
         &\quad+\left\langle[\psi,\eta'],v\right\rangle\langle\psi',w\rangle\langle\eta,u\rangle-\left\langle [\eta',\psi],u\right\rangle\langle\eta,w\rangle\langle\psi',v\rangle-\left\langle[\psi,\eta'],w\right\rangle\langle\eta,u\rangle\langle\psi',v\rangle\\
         &\quad +\left\langle[\eta,\psi'],v\right\rangle\langle\eta',w\rangle\langle\psi,u\rangle-\left\langle [\psi',\eta],u\right\rangle \langle\psi,w\rangle\langle \eta',v\rangle-\left\langle[\eta,\psi'],w\right\rangle\langle\psi,u\rangle\langle\eta',v\rangle\\
         &\quad-\left\langle[\eta,\eta'],v\right\rangle\langle\psi',w\rangle\langle\psi,u\rangle+\left\langle[\eta',\eta],u\right\rangle\langle\psi,w\rangle\langle\psi',v\rangle+\left\langle [\eta,\eta'],w\right\rangle\langle\psi,u\rangle\langle\psi',v\rangle\\
         &\quad +\left\langle \eta',[\psi,v]\right\rangle \langle\psi',w\rangle\langle \eta,u\rangle-\left\langle \psi',[\psi,v]\right\rangle\langle\eta,u\rangle\langle\eta',w\rangle\\
         &\quad-\left\langle \eta',[\eta,v]\right\rangle\langle\psi',w\rangle\langle\psi,u\rangle+\left\langle\psi',[\eta,v]\right\rangle\langle\eta',w\rangle\langle\psi,u\rangle\\
         &\quad -\left\langle\eta,[\psi',u]\right\rangle \langle\psi,w\rangle \langle\eta',v\rangle+\left\langle \psi,[\psi',u]\right\rangle\langle\eta,w\rangle\langle\eta',v\rangle\\
        &\quad+\left\langle\eta,[\eta',u]\right\rangle\langle\psi,w\rangle\langle\psi',v\rangle-\left\langle\psi,[\eta',u]\right\rangle\langle\eta,w\rangle\langle\psi',v\rangle.
    \end{align*}
    Therefore,
    \begin{align*}
    &\left\langle \left[\iota_u(\eta\wedge \psi),\iota_v(\eta'\wedge\psi')\right],w\right\rangle-\left\langle \left[\iota_v(\eta\wedge \psi),\iota_u(\eta'\wedge\psi')\right],w\right\rangle\\
    &=\left([\psi,\psi']\wedge\eta\wedge\eta'-[\psi,\eta']\wedge\eta\wedge\psi'-[\eta,\psi']\wedge\psi\wedge\eta'+[\eta,\eta']\wedge\psi\wedge\psi'\right)(u,v,w)\\
    &\quad +\langle\psi',w\rangle\left(\langle\eta',[\psi,v]\rangle\langle\eta,u\rangle-\langle\eta',[\psi,u]\rangle\langle\eta,v\rangle-\langle\eta',[\eta,v]\rangle\langle\psi,u\rangle+\langle\eta',[\eta,u]\rangle\langle\psi,v\rangle\right)\\
    &\quad -\langle\eta',w\rangle\left(\langle\psi',[\psi,v]\rangle\langle\eta,u\rangle-\langle\psi',[\psi,u]\rangle\langle\eta,v\rangle-\langle\psi',[\eta,v]\rangle\langle\psi,u\rangle+\langle\psi',[\eta,u]\rangle\langle\psi,v\rangle\right)\\
    &\quad +\langle\psi,w\rangle\left(\langle\eta,[\psi',v]\rangle\langle\eta',u\rangle-\langle\eta,[\psi',u]\rangle\langle\eta',v\rangle-\langle\eta,[\eta',v]\rangle\langle\psi',u\rangle+\langle\eta,[\eta',u]\rangle\langle\psi',v\rangle\right)\\
    &\quad -\langle\eta,w\rangle\left(\langle\psi,[\psi',v]\rangle\langle\eta',u\rangle-\langle\psi,[\psi',u]\rangle\langle\eta',v\rangle-\langle\psi,[\eta',v]\rangle\langle\psi',u\rangle+\langle\psi,[\eta',u]\rangle\langle\psi',v\rangle\right).
    \end{align*}

    On the other hand, for any $\epsilon\in\Gamma(M,L^*)$, we have 
    \begin{align*}
        &\iota_{\pr_L\left( [\iota_u(\eta\wedge\psi),v]-[v,\iota_u(\eta\wedge\psi)]\right)}\epsilon\\
        &=\left\langle\epsilon,\left[\langle\eta,u\rangle\psi,v\right]-\left[\langle\psi,u\rangle\eta,v\right]\right\rangle -\left\langle\epsilon,\left[v,\langle\eta,u\rangle\psi\right]-\left[v,\langle\psi,u\rangle\eta\right]\right\rangle \\
        &=\pi(\epsilon)\left(\langle\eta,u\rangle\langle\psi,v\rangle-\langle\psi,u\rangle\langle\eta,v\rangle\right)-2\left\langle\epsilon,\left[v,\langle\eta,u\rangle\psi\right]-\left[v,\langle\psi,u\rangle\eta\right]\right\rangle\\
        &=\pi(\epsilon)\left(\langle\eta,u\rangle\langle\psi,v\rangle-\langle\psi,u\rangle\langle\eta,v\rangle\right)-2\langle\eta,u\rangle\left\langle\epsilon,\left[v,\psi\right]\right\rangle+2\langle\psi,u\rangle\left\langle\epsilon,\left[v,\eta\right]\right\rangle\\
        &=\pi(\epsilon)\left(\langle\eta,u\rangle\langle\psi,v\rangle-\langle\psi,u\rangle\langle\eta,v\rangle\right)-2\langle\eta,u\rangle\,\pi(\epsilon)\left(\langle v,\psi\rangle\right)+2\langle\psi,u\rangle\,\pi(\epsilon)\left(\langle v,\eta\rangle\right)\\&\quad+2\langle\eta,u\rangle\left\langle\epsilon,\left[\psi,v\right]\right\rangle-2\langle\psi,u\rangle\left\langle\epsilon,\left[\eta,v\right]\right\rangle.
    \end{align*}
    It follows that
    \begin{align*}
        &\iota_{\pr_L \left( [\iota_u(\eta\wedge\psi),v]-[v,\iota_u(\eta\wedge\psi)]-[\iota_v(\eta\wedge\psi),u]+[u,\iota_v(\eta\wedge\psi)]\right)}\epsilon\\
        &=2\langle\eta,u\rangle\left\langle\epsilon,\left[\psi,v\right]\right\rangle-2\langle\psi,u\rangle\left\langle\epsilon,\left[\eta,v\right]\right\rangle-2\langle\eta,v\rangle\left\langle\epsilon,\left[\psi,u\right]\right\rangle+2\langle\psi,v\rangle\left\langle\epsilon,\left[\eta,u\right]\right\rangle
    \end{align*}
    and hence
    \begin{align*}
     &\tfrac{1}{2}\left\langle\iota_{\pr_L \left( [\iota_u(\eta\wedge\psi),v]-[v,\iota_u(\eta\wedge\psi)]-[\iota_v(\eta\wedge\psi),u]+[u,\iota_v(\eta\wedge\psi)]\right)}(\eta'\wedge \psi') ,w\right\rangle\\
     =&\left(\langle\eta',[\psi,v]\rangle\langle\eta,u\rangle-\langle\eta',[\psi,u]\rangle\langle\eta,v\rangle-\langle\eta',[\eta,v]\rangle\langle\psi,u\rangle+\langle\eta',[\eta,u]\rangle\langle\psi,v\rangle\right)\langle\psi',w\rangle\\
    &\quad -\left(\langle\psi',[\psi,v]\rangle\langle\eta,u\rangle-\langle\psi',[\psi,u]\rangle\langle\eta,v\rangle-\langle\psi',[\eta,v]\rangle\langle\psi,u\rangle+\langle\psi',[\eta,u]\rangle\langle\psi,v\rangle\right)\langle\eta',w\rangle.
    \end{align*}
    Putting everything together, we obtain
    \begin{align*}
        &\langle [\iota_u\varphi,\iota_v\varphi']-[\iota_v\varphi,\iota_u\varphi'],w\rangle\\
        &+\tfrac{1}{2}\langle \iota_{\pr_L\left([v,\iota_u\varphi]-[\iota_u\varphi,v]-[u,\iota_v\varphi]+ [\iota_v\varphi,u]\right)}\varphi',w\rangle\\
        &+\tfrac{1}{2}\langle \iota_{\pr_L\left( [v,\iota_u\varphi']-[\iota_u\varphi',v]-[u,\iota_v\varphi']+ [\iota_v\varphi',u]\right)}\varphi,w\rangle\\
        &=\left([\psi,\psi']\wedge\eta\wedge\eta'-[\psi,\eta']\wedge\eta\wedge\psi'-[\eta,\psi']\wedge\psi\wedge\eta'+[\eta,\eta']\wedge\psi\wedge\psi'\right)(u,v,w)\\
        &=[\varphi,\varphi'](u,v,w).
    \end{align*}

    \section{Elements of Hamilton--Nash--Moser theory}\label{appendix HNM}
    The main analytic tool for our proof of the local completeness
of the deformation family in Section \ref{defothm:sec} is Hamilton--Nash--Moser theory. Most of the relevant notions and results required are reviewed in the authors' previous work~\cite{paperoncomplexdefs} in the preliminary Section 3. In the following, we will summarize the notions and results used in the present paper, which are not included in~\cite{paperoncomplexdefs}. For a complete introduction to the topic, we refer to~\cite{hamilton82}.

In the following, we will denote by $\|\cdot\|_{k,2}$ the $W^{k,2}$-Sobolev norm on the space of sections to some vector bundle over a compact manifold (with respect to some choice of bundle metric, connection and Riemannian metric). Moreover, we will denote by $\|\cdot\|_{k,\infty}$ the $C^k$-norm on the space of sections to some vector bundle.

In analogy with the bracket on smooth $(0,p)$-forms with values in $T^{1,0}M$ on a complex manifold (cf.~\cite[Ex 3.8]{paperoncomplexdefs}) we find that the Schouten bracket on a Lie algebroid $L$ over a compact manifold is tame linear in both arguments separately:
For $\varphi\in\Gamma(M,\Lambda^k L)$ and $\psi\in \Gamma(M,\Lambda^l L)$, we have $\Vert [\varphi,\psi]\Vert_{n,2}\leq C\Vert \varphi\Vert_{n+1,2}\Vert\psi\Vert_{n+1,2}$ for $n\geq \dim M/2+1$. Similarly, we have $\Vert [\alpha,\beta]\Vert_{n,\infty}\leq C\Vert \alpha\Vert_{n+1,\infty}\Vert\beta\Vert_{n+1,\infty}$ for all $n\in \N_0$. 

Moreover, generalizing the special case discussed in~\cite[Ex 3.8]{paperoncomplexdefs}, we note that the Green's operator $\G$ of an elliptic linear partial differential operator of order $r$ on a vector bundle $V$ over a compact manifold satisfies $\|\G v\|_{k+r,2}\leq C\|v\|_{k,2}$ and is therefore tame linear \cite[Thm 7.4 in appendix]{kodairabook}. In particular, the operator $\Q_L=\dd_L^*\G_{L}$ defined in Section~\ref{sect: LA of GC} is tame linear as it satisfies $\Vert \Q_{L} \xi\Vert_{n,2}\leq C\Vert \xi\Vert_{n-1,2}$ for every $\xi\in\Gamma(M,\Lambda^k L^*)$.

The following is a generalization of~\cite[Prop 3.22]{paperoncomplexdefs}:
\begin{proposition}
\label{prop: inverse family from elliptic}
    Let $\mathcal{U}\subset \mathcal{F}$ be an open neighborhood of $0$ in a tame Fréchet space $\mathcal{F}$. Let $M$ be a compact   manifold,  
    $\dd_i : \Gamma(M,E_i) \to \Gamma(M,E_{i+1})$ an elliptic complex such that $\dd_i$ is of order one for every $i$ and for a fixed $j\in \mathbb{Z}$ let 
    $$A\colon \mathcal{U}\times \Gamma(M,E_j)\longrightarrow \Gamma(M,E_j)$$
    be a smooth tame family of linear maps such that
    $$A(\varphi)\alpha=\alpha+\delta(\varphi)\alpha,$$
    where $\delta(\varphi)$ is such that $\alpha\mapsto\Delta_j \delta(\varphi)\alpha$ is a linear partial differential operator of order $\leq2$, $\delta(\varphi)\alpha\in \im \dd_j^*$ for every $(\varphi,\alpha)\in \mathcal{U}\times \Gamma(M,E_j)$ and such that $\delta(0)=0$. Then there is an open neighborhood $\mathcal{U}'\subset \mathcal{U}$ of $0$ such that the restriction
    $$A(\varphi)\colon \left(\im \dd_j^*\cap \Gamma(M,E_j)\right)\longrightarrow \left(\im \dd_j^*\cap \Gamma(M,E_j)\right)$$
    is invertible for every $\varphi \in \mathcal{U}'$ and the family of inverses is smooth tame. It follows that $A(\varphi)$ is invertible for every $\varphi\in \mathcal{U}'$ and the family of inverses is smooth tame.
\end{proposition}
\begin{proof} The proof works analogously to the proof of~\cite[Prop 3.22]{paperoncomplexdefs}.
\end{proof}

In the present paper, we make use of the definition of tame submanifolds and tame Lie subgroups. These are stated below.
\begin{definition}
    A \emph{tame (Fréchet) submanifold} $\mathcal{N}$ of a tame Fréchet manifold $\mathcal{M}$ is a subset that can be covered by (smooth tame) charts $\{(\mathcal{U}_i,\varphi_i\colon \mathcal{U}_i\longrightarrow E_i\oplus F_i)\}$ of $\mathcal{M}$, where $E_i\oplus F_i$ is a direct sum of graded Fréchet spaces such that $\varphi_i(\mathcal{N}\cap \mathcal{U}_i)=\varphi_i(\mathcal{U}_i)\cap (E_i\times \{0\})$.
\end{definition}

\begin{definition}
   A \emph{tame (Fréchet) Lie subgroup} of a tame Lie group $\mathcal{K}$ is a tame Fréchet submanifold, which is a subgroup of $\mathcal{K}$.
\end{definition}

The push-forward of a vector field by a diffeomorphism is a smooth tame map (see e.g.~\cite[Sec 3.3.1]{paperoncomplexdefs}).
In the following we show that the pull-back of a differential form by a diffeomorphism is a smooth tame map.
\begin{proposition}\label{prop: pullback}
    The pull-back of a $k$-form by a diffeomorphism seen as a map
    \begin{align*}
        \Diff(M)\times \Omega^k(M)&\longrightarrow \Omega^k(M)\\
        (f,\omega)&\longmapsto f^*\omega
    \end{align*}
    is smooth tame.
\end{proposition}
\begin{proof}
    First, we observe that for $k=0$, the pull-back of $g\in \Omega^0(M)=C^{\infty}(M)$ simply becomes the composition map $f^*g=g\circ f$ for $f\in \Diff(M)$, which is a smooth tame map (see \cite[Thm II.2.3.3]{hamilton82}).

    Next, given an exact one-form $\alpha\in \Omega^1(M)$, we can write it as $\alpha=\dd g$ for some $g\in C^{\infty}(M)$. The function $g$ is uniquely determined up to a constant and there is a tame linear map $\alpha\mapsto g$ such that $\dd g=\alpha$, given by $\Q_{\mathrm{dR}}=\dd^*\G_{\mathrm{dR}}$, where $\dd^*$ is the formal adjoint of the de Rham differential with respect to the $L^2$-inner product induced by the auxiliary Riemannian metric on $M$ and $\G_{\mathrm{dR}}$ is the Green's operator of the Laplacian $\Delta=\dd^*\dd+\dd\dd^*$. It follows from the discussion above that $\QdR$ is tame linear. The pull-back of $\alpha$ by $f\in\Diff(M)$ is then given by
    \[f^*\alpha=f^*\dd g=\dd (g\circ f)\]
    and hence smooth tame as a composition of smooth tame maps.

    Moreover, for some $g\in C^{\infty}(M)$ and exact one-forms $\alpha_1,\dots,\alpha_k$, the pull-back of $\omega:=g\,\alpha_1\wedge\dots\wedge \alpha_k$ by $f\in \Diff(M)$ is given by
    \[f^*\omega=(g\circ f)\, f^*\alpha_1\wedge\dots\wedge f^*\alpha_k.\]
    Since the wedge product is a tame multi-linear map, this shows that
    \begin{equation}\label{eq: pullback smooth tame}
    \begin{aligned}
             \Diff(M)\times C^{\infty}(M)\times \left(\dd \Omega^0(M)\right)^k&\longrightarrow \Omega^k(M)\\
        (f,g,\alpha_1,\dots,\alpha_k)&\longmapsto f^*(g\,\alpha_1\wedge\dots\wedge\alpha_k)   
    \end{aligned}
    \end{equation}
    is a smooth tame map.

    Locally, every $k$-form can be written as a finite sum of elements of the form $g\,\alpha_1\wedge\dots\wedge\alpha_k$, where $g\in C^{\infty}(M)$ and $\alpha_i\in\dd\Omega^0(M)$. These local descriptions can be patched together using a partition of unity and since $M$ is compact, in this way we can write every $k$-form as a finite sum of elements of the form $g\,\alpha_1\wedge\dots\wedge\alpha_k$ as above. In particular, this shows that the pull-back of any $k$-form by some diffeomorphism is a finite sum of terms of the form as in~\eqref{eq: pullback smooth tame} which are smooth tame. It follows that the pull-back is smooth and that all its derivatives satisfy tame estimates. This finishes the proof.
\end{proof}

\printbibliography[title=References]
\end{document}